\documentclass[10pt,a4paper]{article}
\usepackage[a4paper,top=3cm,bottom=3cm,left=2.5cm,right=2.5cm,%
bindingoffset=5mm]{geometry}
\pdfoutput=1
\usepackage[utf8]{inputenc}
\usepackage[english]{babel}
\usepackage{amsmath}
\usepackage{amsfonts}
\usepackage{amssymb}
\usepackage{graphicx}
\usepackage{authblk}
\usepackage{amsthm}
\usepackage{epstopdf}
\usepackage{tabularx}
\usepackage{multirow}
\usepackage{booktabs}
\usepackage{caption}
\usepackage{subcaption}
\usepackage[percent]{overpic}
\newcommand{\numberset}{\mathbb}

\newtheorem{theorem}{Theorem}[section]

\newtheorem{lemma}[theorem]{Lemma}
\newtheorem{remark}[theorem]{Remark}
\newtheorem{proposition}{Proposition}[section]

\newcommand{\cfan}[1]{\textnormal{\textbf{\detokenize{#1}}}}

\newcommand{\N}{\numberset{N}}
\newcommand{\R}{\numberset{R}}

\newcommand{\Pk}{\numberset{P}}
\newcommand{\Pkc}{\numberset{P}^{\rm cont}}

\newcommand{\diver}{{\rm div}}
\newcommand{\bdiver}{\boldsymbol{\diver}}
\newcommand{\bnabla}{\boldsymbol{\nabla}}

\newcommand{\bcurl}{\boldsymbol{{\rm curl}}}
\newcommand{\beps}{\boldsymbol{\epsilon}}

\newcommand{\jump}[1]{\lbrack\!\lbrack\,#1\,\rbrack\!\rbrack}
\newcommand{\media}[1]{\left\{\!\left\{\,#1\,\right\}\!\right\}}
\newcommand{\jumpmedia}[1]{\left(\!\left(\,#1\,\right)\!\right)}

\newcommand{\nn}{\boldsymbol{n}}

\newcommand{\Edges}{\Sigma_h}
\newcommand{\EdgesB}{\Sigma_h^{\partial}}
\newcommand{\EdgesI}{\Sigma_h^{\rm {int}}}
\newcommand{\EdgesE}{\Sigma_h^E}

\renewcommand{\epsilon}{\varepsilon}
\renewcommand{\theta}{\vartheta}
\renewcommand{\rho}{\varrho}
\renewcommand{\phi}{\varphi}
\newcommand{\zz}{\boldsymbol{\zeta}}

\newcommand{\BB}{\boldsymbol{B}}
\newcommand{\HH}{\boldsymbol{H}}
\newcommand{\GG}{\boldsymbol{G}}
\newcommand{\TT}{\boldsymbol{\Theta}}

\newcommand{\uu}{\boldsymbol{u}}
\newcommand{\vv}{\boldsymbol{v}}
\newcommand{\ww}{\boldsymbol{w}}
\newcommand{\cc}{\boldsymbol{\chi}}
\newcommand{\ff}{\boldsymbol{f}}

\newcommand{\qq}{\boldsymbol{q}}
\newcommand{\pp}{\boldsymbol{p}}
\newcommand{\xx}{\boldsymbol{x}}

\newcommand{\ns}{\nu_{\rm S}}
\newcommand{\nm}{\nu_{\rm M}}
\newcommand{\bdma}{\mu_a}
\newcommand{\bdmc}{\mu_c}
\newcommand{\bdmJU}{\mu_{J_1}}
\newcommand{\bdmJD}{\mu_{J_2}}
\newcommand{\bdmK}{\mu_{K}}

\newcommand{\bdmY}{\mu_{Y}}

\newcommand{\Huno}{\boldsymbol{H}^1(\Omega)}
\newcommand{\LL}{\boldsymbol{L}}
\newcommand{\WW}{\boldsymbol{W}}
\newcommand{\Ldue}{\boldsymbol{L}^2(\Omega)}
\newcommand{\Hunoenne}{\boldsymbol{H}^1_{\nn}(\Omega)}
\newcommand{\WWc}{\boldsymbol{W}}
\newcommand{\WWs}{\boldsymbol{W_*}}
\newcommand{\WWd}{\boldsymbol{W}^h_k}
\newcommand{\Hunozero}{\boldsymbol{H}^1_0(\Omega)}
\newcommand{\VVc}{\boldsymbol{V}}
\newcommand{\VVs}{\boldsymbol{V_*}}
\newcommand{\VVd}{\boldsymbol{V}^h_k}
\newcommand{\ZZc}{\boldsymbol{Z}}
\newcommand{\ZZs}{\boldsymbol{Z_*}}
\newcommand{\ZZd}{\boldsymbol{Z}^h_k}
\newcommand{\Qc}{Q}
\newcommand{\Qd}{Q^h_{k}}
\newcommand{\Rd}{R^h_{k}}

\newcommand{\am}{a^{\rm M}}
\newcommand{\as}{a^{\rm S}}

\newcommand{\ash}{a^{\rm S}_h}

\newcommand{\astabt}{\mathcal{A}_{3\rm{f}}}
\newcommand{\astabf}{\mathcal{A}_{4\rm{f}}}
\newcommand{\ccoe}{\alpha_*}
\newcommand{\clemma}{c_{\rm L}}

\newcommand{\data}{\gamma_{\rm data}}

\newcommand{\gunoinf}{\Gamma}
\newcommand{\gunohinf}{\Gamma_{h}}
\newcommand{\gzhinf}{\gamma_{h}}
\newcommand{\intcip}{\mathcal{I}_{\mathcal{O}}}
\newcommand{\Pzerok}[1]{\Pi_{#1}}
\newcommand{\PWWd}{\mathcal{I}_{\WWc}}
\newcommand{\PVVd}{\mathcal{I}_{\VVc}}
\newcommand{\Pt}{\mathcal{I}_{\rm L}}
\newcommand{\Pth}{\mathcal{I}_{{\rm L}, h}}
\newcommand{\omegazh}{\Omega_h^{\boldsymbol{\zeta}}}
\newcommand{\omegaEh}{\Omega_h^{M}}
\newcommand{\omegaz}{\omega_{\boldsymbol{\zeta}}}

\newcommand{\normaunoh}[1]{\left\|#1\right\|_{1,h}}
\newcommand{\normamt}[1]{\left\|#1\right\|_{3\rm f}}
\newcommand{\normamf}[1]{\left\|#1\right\|_{4\rm f}}
\newcommand{\normamfc}[1]{\left\|#1\right\|_{4\rm f, \cc}}
\newcommand{\normaupwc}[1]{\left|#1\right|_{\rm upw, \cc}}
\newcommand{\normacipT}[1]{\left|#1\right|_{J_h, \TT}}
\newcommand{\normaupw}[1]{\left|#1\right|_{\rm upw, \uu_h}}
\newcommand{\normacip}[1]{\left|#1\right|_{J_h, \BB_h}}
\newcommand{\normaKc}[1]{\left|#1\right|_{K_h, \cc}}
\newcommand{\normaK}[1]{\left|#1\right|_{K_h, \uu_h}}
\newcommand{\normaY}[1]{\left|#1\right|_{Y_h}}
\newcommand{\normaV}[1]{\left\|#1\right\|_{(\cc, \TT)}}
\newcommand{\normastab}[1]{\left\|#1\right\|_{\rm stab}}

\newcommand{\normai}[1]{{\left\vert\kern-0.25ex\left\vert\kern-0.25ex\left\vert#1\right\vert\kern-0.25ex\right\vert\kern-0.25ex\right\vert}}

\newcommand{\uui}{\PVVd \uu}
\newcommand{\vvi}{\PVVd \vv}
\newcommand{\eei}{\boldsymbol{e}_{\mathcal{I}}}
\newcommand{\eeh}{\boldsymbol{e}_h}
\newcommand{\BBi}{\PWWd \BB}
\newcommand{\HHi}{\PWWd \HH}
\newcommand{\EEi}{\boldsymbol{E}_{\mathcal{I}}}
\newcommand{\EEh}{\boldsymbol{E}_h}

\newcommand{\lmt}{\Lambda_{3{\rm f}}}
\newcommand{\lmf}{\Lambda_{4{\rm f}}}
\newcommand{\ls}{\Lambda_{\rm stab}}

\usepackage{color}

\newcommand{\noStab}{\texttt{noStab}}
\newcommand{\threeF}{\texttt{stab3F}}
\newcommand{\fourF}{\texttt{stab4F}}

\title{Pressure and convection robust Finite Elements for Magnetohydrodynamics}

\author[1,2]{L. Beir\~ao da Veiga \thanks{lourenco.beirao@unimib.it}}

\author[1]{F. Dassi \thanks{franco.dassi@unimib.it}}

\author[3]{G. Vacca \thanks{giuseppe.vacca@uniba.it}}

\affil[1]{Dipartimento di Matematica e Applicazioni,
Universit\`a degli Studi di Milano Bicocca,
Via Roberto Cozzi 55 - 20125 Milano, Italy}

\affil[2]{IMATI-CNR, Via Ferrata 2, 20127 Pavia}

\affil[3]{Dipartimento di Matematica, 
Universit\`a degli Studi di Bari, 
Via Edoardo Orabona 4  - 70125 Bari, Italy}

\begin{document}
\maketitle

\abstract{
We propose and analyze two convection quasi-robust and pressure robust finite element methods for a fully nonlinear time-dependent magnetohydrodynamics problem. 
 Both methods employ  the $H_{\rm div}$ conforming BDM element coupled with an appropriate pressure space guaranteeing the exact diagram for the fluid part, and the $H^1$ conforming Lagrange element for the approximation of the magnetic fluxes, and  make use of suitable DG upwind terms and  CIP stabilizations to handle the fluid and magnetic convective terms.
The main difference between the two approaches here proposed (labeled  as three-field scheme and four field-scheme respectively) lies in the strategy adopted to enforce the divergence-free condition of the magnetic field.
The three-filed scheme implements a grad-div stabilization, whereas the four-field scheme introduces a suitable Lagrange multiplier and additional stabilization terms in the formulation. 
The developed error estimates for the two schemes are uniform in both diffusion parameters and optimal with respect to the diffusive norm. 
Furthermore, in the convection dominated regime, being $k$ the degree of the method and $h$ the mesh size, we are able to prove $O(h^k)$ 
and  $O(h^{k+1/2})$ pre-asymptotic error reduction rate for the three-field scheme and four-filed scheme respectively.
A set of numerical tests support our theoretical findings. 
}

\section{Introduction}\label{sec:intro}

The field of magnetohydrodynamics (MHD) has garnered increased attention in the realm of computational mathematics in recent years. These equations, found in the study of plasmas and liquid metals, have diverse applications in geophysics, astrophysics, and engineering.
The combination of equations from fluid dynamics and electromagnetism results in various models with different formulations and finite element choices. This diversity yields a wide array of methods, each with its own strengths and limitations 
(see for instance \cite{Gerbeau,Gerbeau2,Guermond,Schotzau,Greif,Prohl,Houston,Dong,Hiptmair,Wacker,Badia,Zhang,VEM,hu2021helicity}). 

This article is motivated by the unsteady MHD problem in three space dimensions, here conveniently scaled for ease of exposition. The governing equations involve the velocity field $\uu$, fluid pressure $p$, and magnetic induction $\BB$: 
\begin{equation}
\label{eq:non-linear primale}
\left\{
\begin{aligned}
&\begin{aligned}
\partial_t\uu -  \ns \, \bdiver (\beps (\uu) ) + (\bnabla \uu ) \,\uu    +  \BB \times \bcurl(\BB)  - \nabla p &= \ff \qquad  & &\text{in $\Omega \times I$,} 
\\
\partial_t\BB +  \nm \, \bcurl (\bcurl (\BB) ) - \bcurl(\uu \times \BB)  &= \GG \qquad  & &\text{in $\Omega \times I$,} 
\\
\diver \, \uu = 0 \, , \qquad \diver \, \BB &= 0 \qquad & &\text{in $\Omega \times I$,} 
\end{aligned}
\end{aligned}
\right.
\end{equation}
to be completed with suitable boundary and initial conditions, and where the parameters $\ns$ and $\nm$ represent fluid and magnetic diffusive coefficients, while external forces $\ff$ and $\GG$ account for volumetric effects.

In practical applications, such as aluminum electrolysis or space weather prediction, the scaled diffusion parameters $\ns$ and $\nm$ are often substantially small, see e.g. \cite{Berton, Armero}. It is well known that Finite element schemes in fluid dynamics may encounter instabilities when the convective term dominates the diffusive term, necessitating stabilization techniques. This challenge is addressed in the literature, including references such as \cite{volker:book}, \cite{garcia}, and specific papers like \cite{BH:1982,BFH:2006,OLHL:2009,MT:2015,BDV:NS-TD}. In the context of MHD equations, similar stability issues arise, particularly concerning the electromagnetic aspect which further complicates the problem. Without proper attention, even moderately small diffusion parameters can significantly impact the accuracy of the velocity solution. 

From the theoretical standpoint, a method is said to be quasi-robust if, assuming a sufficiently regular solution, it enjoys error estimates which do not explode for small values of the diffusion parameter, possibly in a norm which includes also some control of the convective term. Such estimates are typically expected to yield convergence rates that are optimal in the diffusive norm. An additional feature which further expresses the robustness of a scheme is the capability of (provably) exhibiting error reduction rates which gain an $O(h^{1/2})$ factor whenever $\nu \ll h$ (the latter representing the mesh size). See, for instance, the classification in \cite{garcia}.
On this authors knowledge, while there are a few articles dealing with the simpler linearized case,
no contribution in the literature develops quasi-robust error estimates (with respect to both diffusion parameters $\ns,\nm$) for the full nonlinear MHD system. 

The present article takes the steps from the results in \cite{MHD-linear}, where a quasi-robust numerical scheme was developed and analysed for the \emph{linearized} stationary version of \eqref{eq:non-linear primale}; we here tackle the more complex nonlinear case.
The method here investigated makes use of $H_{\rm div}$ conforming elements for the fluid part (coupled with a suitable pressure space guaranteeing the exact diagram), combined with an upwind stabilization and DG techniques to preserve the consistency of the discrete formulation. This idea is not novel, see for instance \cite{Hdiv1,Hdiv2,Han}, and has the additional advantage of yielding a pressure-robust method, a property which was recently recognized as critical in incompressible fluid flows, e.g. \cite{linke-merdon:2016,john-linke-merdon-neilan-rebholz:2017}.
We here assume that the domain is convex, which allows to use a globally continuous discrete magnetic field, namely a standard Lagrange finite element space. The proposed approach is combined with a specific stabilization for the magnetic equation, in the spirit of the continuous interior penalty approach (CIP, \cite{CIP,BFH:2006}) and as such taking the form of scaled jumps across element edges. Finally, also a grad-div stabilization is adopted to better take into account the divergence-free condition of the magnetic field, which, differently from the velocities, is not enforced exactly.

Denoting by $k$ the polynomial order of the scheme, we are able to show $O(h^k)$ velocity error estimates for regular solutions in a space-time norm that also includes stabilization terms. Such bounds are uniform with respect to small $\ns$ and $\nm$, thus expressing the quasi-robustness of the scheme, and independent from the pressure solution, thus expressing the pressure robustness.
On the other hand, for this three-field method we are unable to derive the additional $O(h^{1/2})$ factor for the  error reduction rate in convection dominated regimes. We clearly identify the cause of this ``limitation'' in the discrete ${\rm div} {\bf B} = 0$ condition, which is not imposed strongly enough.
We therefore propose also a second (four-field) scheme, in which (1) we impose such solenoidal condition through the introduction of a suitable Lagrange multiplier and (2) add additional stabilization terms to the formulation. For such four-field scheme we are indeed able to prove also the additional $O(h^{1/2})$ factor in the error reduction when $\nu \ll h$.

In the final part of the article we develop some numerical tests in three dimensions in order to evaluate the scheme from the practical perspective and make some comparison among the 3-field method, the 4-field method and a more basic 3-field method without any specific stabilization for the magnetic part of the equations.

The paper is organized as follows. We introduce the continuous problem in Section \ref{sec:cont} and some preliminary results in Section \ref{sec:notations}. Afterwards, the proposed numerical methods are described in Section \ref{sec:stab}. The converge estimates for the velocity and the magnetic field are developed in Sections \ref{sec:theo-3f} and \ref{sec:theo-4f} for the three and four-field schemes, respectively. In Section \ref{sec:pressione} we present briefly some error estimates for the pressure variable.
Finally, numerical results are shown in Section \ref{sec:num}.

%
%
%
%
%
%
%
%
%
%
%
%
%
%
%
%
%
%

\section{Continuous problem}\label{sec:cont}

We start this section with some standard notations. 
Let the computational domain $\Omega \subset \R^3$ be a convex polyhedron with regular boundary $\partial \Omega$ having outward pointing unit normal $\nn$.
%
The symbol $\nabla$    
denotes the gradient for scalar functions while  
$\bnabla$, $\beps$,  $\bcurl$ and $\diver$ 
denote   
the gradient, the symmetric gradient operator, the curl operator,
and  the divergence operator for vector valued functions  respectively. Finally, 
$\bdiver$ denotes the vector valued divergence operator for tensor fields. 

Throughout the paper, we will follow the usual notation for Sobolev spaces
and norms \cite{Adams:1975}.
Hence, for an open bounded domain $\omega$,
the norms in the spaces $W^r_p(\omega)$ and $L^p(\omega)$ are denoted by
$\|{\cdot}\|_{W^r_p(\omega)}$ and $\|{\cdot}\|_{L^p(\omega)}$, respectively.
Norm and seminorm in $H^{r}(\omega)$ are denoted respectively by
$\|{\cdot}\|_{r,\omega}$ and $|{\cdot}|_{r,\omega}$,
while $(\cdot,\cdot)_{\omega}$ and $\|\cdot\|_{\omega}$ denote the $L^2$-inner product and the $L^2$-norm (the subscript $\omega$ may be omitted when $\omega$ is the whole computational
domain $\Omega$).
For the functional spaces introduced above we use the bold symbols to denote the corresponding sets of vector valued functions.
We further introduce the following spaces
\[
\begin{aligned}
\Hunoenne &:= \{\vv \in \Huno \,\,\, \text{s.t.} \,\,\, \vv \cdot \nn = 0 \,\,\, \text{on $\partial \Omega$} \} \,, 
\\
\HH_0(\diver, \Omega) &:= \{\vv \in \boldsymbol{L}^2(\Omega) \,\,\,\text{s.t.} \,\,\, \diver \,\vv \in L^2(\Omega) \,\,\, \text{and} \,\,\, \vv \cdot \nn = 0 \,\,\, \text{on $\partial \Omega$} \} \,,
\\
\HH_0(\diver^0, \Omega) &:= \{\vv \in \HH_0(\diver, \Omega) \,\,\,\text{s.t.} \,\,\, \diver \,\vv = 0 \} 
\,.
\end{aligned}
\]
For a Banach space $V$ we denote with $V'$ the dual space of $V$.
Let $(T_0, T_F) \subset \R$ denote the time interval of interest. For a space-time function $v(\xx, t)$ defined on $\omega \times (T_0, T_F)$, we denote with $\partial_t v$ the derivative with respect to the time variable. Furthermore, using standard notations \cite{quarteroni-valli:book}, for a Banach space $V$ with norm $\|\cdot\|_V$, we introduce the Bochner spaces $W^s_q(T_0, T_F; V)$ and $H^s(T_0, T_F; V)$  endowed with norms $\|{\cdot}\|_{W^s_q(T_0, T_F; V)}$ and $\|{\cdot}\|_{H^s(T_0, T_F; V)}$ respectively.

\smallskip
Let now $\Omega \subseteq \R^3$ be the polyhedral convex domain, let $T > 0$ be the final time and set $I:= (0, T)$. 
We consider the unsteady MagnetoHydroDynamic (MHD) equation 
(see for instance \cite[Chapter 2]{Gerbeau}):
\begin{equation}
\label{eq:primale}
\left\{
\begin{aligned}
&\begin{aligned}
\partial_t \uu -  \ns \, \bdiver (\beps (\uu) ) + (\bnabla \uu ) \,\uu    +  \BB \times \bcurl(\BB)  - \nabla p &= \ff \qquad  & &\text{in $\Omega \times I$,} 
\\
\diver \, \uu &= 0 \qquad & &\text{in $\Omega \times I$,} 
\\
\partial_t \BB +  \nm \, \bcurl (\bcurl (\BB) ) - \bcurl(\uu \times \BB)  &= \GG \qquad  & &\text{in $\Omega \times I$,} 
\\
\diver \, \BB &= 0 \qquad & &\text{in $\Omega \times I$,} 
\end{aligned}
\end{aligned}
\right.
\end{equation}
coupled with the homogeneous boundary conditions
\begin{equation}
\label{eq:bc cond}
\uu = 0 \,, \quad 
\BB \cdot \nn = 0 \,,  \quad 
\bcurl(\BB) \times \nn = 0  \quad \text{on $\partial \Omega$,}
\end{equation}
and the initial conditions
\begin{equation}
\label{eq:ini cond}
\uu(\cdot, 0) = \uu_0 \,, \quad 
\BB(\cdot, 0) = \BB_0  \quad \text{in $\Omega$.}
\end{equation}
We assume that the external loads $\ff \in L^2(0, T; {\LL^2(\Omega)})$ and 
$\GG \in L^2(0, T; \boldsymbol{L}^2(\Omega))$, and initial data
$\uu_0$, $\BB_0 \in \HH_0(\diver^0, \Omega)$.
The parameters $\ns$, $\nm \in \R^+$ in \eqref{eq:primale} represent the viscosity of the fluid and the inverse of the magnetic permeability of the medium, respectively.

Notice that the third and fourth equations in \eqref{eq:primale}, 
the boundary conditions \eqref{eq:bc cond} and the initial condition \eqref{eq:ini cond},  yield the compatibility condition  
$(\GG(\cdot, t) \,, \nabla \phi)  = 0$  for all $\phi \in H^1(\Omega)$, a.e. in $I$.

We now derive the variational formulation for Problem \eqref{eq:primale}.
Consider the following spaces
\begin{equation}
\label{eq:spazi_c}
\VVc := \Hunozero \,, 
\quad
\WWc := \Hunoenne \,,
\quad
\Qc  := L^2_0(\Omega) = \left\{ q \in L^2(\Omega) \quad \text{s.t.} \quad (q, \,1) = 0 \right\}  \,,
\end{equation}
representing the velocity field space, the magnetic induction space and the pressure space, respectively, endowed with the standard norms, and the forms
\begin{equation}
\label{eq:forme_c1}
\as(\uu,  \vv) :=  (\beps(\uu), \, \beps (\vv) ) \,,
\qquad
c(\cc;  \uu, \vv) :=  \left( ( \bnabla \uu ) \, \cc ,\, \vv  \right) \,,
\end{equation}
and 
\begin{equation}
\label{eq:forme_c2}
\begin{aligned}
\am(\BB,  \HH) &:=   \left( \bcurl(\BB) ,\, \bcurl(\HH)\right) +
\left(\diver(\BB) \,, \diver(\HH) \right),
\\
b(\vv, q) &:=  (\diver \vv, \, q) \,,
\\
d(\TT; \HH, \vv) &:=  ( \bcurl (\HH ) \times \TT ,\, \vv) \,.
\end{aligned}
\end{equation}
%
Let us introduce the kernel of the  bilinear form $b(\cdot,\cdot)$
that corresponds to the functions in $\VVc$ with vanishing divergence
\begin{equation}
\label{eq:Z}
\ZZc := \{ \vv \in \VVc \quad \text{s.t.} \quad \diver \, \vv = 0  \}\,.
\end{equation}
We consider the following variational problem  \cite[Definition 2.18]{Gerbeau}: find
\begin{itemize}
\item $\uu \in L^\infty(0, T; \HH_0(\diver, \Omega)) \cap L^2(0, T; \VVc)$,
\item $p \in L^2(0, T; \Qc)$,
\item $\BB \in L^\infty(0, T; \HH_0(\diver, \Omega)) \cap L^2(0, T; \WWc)$,
\end{itemize}
such that for a.e. $t \in I$
\begin{equation}
\label{eq:variazionale}
\left\{
\begin{aligned}
\left(\partial_t \uu, \vv \right) + \ns \as(\uu, \vv) + c(\uu; \uu, \vv)  
-d(\BB; \BB, \vv)  + b(\vv, p) 
&= 
(\ff, \vv)  
&\,\, & \text{$\forall \vv \in \VVc$,} 
\\
b(\uu, q) &= 0 
&\,\, & \text{$\forall q \in \Qc$,}
\\
\left(\partial_t \BB, \HH \right) + \nm \am(\BB, \HH) + d(\BB; \HH, \uu) 
&= 
(\GG, \HH)  
&\,\, & \text{$\forall \HH \in \WWc$,}
\end{aligned}
\right.
\end{equation}
coupled with initial conditions \eqref{eq:ini cond}.
Note that the condition $\diver\BB=0$ is implied by the last equation, see for instance the analogous proof for the linear case in \cite{MHD-linear}.

\begin{proposition}
Assume that the domain $\Omega$ is a convex polyhedron. Then
Problem \eqref{eq:variazionale} admits solution. 
Additionally,  Problem \eqref{eq:variazionale} is a variational formulation of Problem \eqref{eq:primale}.
Moreover a solution $(\uu, p, \BB)$ satisfies the stability bound
\begin{equation}
\label{eq:cont.stab}
\begin{aligned}
\Vert \uu(\cdot, T)\Vert^2+ 
\Vert \BB(\cdot, T)\Vert^2 +
\ns \Vert \uu \Vert_{L^2(0,T; \VVc)}^2 +
\nm \Vert \BB \Vert_{L^2(0,T; \WWc)}^2 +
\Vert p \Vert_{L^2(0,T; \Qc)}^2 \lesssim
\data^2 \,,
\end{aligned}
\end{equation}
where
\begin{equation}
\label{eq:data}
\data^2 :=
\Vert \ff \Vert_{L^1(0,T; \Ldue)}^2 +
\Vert \GG \Vert_{L^1(0,T; \Ldue)}^2
+
\Vert \uu_0\Vert_{\Ldue}^2 + 
\Vert \BB_0\Vert_{\Ldue}^2  \,.
\end{equation}
\end{proposition}

\begin{proof}
The proof follows combining the arguments in Proposition 2.19, Remark 2.2.1, Proposition 3.18 and Lemma 3.19 in \cite{Gerbeau}. A key ingredient of the proof is the following embedding valid on the convex polyhedron $\Omega$  (see \cite[Theorem 3.9]{Girault-book} and \cite{Amrouche}): 
there exists a positive constant $c_\Omega$ depending only on the domain $\Omega$ s.t.
\begin{equation}
\label{eq:well3}
\Vert \bcurl(\HH)\Vert^2 +  \Vert \diver \HH\Vert^2 
\geq c_\Omega \Vert \HH \Vert_{\WWc}^2 \qquad \text{for all $\HH \in \WWc$.}
\end{equation}
\end{proof}

\section{Notations and preliminary theoretical results}
\label{sec:notations}

In this section we fix some notations and we introduce some preliminary 
theoretical results that will be instrumental in the forthcoming sections.


Let $\{\Omega_h\}_h$ be a family of  conforming decompositions of $\Omega$ into tetrahedral elements $E$ of diameter $h_E$. We denote by 
$h := \sup_{E \in \Omega_h} h_{E}$ the mesh size associated with $\Omega_h$.
Let $\mathcal{N}_h$ be the set of internal vertices of the mesh $\Omega_h$, and for any $\zz \in \mathcal{N}_h$ we set
\[
\omegazh:= \{E \in \Omega_h \, \, \, \text{s.t.} \,\, \, \zz \in E\}\,,
\quad
\omegaz := \cup_{E \in \omegazh} E \,,
\quad
h_{\zz} := \text{diameter of $\omegaz$}
 \,.
\]

\noindent
For any $E \in \Omega_h$ let $\mathcal{N}_h^E$ be the set of the vertices of $E$. 
We denote by $\Edges$ the set of faces of $\Omega_h$ divided into 
internal $\EdgesI$ and external $\EdgesB$ faces; 
for any $E \in \Omega_h$ we denote by $\EdgesE$ the set of the faces of $E$. Furthermore for any $f \in \Edges$ we denote with $h_f$ the diameter of $f$ and
\begin{equation}
\label{eq:omegaf}
\Omega_h^f:= \{E \in \Omega_h \, \, \, \text{s.t.} \,\, \, f \subset \partial E\}\,,
\quad
\omega_f := \cup_{E \in \Omega_h^f} E \,,
\quad
h_{\omega_f} := \text{diameter of $\omega_f$} \,.
\end{equation}

We make the following mesh assumptions. Note that the second condition \textbf{(MA2)}
is required \emph{only} for the analysis of the lowest order case (that is order $1$).

\smallskip
\noindent
\textbf{(MA1) Shape regularity assumption:}

\noindent
The mesh family $\left\{ \Omega_h \right\}_h$ is shape regular: it exists a  positive  constant $c_{\rm{M}}$ such that each element $E \in \{ \Omega_h \}_h$ is star shaped with respect to a ball of radius $\rho_E$ with  $h_E \leq c_{\rm{M}} \rho_E$. 

\smallskip
\noindent
\textbf{(MA2) Mesh agglomeration with stars macroelements:}

\noindent 
There exists a family of conforming meshes $\{ \widetilde{\Omega}_h \}_h$ of $\Omega$ with the following properties:
(i) it exists a positive constant $\widetilde{c}_{\rm{M}}$ such that each element $M \in \widetilde{\Omega}_h$ is a finite  (connected) agglomeration of elements in $\Omega_h$, i.e., it exists $\omegaEh \subset \Omega_h$ with ${\rm{card}}(\omegaEh) \leq \widetilde{c}_{\rm{M}}$ and $M = \cup_{E \in \omegaEh} E$;
(ii) for any $M \in \widetilde{\Omega}_h$ it exists $\zz \in \mathcal{N}_h$ such that $\omegaz \subseteq M$.

\begin{remark}
\label{rm:mesh}
 Assumption \cfan{(MA1)} is classical in FEM. Assumption \cfan{(MA2)} is needed only for $k=1$ and has a purely theoretical purpose (see Lemma \ref{lm:cip} and Lemma \ref{lm:int-i}).  However, it is easy to see that \cfan{(MA2)} is not restrictive, see Remark 3.1 in \cite{MHD-linear}.
\end{remark}

The mesh assumption \textbf{(MA1)} easily implies  the following property.

\noindent
\textbf{(MP1) local quasi-uniformity:}

\noindent
It exists a positive constant $c_{\rm{P}}$ depending on $c_{\rm{M}}$ such that for any $E \in \Omega_h$, $f \in \Edges$  and $\zz \in \mathcal{N}_h$
\[
\max_{E \in \Omega_h} \max_{f \in \EdgesE} \frac{h_E}{h_f} \leq c_{\rm{P}} 
\,,
\quad
\max_{E' \in \Omega_h^f} \frac{h_{\omega_f}}{h_{E'}} \leq c_{\rm{P}}
\,,
\quad
\max_{E', E'' \in \Omega_h^{\zz}} \frac{h_E'}{h_{E''}} \leq c_{\rm{P}} \,,
\quad
\max_{E' \in \Omega_h^{\zz}} \frac{h_{\zz}}{h_{E'}} \leq c_{\rm{P}} 
\,.
\]

\noindent
For $m \in \N$ and for $S \subseteq \Omega_h$, we introduce the polynomial spaces 
\begin{itemize}
\item $\Pk_m(\omega)$ is the set of polynomials on $\omega$ of degree $\leq m$, with $\omega$ a generic set;
\item $\Pk_m(S) := \{q \in L^2\bigl(\cup_{E \in S}E \bigr) \quad \text{s.t.} \quad q|_{E} \in  \Pk_m(E) \quad \text{for all $E \in S$}\}$;
\item $\Pkc_m(S) := \Pk_m(S) \cap C^0\bigl(\cup_{E \in S}E \bigr)$.
\end{itemize}
For $s \in \R^+$ and  $p \in [1,+\infty]$ let us define the broken Sobolev spaces:
\begin{itemize}
\item $W^s_p(S) := \{\phi \in L^2(\Omega) \quad \text{s.t.} \quad \phi|_{E} \in  W^s_p(E) \quad \text{for all $E \in S$}\}$,
\end{itemize}
equipped with  the standard broken norm 
$\Vert \cdot \Vert_{W^s_p(S)}$ 
and seminorm 
$\vert \cdot \vert_{W^s_p(S)}$.

\noindent
For any $E \in \Omega_h$, $\nn_E$  denotes the outward normal vector to $\partial E$.
For any mesh face $f$ let $\nn_f$ be a fixed unit normal vector to the face $f$.
Notice that for any $E \in \Omega_h$ and $f \in \EdgesE$ it holds $\nn_f = \pm \nn_E$.
We assume that for any boundary face $f \in \EdgesB \cap \EdgesE$ it holds
$\nn_f = \nn_E = \nn$, i.e. $\nn_f$ is the outward to $\partial \Omega$.

\noindent
The jump and the average operators on $f \in \EdgesI$ are defined for every piecewise continuous function w.r.t. $\Omega_h$ respectively by
\[
\begin{aligned}
\jump{\phi}_f(\xx) &:=
\lim_{s \to 0^+} \left( \phi(\xx - s \nn_f) - \phi(\xx + s \nn_f) \right)
\\
\media{\phi}_f(\xx) &:= 
\frac{1}{2}\lim_{s \to 0^+} \left( \phi(\xx - s \nn_f) + \phi(\xx + s \nn_f) \right)
\end{aligned}
\]
and $\jump{\phi}_f(\xx) = \media{\phi}_f(\xx) = \phi(\xx)$ on $f \in \EdgesB$.

\noindent
Let $\mathcal{D}$ denote one of the differential operators $\bnabla$, $\beps$, $\bcurl$.
Then, $\mathcal{D}_h$ represents the broken operator defined for all $\boldsymbol{\phi} \in \HH^1(\Omega_h)$ as
$\mathcal{D}_h (\boldsymbol{\phi} )|_E := \mathcal{D} (\boldsymbol{\phi} |_E)$ for all $E \in \Omega_h$.

\noindent
Finally, given $m \in \N$, we denote with
$\Pzerok{m} \colon L^2(\Omega) \to \Pk_m(\Omega_h)$ the  $L^2$-projection operator onto the space of polynomial functions.
%
The above definitions extend to vector valued and tensor valued functions.

%
%

%

In the following $C$ will denote a generic positive constant, independent of the mesh size $h$, of the diffusive coefficients $\ns$ and $\nm$, of the loadings  $\ff$ and $\GG$, of the problem solution $(\uu, p, \BB)$,  but which may depend on $\Omega$, on the order of the method $k$ (introduced in Section \ref{sec:stab}), on the final time $T$ and on the mesh regularity constants $c_{\rm{M}}$ and  $\widetilde{c}_{\rm{M}}$ in Assumptions \textbf{(MA1)} and \textbf{(MA2)}.
The shorthand symbol $\lesssim$ will denote a bound up to $C$.

We mention a list of classical results (see for instance \cite{brenner-scott:book}) that will be useful in the sequel.

\begin{lemma}[Trace inequality]
\label{lm:trace}
Under the mesh assumption \cfan{(MA1)}, for any $E \in \Omega_h$ and for any  function $v \in H^1(E)$ it holds 
\[
\sum_{f \in \EdgesE}\|v\|^2_{f} \lesssim h_E^{-1}\|v\|^2_{E} + h_E\|\nabla v\|^2_{E} \,.
\]
\end{lemma}

\begin{lemma}[Bramble-Hilbert]
\label{lm:bramble}
Under the mesh assumption \cfan{(MA1)}, let $m \in {\mathbb N}$. For any $E \in \Omega_h$ and for any  smooth enough function $\phi$ defined on $\Omega$, it holds 
\[
\Vert\phi - \Pzerok{m} \phi \Vert_{W^r_p(E)} \lesssim h_E^{s-r} \vert \phi \vert_{W^s_p(E)} 
\qquad  \text{$s,r \in \N$, $r \leq s \leq m+1$, $p \in [1, \infty]$.}
\]
\end{lemma}

\begin{lemma}[Inverse estimate]
\label{lm:inverse}
Under the mesh assumption \cfan{(MA1)}, let $m \in {\mathbb N}$. 
Then for any $E \in \Omega_h$, for any $1\leq p,q \leq \infty$, and for any $p_m \in \Pk_m(E)$ it holds 
\[
\|p_m\|_{W^s_p(E)} \lesssim h_E^{3/p - 3/q -s} \|p_m\|_{L^q(E)} 
\]
where the involved constant only depends on $m$, $s$, $p$, $q$ and $c_{\rm{M}}$.
\end{lemma}

\begin{remark}
\label{rm:jumpmedia}
For any face $f \in \Edges$, let $\jumpmedia{\cdot}_f$ denote the jump or the average operator on the face $f$.
We notice that for any $\mathbb{K} \in [L^{\infty}(\Omega_h)]^{3 \times 3}$ and for any $\ww \in \boldsymbol{H}^1(\Omega_h)$ and $\alpha \in \mathbb{Z}$, mesh assumption \cfan{(MA1)} and  Lemma \ref{lm:trace} yield the following estimate
\begin{equation}
\label{eq:utile-jump1}
\begin{aligned}
\sum_{f \in \Edges} h_f^\alpha \|\jumpmedia{\mathbb{K} \ww}_f\|^2_{f}
&\lesssim  
\sum_{E \in \Omega_h} h_E^\alpha \|\mathbb{K}\|^2_{L^\infty(E)} 
\sum_{f \in \EdgesE} \|\ww\|^2_{f}
\\
& \lesssim  
\sum_{E \in \Omega_h} \|\mathbb{K}\|^2_{L^\infty(E)}  
\left( h_E^{\alpha-1} \|\ww\|^2_{E} + h_E^{\alpha+1} \|\bnabla \ww\|^2_{E} \right) \,.
\end{aligned}
\end{equation}
In particular if $\ww \in [\Pk_m(\Omega_h)]^3$ by Lemma \ref{lm:inverse} it holds that
\begin{equation}
\label{eq:utile-jump}
\sum_{f \in \Edges}  h_f^\alpha \|\jumpmedia{\mathbb{K} \ww}_f\|^2_{f}
\lesssim  
\sum_{E \in \Omega_h}  h_E^{\alpha-1} \|\mathbb{K}\|^2_{L^\infty(E)}   \|\ww\|^2_{E} \,.
\end{equation}
\end{remark}

We close this section with the following instrumental result (for $k>1$ we refer to \cite[Lemma 3.2]{CIP}, whereas for $k=1$ we refer to \cite[Lemma 5.6]{MHD-linear}).

\begin{lemma}[]
\label{lm:cip}
Let Assumption \cfan{(MA1)} hold. Furthermore, if $k=1$ let also Assumption \cfan{(MA2)} hold. 
Let
\begin{equation}
\label{eq:mathbbO}
\mathbb{O}_{k-1}(\Omega_h) := 
\Pkc_{k-1}(\Omega_h)  \quad  \text{for $k>1$,}
\qquad
\mathbb{O}_{k-1}(\Omega_h) := \Pk_{0}(\widetilde{\Omega}_h)  \quad \text{for $k=1$.}
\end{equation}
There exists a projection operator $\intcip \colon \Pk_{k-1}(\Omega_h) \to \mathbb{O}_{k-1}(\Omega_h)$
such that for any $p_{k-1} \in \Pk_{k-1}(\Omega_h)$   the following holds:
\[
\sum_{E \in \Omega_h}h_E  \Vert (I- \intcip) p_{k-1} \Vert_E^2 \lesssim 
\sum_{f \in \EdgesI}
h_f^2 \Vert \jump{p_{k-1}}_f\Vert_f^2
\, .
\]
\end{lemma}

\section{Stabilized Finite Elements discretizations}
\label{sec:stab}

In this section we present the two stabilized methods here proposed (three-field and four-field) and prove some technical results that will be useful in the interpolation and convergence analysis of the forthcoming sections.
 Since the novelty of this contribution is in the space discretization, in our presentation and analysis we will focus on the time-continuous case. Clearly, a choice of a time-stepping time integrator will be taken in the numerical tests section.

\subsection{Discrete spaces and interpolation analysis}
\label{sub:spaces}

Let the integer $k \geq 1$ denote the order of the method. We consider the following discrete spaces
\begin{align}
\label{eq:spazi_3f}
\VVd := [\Pk_k(\Omega_h)]^3 \cap \HH_0(\diver, \Omega)\,, \,\,\,
\Qd  &:= \Pk_{k-1}(\Omega_h) \cap L^2_0(\Omega)\,, \,\,\,
\WWd := [\Pkc_k(\Omega_h)]^3 \cap \Hunoenne\,,
\\
\label{eq:spazi_4f}
\Rd  &:= \Pkc_k(\Omega_h) \cap L^2_0(\Omega),
\end{align}
approximating the velocity field space $\VVc$, the pressure space $\Qc$, the magnetic induction space $\WWc$ and the  lagrangian multiplier space (adopted only in the four-field scheme) respectively. 

Notice that in the proposed method we adopt the $\HH(\diver)$-conforming 
$\boldsymbol{{\rm BDM}}_k$ element \cite{BDM} for the approximation of the velocity space
that provides exact divergence-free discrete velocity,
and preserves the pressure-robustness of the resulting scheme \cite{Hdiv1,Hdiv2,Han}.
%
%
Let us introduce the discrete kernel
\begin{equation}
\label{eq:Z_h}
\ZZd := \{ \vv_h \in \VVd \quad \text{s.t.} \quad \diver \, \vv_h = 0  \}\,.
\end{equation}

We now define the interpolation operators $\PVVd$ and $\PWWd$, acting on the spaces $\VVd$ and $\WWd$ respectively, satisfying optimal approximation estimates and suitable \emph{local} orthogonality properties that will be instrumental to prove the convergence results 
(without the need to require a quasi-uniformity property on the mesh sequence, as it often happens in CIP stabilizations for nonlinear problems).  For what concerns the operator $\PVVd$, we recall from \cite{BDM,GSS} the  following result.

\begin{lemma}[Interpolation operator on $\VVd$]
\label{lm:int-v}
Under the Assumption \cfan{(MA1)} let $\PVVd \colon \VVc \to \VVd$ be the interpolation operator defined in equation (2.4) of \cite{BDM}.
The following hold

\noindent
$(i)$ if $\vv \in \ZZc$ then $\vvi \in \ZZd$;

\noindent
$(ii)$ for any $\vv \in \ZZc$
\begin{equation}
\label{eq:orth-v}
\left(\vv - \, \vvi, \,  \pp_{k-1} \right) = 0 \quad \text{for all $\pp_{k-1} \in [\Pk_{k-1}(\Omega_h)]^3$;}
\end{equation}

\noindent
$(iii)$ for any $\vv \in \VVc \cap \HH^{s+1}(\Omega_h)$, with $0 \leq s \leq k$ and for all $E \in \Omega_h$, it holds
\begin{equation}
\label{eq:int-v}
\vert \vv - \vvi \vert_{m,E} \lesssim h_E^{s+1-m} \vert \vv \vert_{s+1,E} 
\qquad \text{for $0\leq m\leq s+1$;}
\end{equation}
\noindent
$(iv)$ for any $\vv \in \VVc \cap \WW^1_\infty(\Omega_h)$ and for all $E \in \Omega_h$, it holds
\begin{equation}
\label{eq:int-v-inf}
\Vert \vv - \vvi \Vert_{\LL^\infty(E)} +
h_E \Vert \bnabla (\vv - \vvi) \Vert_{\LL^\infty(E)}  
 \lesssim h_E \Vert  \vv \Vert_{\WW^1_\infty(E)} \,.
\end{equation}
\end{lemma}

Concerning the approximation property of the operator $\PWWd$ we state the following Lemma and refer to \cite[Lemma 4.3]{MHD-linear} for the proof. 

\begin{lemma}[Interpolation operator on $\WWd$]
\label{lm:int-i}
Let Assumption \cfan{(MA1)} hold. Furthermore, if $k=1$ let also Assumption \cfan{(MA2)} hold.
Then there exists an interpolation operator $\PWWd \colon \WWc \to \WWd$
satisfying the following 

\noindent
$(i)$ referring to \eqref{eq:mathbbO}, for any $\HH \in \WWc$
\begin{equation}
\label{eq:int-orth}
\left( \HH - \HHi, \, \qq_{k-1} \right) = 0 
\qquad
\text{for any $\qq_{k-1} \in [\mathbb{O}_{k-1}(\Omega_h)]^3$;}
\end{equation}

\noindent
$(ii)$ for any $\HH \in \WWc \cap \HH^{s+1}(\Omega_h)$ with $0 \leq s \leq k$, for $\alpha=0,1,2$, it holds

\begin{equation}
\label{eq:err-int}
\sum_{E \in \Omega_h} h_E^{-\alpha} \Vert \HH - \HHi \Vert_E^2
\lesssim h^{2s+2-\alpha} \vert \HH \vert_{s+1,\Omega_h}^2 \,,
\quad
\Vert \bnabla (\HH - \HHi) \Vert \lesssim h^{s} \vert \HH \vert_{s+1,\Omega_h} \,. 
\end{equation}
\end{lemma}

Finally we state the following useful lemma.
Analogous result can be found in \cite[Corollary 1]{BF:2007} and \cite{bertoluzza,Crouzeix}, but we here prefer to derive a simpler proof in the present less general context.

\begin{lemma}(Interpolation operator on $\Pkc_k(\Omega_h)$)
\label{lm:Pt}
Let $s,p \in\R$ such that $sp > 3$ and let $\Pt \colon W^s_p(\Omega) \to \Pkc_k(\Omega_h)$  be the Lagrangian interpolator on the space $\Pkc_k(\Omega_h)$.
Under the Assumption \cfan{(MA1)} the following hold

\noindent 
$(i)$ there exists a real positive constant $\clemma$ such that
for any $u \in H^{3/2 + \epsilon}(\Omega) \cap W^1_\infty(\Omega_h)$ for $\epsilon>0$ and for any $S\subseteq \Omega_h$ it holds
\begin{equation}
\label{eq:Pt2}
\begin{gathered}
\Vert \Pt u \Vert_{L^\infty(\cup_{E \in S}E)} \leq \clemma
\Vert  u \Vert_{L^\infty(\cup_{E \in S}E)}  \,,
\qquad
\Vert \Pt u \Vert_{W_1^\infty(S)} \leq \clemma
\Vert  u \Vert_{W_1^\infty(S)} \,,
\\
\Vert u - \Pt u\Vert_{L^\infty(E)} + h_E \Vert \nabla(u - \Pt u) \Vert_{L^\infty(E)}
\lesssim h_E \Vert u \Vert_{W^1_\infty(E)} \,;
\end{gathered}
\end{equation}

\noindent
$(ii)$ for any $u \in H^{3/2 + \epsilon}(\Omega) \cap W^1_\infty(\Omega_h)$ for $\epsilon >0$ and $v_h \in \Pkc_k(\Omega_h)$ it holds
\begin{align}
\label{eq:Pt3}
\sum_{E \in \Omega_h} h_E^{-2} \Vert uv_h - \Pt(u v_h)\Vert_{E}^2 \lesssim \Vert  u \Vert_{W^1_\infty(\Omega_h)}^2 \, \Vert v_h \Vert^2 \,,
\\
\label{eq:Pt4}
\sum_{E \in \Omega_h} h_E^{-2}
\Vert \Pt(u v_h) - (\Pt u)v_h\Vert_{E}^2 \lesssim  \Vert  u \Vert_{W^1_\infty(\Omega_h)}^2 \, \Vert v_h \Vert^2 \,.
\end{align}
\end{lemma}

\begin{proof}
Item $(i)$ is a classical result in finite elements theory
(see for instance \cite{quarteroni-valli:book}).
We now prove bound \eqref{eq:Pt3}.
Let $\Pth \colon W^s_p(\Omega_h) \to \Pk_k(\Omega_h)$ be the piece-wise discontinuous counterpart of $\Pt$, let $\widetilde{u} := u - \Pi_0 u$ and $v_h \in \Pkc_k(\Omega_h)$.
Then, being $\Pth((\Pi_0 u) v_h) = (\Pi_0 u) v_h \in \Pk_k(\Omega_h)$,
employing Lemma \ref{lm:bramble} and Lemma \ref{lm:inverse},  we have 
\[
\begin{aligned}
\sum_{E \in \Omega_h} &h_E^{-2} \Vert uv_h - \Pt(u v_h)\Vert_{E}^2 =
\sum_{E \in \Omega_h} h_E^{-2} \Vert uv_h - \Pth(u v_h)\Vert_{E}^2 
\\
&= \sum_{E \in \Omega_h} h_E^{-2} \Vert \widetilde{u} v_h - \Pth(\widetilde{u} v_h)\Vert_{E}^2
\lesssim \sum_{E \in \Omega_h} h_E^{-2} \left(\Vert \widetilde{u} v_h \Vert_E^2  +
\Vert \Pth(\widetilde{u} v_h)\Vert_{E}^2 \right)
\\
& \lesssim \sum_{E \in \Omega_h} h_E^{-2} 
\Vert \widetilde{u}\Vert_{L^\infty(E)}^2 
\left( \Vert v_h \Vert_E^2  +
|E| \,  \max_{\xx \in \mathcal{N}_h^E} |v_h(\xx)|^2 \right)
\\
& \lesssim \sum_{E \in \Omega_h} \Vert u \Vert_{W^1_\infty(E)}^2 
\left( \Vert v_h \Vert_E^2  + |E| \, \Vert v_h \Vert_{L^\infty(E)}^2 \right) 
\\
& \lesssim \sum_{E \in \Omega_h} \Vert u \Vert_{W^1_\infty(E)}^2  \Vert v_h \Vert_E^2
\leq
\Vert u \Vert_{W^1_\infty(\Omega_h)}^2  \Vert v_h \Vert^2 \,.
\end{aligned}
\]
Regarding  \eqref{eq:Pt4}, from \eqref{eq:Pt2} and \eqref{eq:Pt3} we infer
\[
\begin{aligned}
\sum_{E \in \Omega_h} &h_E^{-2} \Vert \Pt(uv_h) - (\Pt u) v_h\Vert_{E}^2 \lesssim
\sum_{E \in \Omega_h} h_E^{-2} \left(
\Vert  u v_h - \Pt(uv_h)\Vert_{E}^2 +
\Vert  u v_h - (\Pt u) v_h\Vert_{E}^2  \right)
\\
&\lesssim
\Vert u \Vert_{W^1_\infty(\Omega_h)}^2  \Vert v_h \Vert^2 + 
 \sum_{E \in \Omega_h} h_E^{-2} \Vert u - \Pt u \Vert_{L^\infty(E)}^2  \Vert v_h\Vert_{E}^2
\\
& \lesssim 
\Vert u \Vert_{W^1_\infty(\Omega_h)}^2  \Vert v_h \Vert^2 +
\sum_{E \in \Omega_h} \Vert u \Vert_{W^1_\infty(E)}^2 \Vert v_h\Vert_{E}^2
\lesssim 
\Vert u \Vert_{W^1_\infty(\Omega_h)}^2  \Vert v_h \Vert^2 \,.
\end{aligned}
\] 
\end{proof}

\begin{remark}
\label{rm:jumpmediaeei}
With the same notations of Remark \ref{rm:jumpmedia}, 
for any $\vv \in \VVc \cap \HH^{s+1}(\Omega_h)$ and for any 
$\HH \in \WWc \cap \HH^{s+1}(\Omega_h)$
 with $0 \leq s \leq k$
combining \eqref{eq:utile-jump1} with \eqref{eq:int-v} or \eqref{eq:err-int}, for $\alpha = -1, 0, 1$, the following hold
\begin{align}
\label{eq:utilejump2}
\sum_{f \in \Edges} h_f^\alpha \|\jumpmedia{\vv - \vvi}_f\|^2_{f} +
\sum_{f \in \Edges} h_f^{\alpha+2} \|\jumpmedia{\bnabla_h(\vv - \vvi)}_f\|^2_{f}
&\lesssim   
h^{2s + 1 + \alpha} \vert \vv \vert_{s+1,\Omega_h}^2 \,,
\\
\label{eq:utilejump2-h}
\sum_{f \in \Edges} h_f^\alpha \|\HH - \HHi_f\|^2_{f} 
&\lesssim   
h^{2s + 1 + \alpha} \vert \HH \vert_{s+1,\Omega_h}^2 \,.
\end{align}
\end{remark}

We note that the (possible) negative powers of $h_E$ on the left hand side of the above results express the locality of the estimates, which concur in avoiding a quasi-uniformity mesh assumption in our analysis.

\begin{remark}
\label{rm:support}
Let $\pp_m \in [\Pk_m(\Omega_h)]^3$, $\uu \in \VVc \cap \WW^1_\infty(\Omega_h)$ and let $\widetilde{\uu}$ denote $\uui$ (resp. $\Pi_0 \uu$). 
Then employing  \eqref{eq:int-v-inf} (resp. Lemma \ref{lm:bramble}) and  Lemma \ref{lm:inverse} we infer 
\begin{equation}
\label{eq:support}
\begin{aligned}
\Vert |\bnabla_h \pp_m| \, |\uu - \widetilde{\uu}| \Vert^2
&=
\sum_{E \in \Omega_h}
\Vert |\bnabla_h \pp_m| \, |\uu - \widetilde{\uu}|  \Vert_E^2
\leq 
\sum_{E \in \Omega_h}
\Vert \uu - \widetilde{\uu} \Vert_{\LL^\infty(E)}^2 \, 
\Vert \bnabla \pp_m \Vert_E^2 
\\ 
&\lesssim
\sum_{E \in \Omega_h}
h_E^2 \Vert \bnabla_h \uu \Vert_{\LL^\infty(E)}^2 \, 
h_E^{-2} \Vert \pp_m \Vert_E^2 
\lesssim
\Vert  \uu \Vert_{\WW^1_\infty(\Omega_h)}^2 \, 
\Vert \pp_m \Vert^2 \,.
\end{aligned} 
\end{equation}
Analogous result can be obtained replacing $\uu$ with $\BB \in \WWc \cap \WW^1_\infty(\Omega_h)$ and $\widetilde{\uu}$ with $\widetilde{\BB}= \Pi_0 \BB$. 
\end{remark}

\subsection{Discrete forms}
\label{sub:forms}

In the present section we define the discrete forms at the basis of the proposed stabilized schemes.
Let $\epsilon>0$ and
\begin{equation}
\label{eq:stars}
\VVs := \VVc \cap \HH^{3/2+\epsilon}(\Omega) \,,
\qquad
\ZZs := \ZZc \cap \HH^{3/2+\epsilon}(\Omega) \,,
\qquad
\WWs := \WWc \cap \HH^{3/2+\epsilon}(\Omega) \,.
\end{equation}
Due to the coupling between fluid-dynamic equation and magnetic equation, in addition to the classical upwinding, in the proposed schemes we consider extra stabilizing forms (in the spirit of continuous interior penalty \cite{BFH:2006,CIP}) that penalize the jumps and the gradient jumps along the convective directions $\uu_h$ and $\BB_h$. 
Rearranging in the non-linear setting the formulation in \cite{MHD-linear}, we here consider several forms.

\noindent
$\bullet$ DG counterparts of the continuous forms in \eqref{eq:forme_c1}. Let
\[
\ash(\cdot, \cdot) \colon (\VVs \oplus \VVd) \times \VVd \to \R,
\quad
c_h(\cdot; \cdot, \cdot) \colon (\VVs \oplus \VVd) \times (\VVs \oplus \VVd) \times \VVd(\Omega_h) \to \R,
\] 
be defined respectively by
\begin{equation}
\label{eq:forme_d}
\begin{aligned}
\ash(\uu,  \vv_h) &:=  (\beps_h(\uu) ,\, \beps_h(\vv_h))
- \sum_{f \in \Edges} (\media{\beps_h(\uu)\nn_f}_f ,\, \jump{\vv_h}_f)_f  +
\\
&- \sum_{f \in \Edges} (\jump{\uu}_f ,\, \media{\beps_h(\vv_h) \nn_f}_f)_f 
+ 
\bdma \sum_{f \in \Edges} h_f^{-1} (\jump{\uu}_f ,\,\jump{\vv_h}_f)_f
\\
c_h(\cc; \uu, \vv_h) &:=  (( \bnabla_h \uu ) \, \cc, \, \vv_h )
- \sum_{f \in \EdgesI} ( (\cc \cdot \nn_f) \jump{\uu}_f ,\, \media{\vv_h}_f)_f +
\\
& + 
\bdmc \sum_{f \in \EdgesI} (\vert \cc \cdot \nn_f \vert \jump{\uu}_f, \, \jump{\vv_h}_f )_f
\end{aligned}
\end{equation}
where the penalty parameters $\bdma$ and $\bdmc$ have to be sufficiently large in order to guarantee the coercivity of the form $\ash(\cdot, \cdot)$  and the stability effect in the convection dominated regime due to the upwinding \cite{GSS,Hdiv2,upwinding, DiPietro, Han}.

\noindent 
$\bullet$ Stabilizing CIP form that penalizes the jumps and the gradient jumps along the convective directions $\BB_h$. Let 
\[
J_h(\cdot; \cdot, \cdot) \colon \WWd \times (\VVs \oplus \VVd) \times \VVd \to \R
\]
be the  form defined by
\begin{equation}
\label{eq:J}
\begin{aligned}
&J_h(\TT; \uu, \vv_h) :=   
\\
& \sum_{f \in \EdgesI} \max\{\Vert \TT\Vert_{L^\infty(\omega_f)}^2, 1 \}
\left( 
\bdmJU ( \jump{\uu}_f,  \jump{\vv_h}_f )_f + 
\bdmJD h_f^2  (\jump{\bnabla_h \uu}_f, \jump{\bnabla_h \vv_h}_f)_f 
\right)\,.
\end{aligned}
\end{equation}

\noindent 
$\bullet$ Stabilizing CIP form that penalizes the gradient jumps along the convective directions $\uu_h$, this form is needed only in the four-field formulation \eqref{eq:fem-4f}. Let 
\[
K_h(\cdot; \cdot, \cdot) \colon \VVd \times (\WWs \oplus \WWd) \times  \WWd \to \R
\]
be the form defined by
\begin{equation}
\label{eq:K}
K_h(\cc; \BB, \HH_h) :=   
\bdmK \sum_{f \in \EdgesI} h_f^2 \max \{\Vert \cc\Vert_{L^\infty(\omega_f)}^2, 1 \}
(\jump{\bnabla \BB}_f, \jump{\bnabla \HH_h}_f)_f \,.
\end{equation}

\noindent
$\bullet$ Stabilizing form for the multipliers needed only in the four-field formulation \eqref{eq:fem-4f}. Let
\[
Y_h(\cdot, \cdot) \colon \Pkc_k(\Omega_h) \times  \Pkc_k(\Omega_h) \to \R
\]
defined by
\begin{equation}
\label{eq:Y}
Y_h(\phi_h, \psi_h) :=   
\bdmY \sum_{f \in \EdgesI} 
h_f^2  (\jump{\nabla \phi_h}_f, \jump{\nabla \psi_h}_f)_f \,.
\end{equation}
In \eqref{eq:J},  \eqref{eq:K} and \eqref{eq:Y}, $\bdmJU$, $\bdmJD$, $\bdmK$ and $\bdmY$ are user-dependent (positive) parameters. 

\begin{remark}
\label{rm:parametri}
The positive parameters $\bdmc$, $\bdmJU$, $\bdmJD$, $\bdmK$ fixed once and for all, 
are introduced in order to allow some tuning of the different stabilizing terms.
Since such uniform parameters do not affect the theoretical derivations, for the time being we set all the parameters equal to $1$.
We will be more precise about the practical values of such constants  in the numerical tests section.
\end{remark}


%


\subsection{Discrete three-field scheme}
\label{sub:scheme3f}

Referring to the spaces \eqref{eq:spazi_3f}, the forms \eqref{eq:forme_c2}, \eqref{eq:forme_d}, \eqref{eq:J}, the stabilized three-field method for the MHD equation is given by: find
\begin{itemize}
\item $\uu_h \in L^\infty(0, T; \VVd)$,
\item $p_h \in L^2(0, T; \Qd)$,
\item $\BB_h \in L^\infty(0, T; \WWd)$, 
\end{itemize}
such that for a.e. $t \in I$
\begin{equation}
\label{eq:fem-3f}
\left\{
\begin{aligned}
\left(\partial_t \uu_h, \vv_h\right) + \ns \ash(\uu_h, \vv_h) + c_h(\uu_h; \uu_h, \vv_h) + &
\\ 
-d(\BB_h; \BB_h, \vv_h)   + J_h(\BB_h; \uu_h, \vv_h) + b(\vv_h, p_h) 
&= 
(\ff, \vv_h)  
&\,\,\, & \text{$\forall \vv_h \in \VVd$,} 
\\
b(\uu_h, q_h) &= 0 
&\,\,\, & \text{$\forall q_h \in \Qd$,}
\\
\left(\partial_t \BB_h, \HH_h\right) + \nm \am(\BB_h, \HH_h) + d(\BB_h; \HH_h, \uu_h)  + &
\\ 
+(\diver \, \BB_h, \diver \, \HH_h)&=(\GG, \HH_h)  
&\,\,\, & \text{$\forall \HH_h \in \WWd$,}
\end{aligned}
\right.
\end{equation}
coupled with initial conditions (cf. Lemma \ref{lm:int-v} and Lemma \ref{lm:int-i})
\begin{equation}
\label{eq:ini cond-d}
\uu_h(\cdot, 0) = \uu_{h,0}:= {\cal I}_{\VVc}\uu_0\,, \quad
\BB_h(\cdot, 0) = \BB_{h,0}:={\cal I}_{\WWc}\BB_0\,.
\end{equation}
Notice that \eqref{eq:int-v} and \eqref{eq:err-int} easily imply
\begin{equation}
\label{eq:initial-bound}
\|\uu_{h,0}\| \lesssim \|\uu_{0}\| \,,
\qquad
\|\BB_{h,0}\| \lesssim \|\BB_{0}\| \,.
\end{equation}

In Section \ref{sec:theo-3f}  we assess the quasi-robustness of the scheme \eqref{eq:fem-3f} by deriving ``optimal'' (in the sense of best approximation) $h^k$ error estimates in a suitable discrete norm,  which do not degenerate for small values of the diffusion parameters $\ns$, $\nm$.


\subsection{Discrete  four-field scheme}
\label{sub:scheme4f}

This alternative approach does not rely only on the time derivative of $B_h$ in order to impose the solenoidal condition, but enforces it more strongly through a Lagrange multiplier. 
Referring to the spaces \eqref{eq:spazi_3f} and \eqref{eq:spazi_4f}, the forms \eqref{eq:forme_c2}, \eqref{eq:forme_d}, \eqref{eq:J}, \eqref{eq:K}, \eqref{eq:Y}, the stabilized four-field method for the MHD equation is given by: find
\begin{itemize}
\item $\uu_h \in L^\infty(0, T; \VVd)$, 
\item $p_h \in L^2(0, T; \Qd)$,
\item $\BB_h \in L^\infty(0, T; \WWd)$, 
\item $\phi_h  \in L^2(0, T; \Rd)$,
\end{itemize}
such that for a.e. $t \in I$
\begin{equation}
\label{eq:fem-4f}
\left\{
\begin{aligned}
\left(\partial_t \uu_h, \vv_h\right) + \ns \ash(\uu_h, \vv_h) + c_h(\uu_h; \uu_h, \vv_h) + &
\\ 
-d(\BB_h; \BB_h, \vv_h)   + J_h(\BB_h; \uu_h, \vv_h) + b(\vv_h, p_h) 
&= 
(\ff, \vv_h)  
&\,\,\, & \text{$\forall  \vv_h \in \VVd$,} 
\\
b(\uu_h, q_h) &= 0 
&\,\,\, & \text{$\forall q_h \in \Qd$,}
\\
\left(\partial_t \BB_h, \HH_h\right) + \nm \am(\BB_h, \HH_h) + d(\BB_h; \HH_h, \uu_h)  + &
\\ 
+ K_h(\uu_h; \BB_h, \HH_h) - b(\HH_h, \phi_h)
&= 
(\GG, \HH_h)  
&\,\,\, & \text{$\forall  \HH_h \in \WWd$,}
\\
Y_h(\phi_h, \psi_h) + b(\BB_h, \psi_h)
&= 0 
&\,\,\, & \text{$\forall \psi_h \in \Rd$,}
\end{aligned}
\right.
\end{equation}
coupled with initial conditions \eqref{eq:ini cond-d}.

In Section \ref{sec:theo-4f} we will show that the scheme \eqref{eq:fem-4f} is quasi-robust, i.e. the error does not degenerate for small values of the diffusion coefficients; furthermore, for this method the pre-asymptotic order of convergence in convection dominated cases is optimal (that is $O(h^{k+1/2})$) in the chosen discrete norm.

\begin{remark}[$H_{\textrm{curl}}$ conforming elements for the magnetic fluxes]
Notice that, when the domain $\Omega$ is non-convex, 
it is mandatory to consider an $H_{\textrm{curl}}$ conforming space (for instance the N\'ed\'elec element)  for the approximation of the magnetic field (see \cite[Chapter 3]{Gerbeau}).
This choice would be natural within the context of the four-field scheme since we consider a Lagrange multiplier to enforce divergence constraint of $\BB_h$,  and we incorporate a weak divergence constraint via an orthogonality on gradients.
However we prefer to keep the $H^1$-conforming Lagrange element for the magnetic field also for the four-field scheme for essentially two reasons.
The first one is the efficiency of the scheme. Although the non-convex case is surely of interest, a large part of the domains in real applications are still convex, 
and in such cases, the nodal FEM holds a substantially smaller number of degrees of freedom when compared to the $H_{\textrm{curl}}$ conforming one.
The second reason is the clarity of exposition. As it is, the three-field and the four-field methods are more strongly related, also from the coding viewpoint, which makes the presentation easier and more concise.
We finally note that developing the theoretical results for the $H_{curl}$ FEM conforming case, especially considering low regularity magnetic field solutions (not in $H^1$) would need additional developments.
\end{remark}

\begin{remark} 
 \label{rem:3Fto4F}
In the preliminary article \cite{MHD-linear} we tackled the stationary linear counterpart of the same model.
The present contribution borrows from \cite{MHD-linear} some technical results, mainly related to the orthogonality and local approximation properties of the discrete magnetic interpolant (essentially Lemmas \ref{lm:cip} and \ref{lm:int-i}). 
The primary challenges of the fully nonlinear case stem from the fact that, unlike in the linearized problem where certain vector fields are assumed to be known and smooth, these fields in the trilinear forms become unknown and discrete.
This explains why the theoretical error estimates of the three-field scheme in \cite{MHD-linear} exhibit an optimal $O(h^{k+1/2})$ pre-asymptotic error reduction rate in convection dominated cases ($\ns, \nm \lesssim h$) while here we can prove at most $O(h^{k})$. From our convergence bounds below, we identify the reason with an un-sufficiently strong imposition of the solenoidal condition for $\BB_h$. This justifies the introduction of the alternative, and more complex, approach with four fields of the next section (see also Remark \ref{rem:4F-teaser} below).
\end{remark}

\section{Theoretical analysis of the three-field scheme}
\label{sec:theo-3f}

We preliminary make the following assumption on the velocity solution $\uu$ of problem \eqref{eq:variazionale}.

\smallskip
\noindent
\textbf{(RA1-3f) Regularity assumption for the consistency:}

\noindent
Let $(\uu, p, \BB)$ be the solution  of Problem \eqref{eq:variazionale}. 
Then 
$\uu(\cdot, t) \in \ZZs$ (cf. \eqref{eq:stars}) for a.e. $t \in I$.

\smallskip
\noindent
Under the Assumption \textbf{(RA1-3f)}, the discrete forms in \eqref{eq:forme_d} and \eqref{eq:J} satisfy for a.e. $t \in I$ and for all 
$\cc \in \ZZs \oplus \ZZd$ and $\TT \in \WWd$ the following consistency property
\begin{gather}
\label{eq:forme_con}
\ash(\uu, \vv_h) = -(\bdiver (\beps (\uu)) ,\, \vv_h )\,,
\,\,\,
c_h(\cc; \uu, \vv_h) = c(\cc; \uu, \vv_h)\,, 
\,\,\,
J_h(\TT; \uu, \vv_h) = 0 \quad
\text{$\forall \vv_h \in \VVd$,}
\\
\label{eq:forme_con3f}
(\diver \BB, \diver \HH_h) = 0 
\quad 
\text{$\forall \HH_h \in \WWd$,}
\end{gather}
i.e. all the forms in \eqref{eq:fem-3f} are consistent.

\subsection{Stability analysis}
\label{sub:stab-3f}

Recalling the definition \eqref{eq:Z_h}, consider the form
\[
\astabt \colon 
(\ZZd \times \WWd) \times
(\ZZd \times \WWd) \times
(\ZZd \times \WWd) \to \R 
\]
defined by
\begin{equation}
\label{eq:astabt}
\begin{aligned}
\astabt&((\cc, \TT); (\uu_h, \BB_h), (\vv_h, \HH_h))  :=
\ns \ash(\uu_h, \vv_h) + \nm \am(\BB_h, \HH_h) + J_h(\TT; \uu_h, \vv_h)+
\\  
& 
+ c_h(\cc; \uu_h, \vv_h) 
-d(\TT; \BB_h, \vv_h) 
+d(\TT; \HH_h, \uu_h)
+ (\diver \BB_h, \diver \HH_h)
 \,.
\end{aligned}
\end{equation}
Then Problem \eqref{eq:fem-3f} can be formulated as follows: find
$\uu_h \in L^\infty(0, T; \ZZd)$, 
$\BB_h \in L^\infty(0, T; \WWd)$, 
such that for a.e. $t \in I$
\begin{equation}
\label{eq:stab fem-3f}
\begin{aligned}
\left( \partial_t \uu_h, \vv_h\right) &+ 
\left( \partial_t \BB_h, \HH_h\right) +  
\astabt((\uu_h, \BB_h), (\uu_h, \BB_h), (\vv_h, \HH_h))
=\\
& = (\ff, \vv_h)   + (\GG, \HH_h)  
\quad \quad \text{$\forall (\vv_h, \HH_h) \in \ZZd \times \WWd$,}
\end{aligned}
\end{equation}
coupled with initial conditions \eqref{eq:ini cond-d}.

For any $\cc \in\ZZd$, and  $\TT \in \WWd$ we define the following norms and semi-norms  on  $\VVs \oplus \VVd$
\begin{equation}
\label{eq:norme}
\begin{aligned}
\Vert\uu \Vert_{1,h}^2 &:= 
\Vert \beps_h(\uu) \Vert^2 + 
\bdma \sum_{f \in \Edges} h_f^{-1} \Vert \jump{\uu}_f \Vert_{f}^2
\\
\normaupwc{\uu}^2 &:= \sum_{f \in \EdgesI}  \Vert \vert \cc \cdot \nn_f \vert^{1/2} \jump{\uu}_f \Vert_{f}^2
\\
\normacipT{\uu}^2 &:=
\sum_{f \in \EdgesI} \max \{\Vert \TT\Vert_{\LL^\infty(\omega_f)}^2, 1\}
\left( 
\Vert \jump{\uu}_f \Vert_f^2 + 
h_f^2  \Vert \jump{\bnabla_h \uu}_f \Vert_f^2
\right)
\\
\normaV{\uu}^2 &:= 
\ns \Vert \uu\Vert_{1,h}^2 + \normaupwc{\uu}^2 + \normacipT{\uu}^2 \,. 
\end{aligned}
\end{equation}
We also define the following norm on $\WWc$
\begin{equation}
\label{eq:norma mag 3f}
\normamt{\BB}^2 := 
\nm \Vert \bnabla \BB \Vert^2  + \Vert \diver \, \BB \Vert^2\,.
\end{equation}

The following result is instrumental to prove the well-posedness of problem \eqref{eq:stab fem-3f}.

\begin{proposition}[Coercivity of $\astabt$]
\label{prp:coe}
Under the mesh assumption \cfan{(MA1)} 
let $\astabt$ be the form \eqref{eq:astabt}.
If the parameter $\bdma$ in \eqref{eq:forme_d} is sufficiently large there exists a real positive constant $\ccoe$ such that for any $\cc \in \ZZd$ and $\TT \in \WWd$
the following holds
\[
\astabt((\cc, \TT); (\vv_h, \HH_h), (\vv_h, \HH_h)) \geq 
\ccoe \left( \normaV{\vv_h}^2 + \normamt{\HH_h}^2 \right) 
\quad \text{$\forall \vv_h \in \ZZd$, $\forall \HH_h \in \WWd$.}
\]
\end{proposition}

\begin{proof}
The proof easily follows combining standard DG-theory arguments (see  for instance \cite[Lemma 4.12, Lemma 2.18]{DiPietro}) and the embedding \eqref{eq:well3}. 
\end{proof}

\begin{proposition}[Well-posedness of \eqref{eq:stab fem-3f}]
\label{prp:well-3f}
Under the mesh assumption \cfan{(MA1)} 
if the parameter $\bdma$ in \eqref{eq:forme_d} is sufficiently large, Problem \eqref{eq:stab fem-3f} admits a unique solution $(\uu_h, \BB_h)$. 
Moreover the following bounds hold (cf. \eqref{eq:data})
\begin{align}
\label{eq:stab3f-1}
\Vert \uu_h \|_{L^\infty(0, T; \Ldue)}^2 + 
\Vert \BB_h \|_{L^\infty(0, T; \Ldue)}^2 
&\lesssim \data^2 \,,
\\
\label{eq:stab3f-2}
\Vert \uu_h(\cdot, T) \Vert^2 + \Vert \BB_h(\cdot, T) \Vert^2 +
\int_0^T \|\uu_h(\cdot, t)\|_{(\uu_h(\cdot, t), \BB_h(\cdot, t))}^2 \, {\rm d}t  +
\int_0^T \normamt{\BB_h(\cdot, t)}^2 \, {\rm d}t  
&\lesssim \data^2 \,.
\end{align}
\end{proposition}

\begin{proof}
The existence of a unique solution $(\uu_h, \BB_h)$ to the Cauchy problem \eqref{eq:stab fem-3f} can be derived using analogous arguments to that in  \cite{layton2008} (see also \cite{Hm1}) and follows by the Lipschitz continuity of $\astabt$. We now prove the stability bounds. 

\noindent 
$\bullet$ Bound \eqref{eq:stab3f-1}. 
Direct computations yield for a.e. $t \in I$
\begin{equation}
\label{eq:bound3f-1}
\begin{aligned}
& \left(\Vert \uu_h  \Vert^2 + 
\Vert \BB_h \Vert^2 \right)^{1/2} \, \partial_t \big[ \left(\Vert \uu_h \Vert^2 + 
\Vert \BB_h\Vert^2 \right)^{1/2} \big] 
\\
&\lesssim
\partial_t \Vert \uu_h  \Vert^2 + 
\partial_t \Vert \BB_h \Vert^2 +
\|\uu_h\|_{(\uu_h, \BB_h)}^2 +
\normamt{\BB_h}^2
\\
&\lesssim
\left(\partial_t \uu_h  , \uu_h \right) + 
\left(\partial_t \BB_h  , \BB_h \right) + 
\astabt((\uu_h, \BB_h); (\uu_h, \BB_h), (\uu_h, \BB_h))
& \quad & \text{(by Prp. \ref{prp:coe})}
\\
& = (\ff, \uu_h) + (\GG, \BB_h)
& \quad & \text{(by \eqref{eq:stab fem-3f})}
\\
& \leq 
\left( \|\ff\|^2  + \|\GG\|^2\right)^{1/2} 
\left(\Vert \uu_h  \Vert^2 + 
\Vert \BB_h \Vert^2 \right)^{1/2}
\\
& \leq 
\left( \|\ff\|  + \|\GG\|\right) 
\left(\Vert \uu_h  \Vert^2 + 
\Vert \BB_h \Vert^2 \right)^{1/2} \,.
\end{aligned}
\end{equation}
Therefore, for a.e. $t \in I$  recalling \eqref{eq:initial-bound}, we get
\[
\|\uu_h(\cdot, t)\|^2 +
\|\BB_h(\cdot, t)\|^2 \lesssim
\|\uu_{h,0}\|^2 +
\|\BB_{h,0}\|^2 + 
\left(\int_0^t \left(\|\ff(\cdot, s)\| + \|\GG(\cdot, s)\| \right) \, {\rm d}s \right)^2  
\lesssim \data^2 \,.
\]

\noindent
$\bullet$ Bound \eqref{eq:stab3f-2}.
Combining the second and the last row in \eqref{eq:bound3f-1} and \eqref{eq:stab3f-1} for a.e. $t \in I$ we obtain
\[
\partial_t \Vert \uu_h  \Vert^2 + 
\partial_t \Vert \BB_h \Vert^2 +
\|\uu_h\|_{(\uu_h, \BB_h)}^2 +
\normamt{\BB_h}^2 \lesssim 
(\|\ff\|  + \|\GG\|) \, 
\data \,.
\]
The proof follows by integrating the previous bound over $I$, the Young inequality and \eqref{eq:initial-bound}.
\end{proof}

\subsection{Error analysis}
\label{sub:err}

Let $(\uu_h, \BB_h)$ be the solution of Problem \eqref{eq:stab fem-3f}, then for a.e. $t \in I$, we introduce the following  shorthand notation  (for all sufficiently regular $\vv:\Omega\times [0,T] \rightarrow {\mathbb R}^3$):
\begin{equation}
\label{eq:normastab3}
\normastab{\vv}(t) := \Vert \vv(\cdot, t) \|_{(\uu_h(\cdot, t), \BB_h(\cdot, t))} \,.
\end{equation} 

Let $(\uu, \BB)$ and $(\uu_h, \BB_h)$ be the solutions of Problem \eqref{eq:variazionale} and Problem \eqref{eq:stab fem-3f}, respectively.
Then referring to Lemma \ref{lm:int-v} and Lemma \ref{lm:int-i}, let us define the following error functions
\begin{equation}
\label{eq:fquantities}
\eei := \uu - \uui\,, \qquad 
\eeh := \uu_h - \uui\,, \qquad 
\EEi := \BB - \BBi\,, \qquad 
\EEh := \BB_h - \BBi \,.
\end{equation}
Notice  that from Lemma \ref{lm:int-v} $\uui \in \ZZd$, thus $\eeh \in \ZZd$ for a.e. $t \in I$.

We now state the final regularity assumptions required for the theoretical analysis.

\smallskip
\noindent
\textbf{(RA2-3f) Regularity assumptions on the exact solution (error analysis):} 
\begin{itemize}
\item[(i)] $\uu \in L^2(0, T; \WW^{k+1}_{\infty}(\Omega_h)) \cap H^1(0, T; \HH^{k+1}(\Omega_h))$\,, 
$\BB \in H^1(0, T; \HH^{k+1}(\Omega_h))$\,,
\item[(ii)] $\uu, \BB \in L^\infty(0, T; \WW^1_\infty(\Omega_h))$\,.
\end{itemize}
In order to shorten some equations in the following
we set
\begin{equation}
\label{eq:gunoinf}
\gunoinf := 1 + \Vert \uu \Vert_{L^\infty(0, T; \WW^1_\infty(\Omega_h))} +
\Vert \BB \Vert_{L^\infty(0, T; \WW^1_\infty(\Omega_h))} \,.
\end{equation}
We also introduce the following useful quantities for the error analysis
\begin{align}
\label{eq:quantities}
\ls^2 &:= \max \left\{ 
\ns(1 + \bdma + \bdma^{-1})   \,,
h (\data^2 + 1) 
\right\} \,,
\\
\label{eq:quantities3f}
\lmt^2 &:=  (\nm + 1)   \,.
\end{align}
%
The maximum in the quantity $\ls^2$ above is associated to the usual comparison among diffusion and convection.

\begin{proposition}[Interpolation error estimate for the velocity field]
Let Assumption \cfan{(MA1)} hold.
Then, under the regularity assumption  \cfan{(RA2-3f)}, referring to \eqref{eq:fquantities} and \eqref{eq:quantities} for a.e. $t \in I$ the following holds:
\label{prp:interpolation}
\[
\normastab{\eei}^2 \lesssim  
\ls^2 \, \Vert \uu \Vert_{\WW^{k+1}_\infty(\Omega_h)}^2 \, h^{2k}  \,.
\]
\end{proposition}

\begin{proof}
We estimate each term in the definition of $\normastab{\eei}$ (cf. \eqref{eq:normastab3} and \eqref{eq:norme}).
Employing bounds \eqref{eq:int-v} and \eqref{eq:utilejump2} the term 
$\normaunoh{\eei}$ can be bounded as follows:
\begin{equation}
\label{eq:eis}
\ns \normaunoh{\eei}^2 \lesssim
\ns (1 + \bdma) h^{2k} \vert \uu \vert_{k+1, \Omega_h}^2\,.
\end{equation}
Concerning $\normaupw{\eei}$ we infer
\begin{equation}
\label{eq:eiupw}
\begin{aligned}
&\normaupw{\eei}^2
= 
\sum_{f \in \EdgesI}  (\vert \uu_h \cdot \nn_f \vert \, \jump{\eei}_f \,,  \jump{\eei}_f)_{f}
\\
& \lesssim
\sum_{f \in \EdgesI}  \Vert |\jump{\eei}_f | \, |\uu_h \cdot \nn_f| \Vert_{f}^2 +
\sum_{f \in \EdgesI}  \Vert \jump{\eei}_f \Vert_{f}^2
&  &\text{(by Young ineq.)}
\\
& \lesssim
\sum_{E \in \Omega_h}  h_E^{-1} (
\Vert \uu_h  \Vert_{\LL^\infty(E)}^2 +  1)
\Vert \eei  \Vert_E^2 
&  &\text{(by \eqref{eq:utile-jump} \& \eqref{eq:utilejump2})}
\\
& \lesssim
\sum_{E \in \Omega_h}  h_E^{2k+1} (
\Vert \uu_h \Vert_{\LL^\infty(E)}^2  + 1)  
\vert \uu \vert_{k+1,E}^2
&  &\text{(by \eqref{eq:int-v})}
\\
& \lesssim
h^{2k+1} \left(
{ \vert E\vert}\Vert \uu \Vert_{\WW^{k+1}_\infty(\Omega_h)}^2
{ h_E^{-3}} \Vert \uu_h \Vert^2  +
\vert \uu \vert_{k+1,\Omega_h}^2 \right) 
&  &\text{(H\"{o}lder ineq. \& by Lm. \ref{lm:inverse})}
\\
& \lesssim
h^{2k+1} \left(
\Vert \uu \Vert_{\WW^{k+1}_\infty(\Omega_h)}^2
\data^2 +
\vert \uu \vert_{k+1,\Omega_h}^2 \right) \, .
&  &\text{(by \eqref{eq:stab3f-1})}
\end{aligned}
\end{equation}
Using analogous computations for the term $\normacip{\eei}$ we get
\begin{equation}
\label{eq:eicip}
\begin{aligned}
&\normacip{\eei}^2 = 
 \sum_{f \in \EdgesI} 
\max \{\Vert \BB_h\Vert_{\LL^\infty(\omega_f)}^2 , 1 \}
\left( 
\Vert \jump{\eei}_f \Vert_f^2 + 
h_f^2  \Vert \jump{\bnabla_h \eei}_f \Vert_f^2
\right)
\\
& \begin{aligned}
& \lesssim
\sum_{E \in \Omega_h} h_E^{2k+1}
(\Vert \BB_h\Vert_{\LL^\infty(E)}^2+ 1) \,
\vert \uu \vert_{k+1,E}^2 
& \qquad &\text{(by \eqref{eq:utilejump2})}
\\
& \lesssim
h^{2k+1} 
\left( \vert E\vert \Vert \uu \Vert_{\WW^{k+1}_\infty(\Omega_h)}^2  h_E^{-3} \Vert \BB_h\Vert^2 
+ \vert \uu \vert_{k+1,\Omega_h}^2
\right)
& \qquad  &\text{(H\"{o}lder ineq. \& by Lm. \ref{lm:inverse})}
\\
& \lesssim
h^{2k+1} \left(
\Vert \uu \Vert_{\WW^{k+1}_\infty(\Omega_h)}^2 
\data^2 + \vert \uu \vert_{k+1,\Omega_h}^2 \right) \, .
& \qquad &\text{(by \eqref{eq:stab3f-1})}
\end{aligned}
\end{aligned}
\end{equation}
The thesis follows combining the three bounds above.
\end{proof}

\begin{proposition}[Interpolation error estimate for the magnetic field]
\label{prp:interpolationmt}
Let Assumption \cfan{(MA1)} hold.
Furthermore, if $k=1$ let also Assumption \cfan{(MA2)} hold. Then, under the regularity assumption  \cfan{(RA2-3f)}, referring to \eqref{eq:fquantities} and \eqref{eq:quantities3f} for a.e. $t \in I$ the following holds:
\[
 \normamt{\EEi}^2 \lesssim 
\lmt^2 \, \vert \BB \vert_{k+1, \Omega_h}^2 \, h^{2k}  \,.
\]
\end{proposition}

\begin{proof}
The proof is a direct consequence \eqref{eq:err-int}:
\[
\normamt{\EEi}^2  \lesssim
(\nm + 1) \Vert \bnabla \EEi\Vert^2 \lesssim
(\nm + 1) h^{2k} \vert \BB \vert_{k+1, \Omega_h}^2 \,.
\]
\end{proof}

We now prove the following error estimation.
\begin{proposition}[Discretization error]
\label{prp:error equation}
Let Assumption \cfan{(MA1)} hold. Furthermore, if $k=1$ let also Assumption \cfan{(MA2)} hold.
Let $\ccoe$ be the uniform constant in Proposition \ref{prp:coe}.
Then, under the consistency assumption \cfan{(RA1-3f)} and assuming that the parameter $\bdma$ (cf. \eqref{eq:forme_d}) is sufficiently large, referring to \eqref{eq:fquantities}, for a.e. $t \in I$ the following holds
\begin{equation}
\label{eq:error equation}
\frac{1}{2}\partial_t \Vert \eeh  \Vert^2 + 
\frac{1}{2}\partial_t \Vert \EEh \Vert^2 +
\ccoe \left(\normastab{\eeh}^2 +
\normamt{\EEh}^2 \right) \leq 
\sum_{i=1}^4 T_i + T_{5, 3\rm{f}} 
\end{equation}
where
\begin{equation}
\label{eq:Ti-3f}
\begin{aligned}
&T_1 := 
\left(\partial_t \eei  \,, \eeh \right) + 
\left(\partial_t \EEi  \,, \EEh \right) +
\ns \ash(\eei, \eeh) + \nm \am(\EEi, \EEh) +
J_h(\BB_h; \eei, \eeh),
\\
&\begin{aligned}
T_2 &:= c(\uu; \uu; \eeh) - c_h(\uu_h; \uui, \eeh) \,,
&\qquad 
T_3 &:= d(\BB_h; \BBi, \eeh) - d(\BB; \BB, \eeh) \,,
\\
T_4 &:= d(\BB; \EEh, \uu) - d(\BB_h; \EEh, \uui) \,,
&\qquad 
T_{5, 3\rm{f}} &:= (\bdiver \EEi, \bdiver \EEh) \,.
\end{aligned}
\end{aligned}
\end{equation}
\end{proposition}

\begin{proof}
Direct calculations yield
\[
\begin{aligned}
&\frac{1}{2}\partial_t \Vert \eeh  \Vert^2 + 
\frac{1}{2}\partial_t \Vert \EEh \Vert^2 +
\ccoe \left(\normastab{\eeh}^2 +
\normamt{ \EEh}^2 \right) \leq 
\\
&\leq
\left(\partial_t \eeh  , \eeh \right) + 
\left(\partial_t \EEh  , \EEh \right) +
\astabt((\uu_h, \BB_h); (\eeh, \EEh), (\eeh, \EEh))
&  &\text{(Prp. \ref{prp:coe})}
\\
&=
(\ff, \eeh) + (\GG, \EEh) -
\left(\partial_t \uui  , \eeh \right) -
\left(\partial_t \BBi  , \EEh \right) +
\\
& \quad -
\astabt((\uu_h, \BB_h); (\uui, \BBi), (\eeh, \EEh))
&  &\text{(by \eqref{eq:stab fem-3f})}
\\
&=
\left(\partial_t \eei  , \eeh \right) +
\left(\partial_t \EEi  , \EEh \right) +
\ns \as(\uu, \eeh) + \nm \am(\BB, \EEh) + c(\uu; \uu, \eeh)+
\\
& \quad 
- d(\BB; \BB, \eeh) + d(\BB; \EEh, \uu)
-\astabt((\uu_h, \BB_h); (\uui, \BBi), (\eeh, \EEh)) \, .
&  &\text{(by \eqref{eq:variazionale})}
\end{aligned}
\]
The proof follows recalling definition \eqref{eq:astabt} and the consistency equations \eqref{eq:forme_con} and \eqref{eq:forme_con3f}.
\end{proof}

In the following $\theta$ (cf. Lemmas \ref{lm:t1}--\ref{lm:t5}) is a suitable uniformly bounded parameter that will be specified later.

\begin{lemma}[Estimate of $T_1$]
\label{lm:t1}
Under the assumptions of Proposition \ref{prp:error equation} and  the regularity assumption \cfan{(RA2-3f)}, the following holds
\begin{equation}
\label{eq:T1}
\begin{aligned}
T_1 &\leq 
\gunoinf \left( \Vert \eeh \Vert^2 + \Vert \EEh \Vert^2 \right) +
\theta \left( \ns\normaunoh{\eeh}^2 + \frac{1}{2}\normacip{\eeh}^2 + 
\nm\Vert \bnabla \EEh \Vert^2  \right) + \\
& \quad + \frac{C}{\theta} \, (\ls^2 \, \Vert \uu \Vert_{\WW^{k+1}_\infty(\Omega_h)}^2 + \nm \, |\BB|_{k+1, \Omega_h}^2) h^{2k} +
C (|\uu_t|_{k+1,\Omega_h}^2 + |\BB_t|_{k+1,\Omega_h}^2) h^{2k+2} \,.
\end{aligned}
\end{equation}
\end{lemma}

\begin{proof}
We estimate separately each term in $T_1$.
Combining the Young inequality with \eqref{eq:int-v} and \eqref{eq:err-int} we infer 
\[
\left(\partial_t \eei  \,, \eeh \right) + 
\left(\partial_t \EEi  \,, \EEh \right)
\leq 
\Vert \eeh \Vert^2 + \Vert \EEh \Vert^2 +
C \, h^{2k+2} (|\uu_t|_{k+1,\Omega_h}^2 + |\BB_t|_{k+1,\Omega_h}^2)\,.
\]
Using the same computations in the proof of Proposition 5.8 in \cite{MHD-linear} (cf. term $\alpha_1$) we obtain 
\[
\ns \ash(\eei, \eeh)
\leq \theta \ns \, \normaunoh{\eeh}^2 + \frac{C}{\theta} \, \ns (1 + \bdma + \bdma^{-1}) h^{2k} |\uu|_{k+1, \Omega_h}^2 \,.
\]
The third term in $T_1$ can be bounded employing again 
the Young inequality  and \eqref{eq:err-int}
\[
\nm \am(\EEi, \EEh)
\leq \theta  \nm \, \Vert \bnabla \EEh\Vert^2 + \frac{C}{\theta} \, \nm h^{2k} |\BB|_{k+1, \Omega_h}^2 \,.
\]
Finally, recalling the definition of norm $\normacip{\cdot}$, employing Cauchy-Schwarz inequality, the Young inequality, and bound \eqref{eq:eicip}, we have
\[
\begin{aligned}
J_h(\BB_h; \eei, \eeh) &\leq \normacip{\eei} \, \normacip{\eeh}
\\
&\leq \frac{\theta}{2} \normacip{\eeh}^2 + \frac{C}{\theta} \, h^{2k+1} 
(\Vert \uu \Vert_{\WW^{k+1}_\infty(\Omega_h)}^2 \data^2 + \vert \uu \vert_{\HH^{k+1}(\Omega_h)}^2)   \,.
\end{aligned}
\]
\end{proof}

\begin{lemma}[Estimate of $T_2$]
\label{lm:t2}
Under the assumptions of Proposition \ref{prp:error equation} and  the regularity assumption \cfan{(RA2-3f)}, the following holds
\begin{equation}
\label{eq:T2}
T_2 \leq 
C  \gunoinf \, \Vert \eeh \Vert^2 +\theta  \normaupw{\eeh}^2  + \frac{C}{\theta} \ls^2 \,\Vert \uu \Vert_{\WW^{k+1}_\infty(\Omega_h)}^2  \, h^{2k} +  C  \gunoinf \vert \uu \vert_{k+1,\Omega_h}^2  \, h^{2k+2} \,.
\end{equation}
\end{lemma}

\begin{proof}
Recalling the definition of $c_h(\cdot; \cdot, \cdot)$, direct computations yield
\[
\begin{aligned}
T_2 & = c(\uu; \uu; \eeh) - c_h(\uu_h; \uui, \eeh)
\\ 
& = c(\uu; \uu; \eeh) - c_h(\uu_h; \uu, \eeh) + c_h(\uu_h; \eei, \eeh)
\\
& = 
(( \bnabla \uu ) \, (\uu - \uu_h), \, \eeh ) +
(( \bnabla_h \eei ) \, \uu_h, \, \eeh ) + 
\\
& \quad  
-\sum_{f \in \EdgesI} ( (\uu_h\cdot \nn_f) \jump{\eei}_f ,\, \media{\eeh}_f)_f +
\sum_{f \in \EdgesI} (\vert \uu_h \cdot \nn_f \vert \jump{\eei}_f, \, \jump{\eeh}_f )_f \,.
\end{aligned}
\]
Integrating by parts and recalling that $\diver \uu_h= 0$ and
$\jump{\boldsymbol{a} \cdot \boldsymbol{b}}_f = 
\jump{\boldsymbol{a}}_f \cdot \media{\boldsymbol{b}}_f 
+ 
\media{\boldsymbol{a}}_f \cdot \jump{\boldsymbol{b}}_f$
for any $f \in \EdgesI$, 
we obtain
\begin{equation}
\label{eq:T20}
\begin{aligned}
T_2
& = 
(( \bnabla \uu ) \, (\eei - \eeh), \, \eeh )  
 - (( \bnabla_h \eeh ) \, \uu_h, \, \eei) +
\sum_{f \in \EdgesI} ( (\uu_h\cdot \nn_f) , \, \jump{\eeh \cdot \eei}_f)_f + 
\\
& \quad  
-\sum_{f \in \EdgesI} ( (\uu_h\cdot \nn_f) \jump{\eei}_f ,\, \media{\eeh}_f)_f +
\sum_{f \in \EdgesI} (\vert \uu_h \cdot \nn_f \vert \jump{\eei}_f, \, \jump{\eeh}_f )_f
\\
& = 
\biggl( (( \bnabla \uu ) \, (\eei - \eeh), \, \eeh )  
+ (( \bnabla_h \eeh ) \, (\eei - \eeh), \, \eei) \biggr)
- (( \bnabla_h \eeh ) \, \uu, \, \eei) +
\\
& \quad  
+ \sum_{f \in \EdgesI} \left( ( (\uu_h\cdot \nn_f) \jump{\eeh}_f ,\, \media{\eei}_f)_f +
( |\uu_h\cdot \nn_f| \jump{\eei}_f ,\,  \jump{\eeh}_f)_f 
\right)
\\
&=: 
T_{2,1} + T_{2,2} + T_{2,3} \,.  
\end{aligned}
\end{equation}
We now estimate each term in the sum above.
The term $T_{2,1}$ can be bounded as follows 
\[
\begin{aligned}
T_{2,1} &\leq 
\left(\Vert \eei \Vert + \Vert \eeh \Vert \right) 
\left( \Vert \uu \Vert_{\WW^1_\infty(\Omega_h)} \, 
 \Vert  \eeh \Vert  +  
 \Vert |\bnabla_h \eeh| \, \eei \Vert
 \right)
 & \quad &\text{(Cau.-Sch. ineq.)}
 \\
&\leq
C \left(\Vert \eei \Vert + \Vert \eeh \Vert \right) 
 \Vert \uu \Vert_{\WW^1_\infty(\Omega_h)} \, 
 \Vert \eeh \Vert 
 & \quad &\text{(by \eqref{eq:support})}
\\
& \leq
C \Vert \uu \Vert_{\WW^1_\infty(\Omega_h)} \, \Vert \eeh \Vert^2 +
C h^{2k+2}\Vert \uu \Vert_{\WW^1_\infty(\Omega_h)} \, |\uu|_{k+1, \Omega_h}^2 \, .
& \quad &\text{(Young ineq. \& \eqref{eq:int-v})} 
\end{aligned} 
\]
Using the orthogonality \eqref{eq:orth-v}, for the term $T_{2,2}$ we infer
\[
\begin{aligned}
T_{2,2} & = (( \bnabla_h \eeh ) \, (\Pi_0 \uu - \uu), \, \eei)
\leq  
\Vert (\bnabla \eeh )\, 
 (\Pi_0 \uu - \uu) \Vert \,
\Vert \eei \Vert
& \quad &\text{(Cau.-Sch. ineq.)}
\\
& \leq C
\Vert \eeh \Vert \, 
\Vert \uu \Vert_{\WW^1_\infty(\Omega_h)} \,
\Vert \eei \Vert
& \quad &\text{(by \eqref{eq:support})}
\\
& \leq
C \Vert \uu \Vert_{\WW^1_\infty(\Omega_h)} \, \Vert \eeh \Vert^2 +
C h^{2k+2}\Vert \uu \Vert_{\WW^1_\infty(\Omega_h)} \, |\uu|_{k+1, \Omega_h}^2 \, .  
& \quad &\text{(Young ineq. \& \eqref{eq:int-v})}
\end{aligned}
\]
Finally, employing the Young inequality and analogous computations to those in \eqref{eq:eiupw}, the term $T_{2,3}$ is bounded as follows
\[
\begin{aligned}
T_{2,3} &\leq 
\theta \normaupw{\eeh}^2 + 
\frac{C}{\theta} (\Vert |\uu_h \cdot \nn_f| \media{\eei}\Vert_f^2 + \normaupw{\eei}^2)
\\
&\leq 
\theta  \normaupw{\eeh}^2 + \frac{C}{\theta} \, 
h^{2k+1} \left(
\Vert \uu \Vert_{\WW^{k+1}_\infty(\Omega_h)}^2
\data^2 +
\vert \uu \vert_{k+1,\Omega_h}^2 \right)  \,.
\end{aligned}
\]
The proof follows combining in \eqref{eq:T20} the three bounds above.
\end{proof}

\begin{lemma}[Estimate of $T_3$]
\label{lm:t3}
Under the assumptions of Proposition \ref{prp:error equation} and  the regularity assumption \cfan{(RA2-3f)}, the following holds
\begin{equation}
\label{eq:T3}
\begin{aligned}
T_3 &\leq 
C \left(\gunoinf  +  \frac{\theta}{4} h\right)  
( \Vert \EEh \Vert^2 + \Vert \eeh \Vert^2) +
\frac{\theta}{2}  \normacip{\eeh}^2 + 
\\
& \quad +
\frac{C}{\theta}   \gunoinf^2\vert \BB \vert_{k+1,\Omega_h}^2  h^{2k+1} + 
C  \gunoinf \vert \BB \vert_{k+1,\Omega_h}^2  h^{2k+2}  \,.
\end{aligned}
\end{equation}
\end{lemma}

\begin{proof}
A vector calculus identity and an integration by parts yield
\begin{equation}
\label{eq:T30}
\begin{aligned}
T_3 &= d(\BB_h; \BBi, \eeh) - d(\BB; \BB, \eeh)
= 
d(\BB_h - \BB; \BB, \eeh) - d(\BB_h; \EEi, \eeh)
\\
& =
d(\EEh - \EEi; \BB, \eeh) - (\bcurl(\EEi) \times \BB_h, \eeh)
\\
& =
d(\EEh - \EEi; \BB, \eeh) + (\bcurl(\EEi), \eeh \times \BB_h)
\\
& =
d(\EEh - \EEi; \BB, \eeh) + 
\sum_{f \in \EdgesI}(\jump{\eeh}_f \times \BB_h, \EEi \times \nn_f)_f + 
(\EEi, \bcurl_h(\eeh \times \BB_h))
\\
& =: T_{3,1} + T_{3,2} + T_{3,3} \,.
\end{aligned}
\end{equation}
For $T_{3,1}$ applying the Cauchy-Schwarz inequality, the Young inequality and \eqref{eq:err-int}, we infer
\begin{equation}
\label{eq:T31}
\begin{aligned}
T_{3,1} & \leq 
\Vert \BB \Vert_{\WW^1_\infty(\Omega_h)}
(\Vert \EEh \Vert + \Vert \EEi \Vert) \Vert \eeh \Vert
\\
& \leq 
C \Vert \BB \Vert_{\WW^1_\infty(\Omega_h)} 
(\Vert \EEh \Vert^2 + \Vert \eeh \Vert^2) +  
C h^{2k+2} \Vert \BB \Vert_{\WW^1_\infty(\Omega_h)} \, \vert \BB \vert_{k+1,\Omega_h}^2 \,.
\end{aligned}
\end{equation}
Concerning $T_{3,2}$, recalling the definition of $\normacip{\cdot}$, from the Young inequality and \eqref{eq:utilejump2-h} we get 
\begin{equation}
\label{eq:T32}
\begin{aligned}
T_{3,2} & \leq 
\frac{\theta}{4} \sum_{f \in \EdgesI} \Vert \jump{\eeh}_f \times \BB_h \Vert_f^2  +
\frac{C}{\theta} \sum_{f \in \EdgesI} \Vert \EEi \Vert_f^2
\leq
\frac{\theta}{4} \normacip{\eeh}^2 + \frac{C}{\theta}  \, h^{2k+1} \, \vert \BB\vert_{k+1, \Omega_h}^2 \,.
\end{aligned}
\end{equation}
We now analyse the term $T_{3,3}$. Direct computations yield
\begin{equation}
\label{eq:T33}
\begin{aligned}
T_{3,3} & =  
\left(\EEi, \bcurl_h(\eeh \times (\BB_h - \BB))\right) +
\left(\EEi, \bcurl_h(\eeh \times \BB)\right)
\\
& =
\left(\EEi, \bcurl_h(\eeh \times \EEh) \right)  -
\left(\EEi, \bcurl_h(\eeh \times \EEi) \right) +
\\
& \quad + 
\left(\EEi, \bcurl_h(\eeh \times ((I - \Pi_0)\BB)) \right) +
(\EEi, \bcurl_h(\eeh \times (\Pi_0\BB))) 
\\
&=: 
\alpha_1 + \alpha_2 + \alpha_3 + \alpha_4\,.  
\end{aligned}
\end{equation}
For the term $\alpha_1$, from Cauchy-Schwarz inequality, we infer
\begin{equation}
\label{eq:T3alpha1}
\begin{aligned}
\alpha_1 &\leq 
\sum_{E \in \Omega_h}
\Vert  \EEi \Vert_E \left(\Vert \bnabla_h \eeh \Vert_{\LL^\infty(E)} \, 
\Vert \EEh \Vert_E +
\Vert  \eeh \Vert_{\LL^\infty(E)} 
\Vert  \bnabla_h \EEh \Vert_E \right)
\\
& \begin{aligned}
& \leq
\sum_{E \in \Omega_h} C
h_E \Vert  \bnabla \BB \Vert_E \,
h_E^{-5/2}\Vert \eeh \Vert_{E} \, 
\Vert \EEh \Vert_{E}
& \quad & \text{(by \eqref{eq:err-int}, Lm. \ref{lm:bramble} \&  \ref{lm:inverse})}
\\
& \leq
\sum_{E \in \Omega_h} C
\Vert   \BB \Vert_{\WW^1_\infty(E)} \,
\Vert \eeh \Vert_{E} \, 
\Vert \EEh \Vert_{E}
& \quad & \text{(H\"{o}lder ineq.)}
\\
&\leq 
C \Vert \BB \Vert_{\WW^1_\infty(\Omega_h)} 
(\Vert \eeh \Vert^2 +
 \Vert \EEh \Vert^2) \, .
& \quad & \text{(Young ineq.)}
\end{aligned}
\end{aligned}
\end{equation}
Using analogous computations we get
\begin{equation}
\label{eq:T3alpha2}
\begin{aligned}
\alpha_2 &\leq 
\sum_{E \in \Omega_h}
\Vert  \EEi \Vert_E \left(\Vert \bnabla_h \eeh \Vert_{\LL^\infty(E)} \, 
\Vert \EEi \Vert_E +
\Vert  \eeh \Vert_{\LL^\infty(E)} 
\Vert  \bnabla_h \EEi \Vert_E \right)
\\
& \begin{aligned}
& \leq
\sum_{E \in \Omega_h} C
h_E \Vert  \bnabla \BB \Vert_E \,
h_E^{k-3/2}\Vert \eeh \Vert_{E} \, 
\Vert \BB \Vert_{k+1,E}
&  & \text{(by \eqref{eq:err-int}, Lm. \ref{lm:bramble} \&  \ref{lm:inverse})}
\\
& \leq
\sum_{E \in \Omega_h} C
h_E^{k+1}
\Vert   \BB \Vert_{\WW^1_\infty(E)} \,
\Vert \eeh \Vert_{E} \, 
\Vert \BB \Vert_{k+1,E}
&  & \text{(H\"{o}lder ineq.)}
\\
&\leq 
C \Vert \BB \Vert_{\WW^1_\infty(\Omega_h)} 
 \Vert \eeh \Vert^2 +
C h^{2k+2}\Vert \BB \Vert_{\WW^1_\infty(\Omega_h)} \,
\vert \BB \vert_{k+1, \Omega_h}^2 \, .
&  & \text{(Young ineq.)}
\end{aligned}
\end{aligned}
\end{equation}
The term $\alpha_3$ can be bounded as follows
\begin{equation}
\label{eq:T3alpha3}
\begin{aligned}
\alpha_3 &\leq 
\Vert \EEi \Vert 
\left( \Vert |\bnabla_h \eeh | \, 
| (I - \Pi_0)\BB |\Vert +
\Vert \eeh \Vert \, 
\Vert (I - \Pi_0)\BB \Vert_{\WW^1_\infty(\Omega_h)}
\right)
&  & \text{(Cau.-Sch. ineq.)}
\\
&\leq 
C \Vert \EEi \Vert 
\Vert \eeh \Vert \, 
\Vert \BB \Vert_{\WW^1_\infty(\Omega_h)}
&  & \text{(by Rm. \ref{rm:support})}
\\
&\leq 
C \Vert \BB \Vert_{\WW^1_\infty(\Omega_h)} \,
\Vert \eeh \Vert^2 +
C h^{2k+2}\Vert \BB \Vert_{\WW^1_\infty(\Omega_h)} \,
\vert \BB \vert_{k+1, \Omega_h}^2 \, .
&  & \text{(Young ineq. \& \eqref{eq:err-int})}
\end{aligned}
\end{equation}
For the estimate of the term $\alpha_4$ we proceed as follows. 
Being $\eeh \in \ZZd$ and $\Pzerok{0}\BB_h$ constant on each element, 
the vector calculus identity 
\begin{equation}
\label{eq:utilecurl}
\bcurl(A \times B) = (\diver B) \, A - (\diver A) \, B + (\bnabla A) B - (\bnabla B) A
\end{equation}
yields
\[
\bcurl_h(\eeh \times (\Pzerok{0}\BB)) = (\bnabla_h \eeh) (\Pzerok{0}\BB) =: \pp_{k-1}  \in [\Pk_{k-1}(\Omega_h)]^3 \, .
\]
Therefore from \eqref{eq:int-orth}, the Cauchy-Schwarz inequality,
\eqref{eq:err-int} and Lemma \ref{lm:cip}, the Young inequality, and the continuity of the $L^2$-projection w.r.t. the $L^\infty$-norm  
we infer
\begin{equation}
\label{eq:T3alpha40}
\begin{aligned}
\alpha_4 & = 
( \EEi \,, \pp_{k-1}  )
=
( \EEi \,, (I - \intcip)\pp_{k-1}) 
\\
&\leq
\biggl(\sum_{E \in \Omega_h} h_E^{-1}\Vert \EEi \Vert_{E}^2 \biggr)^{1/2} 
\biggl(
\sum_{E \in \Omega_h} h_E\Vert (I - \intcip)\pp_{k-1}  \Vert_{E}^2 
\biggr)^{1/2}
\\
&\leq
\frac{C}{\theta}  \Vert \BB \Vert_{\WW^1_\infty(\Omega_h)}^2 h^{2k+1} \vert \BB \vert_{k+1,\Omega_h}^2 +  
\frac{\theta}{4} \frac{1}{\Vert \BB \Vert_{\WW^1_\infty(\Omega_h)}^2}  \sum_{f \in  \EdgesI} h_f^2
\Vert \jump{\pp_{k-1}}_f \Vert_{f}^2  \,.
\end{aligned}
\end{equation}
Furthermore employing a triangular inequality, bound \eqref{eq:utile-jump}, Lemma \ref{lm:bramble} and Lemma \ref{lm:inverse} we have
\[
\begin{aligned}
&\sum_{f \in  \EdgesI}  h_f^2 
\Vert \jump{\pp_{k-1}}_f \Vert_{f}^2
\leq
\sum_{f \in \EdgesI} \left(
h_f^2  \Vert \jump{\bnabla_h \eeh}_f \BB \Vert_{f}^2 
+  
h_f^2
\Vert \jump{(\bnabla_h \eeh) ((I - \Pzerok{0})\BB)}_f\Vert_{f}^2 
\right)
\\
& \leq
\sum_{f \in \EdgesI} 
h_f^2  \Vert \BB \Vert_{\LL^\infty(\omega_f)}^2 \Vert \jump{\bnabla_h \eeh}_f  \Vert_{f}^2 + 
C\sum_{E \in \Omega_h} 
h_E \Vert (I - \Pzerok{0}) \BB \Vert_{\LL^\infty(E)}^2  
\Vert \bnabla_h \eeh \Vert_E^2 
\\
& \leq
\sum_{f \in \EdgesI} 
h_f^2  \Vert \BB \Vert_{\LL^\infty(\omega_f)}^2 \Vert \jump{\bnabla_h \eeh}_f  \Vert_{f}^2 +
C\sum_{E \in \Omega_h} 
h_E \Vert \BB \Vert_{\WW^1_\infty(E)}^2  
\Vert  \eeh \Vert_E^2 
\\
& \leq
\Vert \BB \Vert_{\WW^1_\infty(\Omega_h)}^2 \normacip{\eeh}^2 + 
C\Vert \BB \Vert_{\WW^1_\infty(\Omega_h)}^2 h \, \Vert  \eeh \Vert^2 \,.
\end{aligned}
\]
Therefore from \eqref{eq:T3alpha40} we infer
\begin{equation}
\label{eq:T3alpha4}
\alpha_4 \leq 
\frac{C}{\theta}  \Vert \BB \Vert_{\WW^1_\infty(\Omega_h)}^2 h^{2k+1} \vert \BB\vert^2_{k+1, \Omega_h} +
\frac{\theta}{4}\normacip{\eeh}^2 + \frac{C \theta}{4} h \, \Vert  \eeh \Vert^2 \,.
\end{equation}
Inserting \eqref{eq:T3alpha1}--\eqref{eq:T3alpha4} in \eqref{eq:T33} we obtain
\begin{equation}
\label{eq:T33F}
\begin{aligned}
T_{3,3} &\leq  
C \left( \Vert \BB \Vert_{\WW^1_\infty(\Omega_h)} +  \frac{\theta}{4} h\right)  
( \Vert \EEh \Vert^2 + \Vert \eeh \Vert^2) +
\frac{\theta}{4}  \normacip{\eeh}^2 + 
\\
& 
\quad + \frac{C}{\theta}  \Vert \BB \Vert_{\WW^1_\infty(\Omega_h)}^2 h^{2k+1} \vert \BB \vert_{k+1,\Omega_h}^2
+
C \Vert \BB \Vert_{\WW^1_\infty(\Omega_h)} h^{2k+2} \vert \BB \vert_{k+1,\Omega_h}^2
\,.
\end{aligned}
\end{equation}
The thesis now follows combining in \eqref{eq:T30}, the bounds \eqref{eq:T31}, \eqref{eq:T32} and \eqref{eq:T33F}.
\end{proof}

\begin{lemma}[Estimate of $T_4$]
\label{lm:t4}
Under the assumptions of Proposition \ref{prp:error equation} and  the regularity assumption \cfan{(RA2-3f)}, the following holds
\begin{equation}
\label{eq:T4}
\begin{aligned}
T_4 \leq 
C  \gunoinf   \left(1 + \frac{1}{\theta} \right) \Vert \EEh\Vert^2 + 
\frac{\theta}{2}  \Vert \diver \EEh \Vert^2 +
C\gunoinf \vert \BB \vert_{k+1,\Omega_h}^2 \, h^{2k} +
C \gunoinf \vert \uu \vert_{k+1,\Omega_h}^2 h^{2k+2} \,.
\end{aligned}
\end{equation}
\end{lemma}

\begin{proof}
Direct computations yield
\[
\begin{aligned}
T_4 & =
d(\BB; \EEh, \uu) - d(\BB_h; \EEh, \uui) =
d(\BB_h; \EEh, \eei) + d(\BB - \BB_h; \EEh, \uu) 
\\
& = 
d(\EEh - \EEi; \EEh, \eei) +
d(\BB; \EEh, \eei)  + 
d(\EEi - \EEh; \EEh, \uu) \,.
\end{aligned}
\]
We now manipulate the last term in the sum above, recalling that both $\uu$ and $\BB$ are solenoidal and that, for the same reason the form $(\cdot, (\bnabla \cdot) \uu)$ is skew-symmetric,  we obtain
\[
\begin{aligned}
&d(\EEi - \EEh; \EEh, \uu)  
=(\bcurl(\EEh), (\EEi - \EEh) \times \uu) 
\\
&=(\EEh, \bcurl((\EEi - \EEh) \times \uu))
&  & \text{(int. by parts)}
\\
&=(\EEh, (\diver \uu) (\EEi - \EEh)) -
(\EEh, (\diver (\EEi - \EEh)) \uu) +
\\
& \quad  +
(\EEh, (\bnabla (\EEi - \EEh)) \uu) -
(\EEh, (\bnabla \uu) (\EEi - \EEh))
&  & \text{(by \eqref{eq:utilecurl})}
\\
&= 
(\EEh, (\diver (\EEh - \EEi)) \uu) +
(\EEh, (\bnabla \EEi) \uu) - 
(\EEh, (\bnabla \uu) (\EEi - \EEh))
&  & \text{($\diver \uu=0$, skew-sym.)}
\\
&= 
(\EEh, (\diver (\EEh - \EEi)) \uu) -
(\EEi, (\bnabla \EEh) \uu) -
(\EEh, (\bnabla \uu) (\EEi - \EEh)) \, .
&  & \text{(skew-sym.)}
\end{aligned}
\]
Therefore from the previous equivalences
\begin{equation}
\label{eq:T40}
\begin{aligned}
T_4 &= \Big[ d(\EEh - \EEi; \EEh, \eei) - (\EEh, (\bnabla \uu) (\EEi - \EEh)) \Big] + d(\BB; \EEh, \eei) +
\\
& \quad - 
(\EEi, (\bnabla \EEh) \uu) + 
(\EEh, (\diver (\EEh - \EEi)) \uu)
\\
& =: T_{4,1} + T_{4,2} + T_{4,3} + T_{4,4} \,. 
\end{aligned}
\end{equation}
The term $T_{4,1}$ can be bounded as follows
\begin{equation}
\label{eq:T41}
\begin{aligned}
T_{4,1} &\leq 
\left(\Vert  |\bnabla \EEh | \, |\eei |  \Vert + 
\Vert \EEh \Vert \, \Vert \uu \Vert_{\WW^1_\infty(\Omega_h)} \right)
\left(\Vert \EEh \Vert + \Vert \EEi \Vert \right)
& \quad & \text{(Cau.-Sch. ineq.)}
\\
 &\leq 
C \Vert \EEh \Vert \, \Vert \uu \Vert_{\WW^1_\infty(\Omega_h)} 
\left(\Vert \EEh \Vert + \Vert \EEi \Vert \right)
& \quad & \text{(by \eqref{eq:support})}
\\
 &\leq 
C \Vert \uu \Vert_{\WW^1_\infty(\Omega_h)} \, \Vert \EEh \Vert^2 + 
C h^{2k+2}\Vert \uu \Vert_{\WW^1_\infty(\Omega_h)} \, \vert \BB \vert_{k+1, \Omega_h}^2 \, .
& \quad & \text{(Young ineq. \& \eqref{eq:err-int})}
\end{aligned}
\end{equation}
Concerning the term $T_{4,2}$ we infer
\begin{equation}
\label{eq:T42}
\begin{aligned}
T_{4,2} &= 
(\bcurl(\EEh) \times \BB, \eei) =
(\bcurl(\EEh) \times (\BB - \Pi_0 \BB), \eei)
& \quad & \text{(by \eqref{eq:orth-v})}
\\
& \leq 
\Vert |\bnabla \EEh| \, |\BB - \Pi_0 \BB| \Vert \,
\Vert \eei \Vert
& \quad & \text{(Cau.-Sch. ineq.)}
\\
& \leq C
\Vert \EEh \Vert \,
\Vert \BB \Vert_{\WW^1_\infty(\Omega_h)} \,
\Vert \eei \Vert
& \quad & \text{(by \eqref{eq:support})}
\\
 &\leq 
C \Vert \BB \Vert_{\WW^1_\infty(\Omega_h)} \, \Vert \EEh \Vert^2 + 
C h^{2k+2}\Vert \BB \Vert_{\WW^1_\infty(\Omega_h)} \, \vert \uu \vert_{k+1, \Omega_h}^2 \, .
& \quad & \text{(Young ineq. \& \eqref{eq:int-v})}
\end{aligned}
\end{equation}
For the term $T_{4,3}$ we proceed as follows
\begin{equation}
\label{eq:T43}
\begin{aligned}
T_{4,3} &\leq 
\Vert \uu\Vert_{\LL^\infty(\Omega)}
\sum_{E \in \Omega_h} 
\Vert \EEi\Vert_E \, 
\Vert \bnabla \EEh\Vert_E 
& \quad & \text{(Cau.-Sch. ineq.)}
\\
&\leq C 
\Vert \uu\Vert_{\LL^\infty(\Omega)}
\sum_{E \in \Omega_h} 
h_E^k\vert \BB\vert_{k+1,E} \, 
\Vert \EEh\Vert_E 
& \quad & \text{(by \eqref{eq:err-int} \& Lm. \ref{lm:inverse})}
\\
 &\leq 
C \Vert \uu \Vert_{\LL^\infty(\Omega)} \, \Vert \EEh \Vert^2 + 
C h^{2k}\Vert \uu \Vert_{\LL^\infty(\Omega)} \, \vert \BB \vert_{k+1, \Omega_h}^2  \, .
& \quad & \text{(Young ineq.)}
\end{aligned}
\end{equation}
Finally for the term $T_{4,4}$, applying  the triangular inequality, 
Young inequality and approximation estimates we infer 
\begin{equation}
\label{eq:T44}
\begin{aligned}
T_{4,4} &\leq 
\Vert \EEh\Vert \, \Vert \diver (\EEh - \EEi) \Vert \, \Vert \uu \Vert_{\LL^\infty}   \\
&\leq
C \left(1 + \frac{1}{\theta}\right) \Vert \uu \Vert_{\LL^\infty(\Omega)}^2 \, \Vert \EEh \Vert^2 + 
\frac{\theta}{2} \Vert \diver \EEh\Vert^2 
+  C h^{2k} \vert \BB \vert_{k+1,\Omega_h}^2 \,.
\end{aligned}
\end{equation}
The proof now follows by \eqref{eq:T40}--\eqref{eq:T44}.
\end{proof}

\begin{lemma}[Estimate of $T_{5, 3\rm{f}}$]
\label{lm:t5}
Under the assumptions of Proposition \ref{prp:error equation} and  the regularity assumption \cfan{(RA2-3f)}, the following holds
\begin{equation}
\label{eq:T5}
T_{5, 3\rm{f}} \leq 
\frac{\theta}{2}  \Vert \diver \EEh \Vert^2 + 
 \frac{C}{\theta} \,  \vert \BB\vert_{k+1, \Omega_h}^2 \, h^{2k}\,.
\end{equation}
\end{lemma}

\begin{proof}
The proof easily follows by the Cauchy-Schwarz inequality, the Young inequality and the interpolation estimate \eqref{eq:err-int}.
\end{proof}

Combining Proposition \ref{prp:error equation} with Lemmas \ref{lm:t1}--\ref{lm:t5}, we finally obtain the main error estimate for the three-field scheme \eqref{eq:fem-3f}.

\begin{proposition}[Error estimate]
\label{prp:conv}
Let Assumption \cfan{(MA1)} hold. Furthermore, if $k=1$ let also Assumption \cfan{(MA2)} hold.
Then, under the consistency assumption \cfan{(RA1-3f)} and the regularity assumption \cfan{(RA2-3f)}
and assuming that the parameter $\bdma$ (cf. \eqref{eq:forme_d}) is sufficiently large and $h \lesssim 1$,
referring \eqref{eq:quantities} and to \eqref{eq:quantities3f},  the following holds
\begin{equation}
\label{eq:conv-3f}
\begin{aligned}
&\Vert (\uu - \uu_h)(\cdot, T) \Vert^2 + \Vert (\BB_h - \BB_h)(\cdot, T)\Vert^2 +
\int_0^T \normastab{(\uu - \uu_h)(t)}^2 \, {\rm d}t +
\\
&+
\int_0^T \normamt{(\BB - \BB_h)(\cdot, t)}^2 \, {\rm d}t   \lesssim 
(\ls^2 \Vert \uu \Vert_{L^2(0,T; \WW^{k+1}_\infty(\Omega_h))}^2 + 
\lmt^2 \Vert \BB \Vert_{L^2(0,T; \HH^{k+1}(\Omega_h))}^2) h^{2k} +
\\
&+
( \Vert \uu \Vert_{H^1(0,T; \HH^{k+1}(\Omega_h))}^2 + 
 \Vert \BB \Vert_{H^1(0,T; \HH^{k+1}(\Omega_h))}^2) h^{2k+2}
 \,,
\end{aligned}
\end{equation}
where the hidden constant depends also on $\gunoinf$.
\end{proposition}

\begin{proof}
We start by noticing that from \eqref{eq:int-v}, \eqref{eq:err-int}, Proposition \ref{prp:interpolation} and Proposition \ref{prp:interpolationmt} we infer
\begin{equation}
\label{eq:conv-3f-ei}
\begin{aligned}
\Vert \eei(\cdot, T) \Vert^2 &+ \Vert \EEi(\cdot, T)\Vert^2 +
\int_0^T \normastab{\eei(\cdot, t)}^2 \, {\rm d}t  +
\int_0^T \normamt{\EEi(\cdot, t)}^2 \, {\rm d}t  
\\
&\lesssim 
h^{2k+2} \left(\vert \uu(\cdot, T) \vert_{k+1, \Omega_h}^2 + 
\vert \BB(\cdot, T)\vert_{k+1, \Omega_h}^2\right) +
\\
&+
(\ls^2 \Vert \uu \Vert_{L^2(0,T; \WW^{k+1}_\infty(\Omega_h))}^2 + 
\lmt^2 \Vert \BB \Vert_{L^2(0,T; \HH^{k+1}(\Omega_h))}^2) h^{2k} 
 \,.
\end{aligned}
\end{equation}
From Proposition \ref{prp:error equation} and Lemmas \ref{lm:t1}--\ref{lm:t5} considering in  \eqref{eq:T1},  \eqref{eq:T2},  \eqref{eq:T3},  \eqref{eq:T4} and  \eqref{eq:T5}
$\theta = \frac{\ccoe}{2}$  (observing that consequently $\theta^{-1}$ is uniformly bounded) we obtain 
\[
\begin{aligned}
\partial_t \Vert \eeh  \Vert^2 &+ 
\partial_t \Vert \EEh \Vert^2 +
\left(\normastab{\eeh}^2 +
\normamt{\BB_h}^2 \right) 
\lesssim  
\left(\gunoinf + \frac{\ccoe}{8}h\right)\left( \Vert \eeh \Vert^2 + \Vert \EEh \Vert^2 \right) + \\
& +
 \, (\ls^2 \, \Vert \uu \Vert_{\WW^{k+1}_\infty(\Omega_h)}^2 + \lmt^2 \, |\BB|_{k+1, \Omega_h}^2) h^{2k} + 
\gunoinf (1 + \gunoinf h) |\BB|_{k+1, \Omega_h}^2 \, h^{2k}+
\\
& +
\gunoinf  (
|\uu|_{k+1,\Omega_h}^2 + |\BB|_{k+1,\Omega_h}^2 + 
|\uu_t|_{k+1,\Omega_h}^2 + |\BB_t|_{k+1,\Omega_h}^2) h^{2k+2}
\end{aligned}
\]
with initial condition 
$\eeh(\cdot, 0) = \EEh(\cdot, 0) = 0$ (cf. \eqref{eq:ini cond-d}).
Therefore, employing the Gronwall lemma we finally have
\begin{equation}
\label{eq:conv-3f-eh}
\begin{aligned}
\Vert \eeh(\cdot, T) \Vert^2 &+ \Vert \EEh(\cdot, T)\Vert^2 +
\int_0^T \normastab{\eeh(\cdot, t)}^2 \, {\rm d}t  +
\int_0^T \normamt{\EEh(\cdot, t)}^2 \, {\rm d}t 
\\
&\lesssim 
(\ls^2 \Vert \uu \Vert_{L^2(0,T; \WW^{k+1}_\infty(\Omega_h))}^2 + 
\lmt^2 \Vert \BB \Vert_{L^2(0,T; \HH^{k+1}(\Omega_h))}^2) h^{2k} +
\\
&+
( \Vert \uu \Vert_{H^1(0,T; \HH^{k+1}(\Omega_h))}^2 + 
 \Vert \BB \Vert_{H^1(0,T; \HH^{k+1}(\Omega_h))}^2) h^{2k+2}
 \,,
\end{aligned}
\end{equation}
where the hidden constant depends also on $\gunoinf$.
The proof now follows by the triangular inequality.
\end{proof}

\begin{remark}\label{rem:4F-teaser}
The above result proves the quasi-robustness of the scheme and guarantees an $O(h^k)$ convergence rate for regular solutions. Furthermore, the pressure robustness of the method is reflected by the independence of the estimates from the pressure variable.
On the other hand, as anticipated in Remark \ref{rem:3Fto4F}, the above analysis is unable to deliver an $O(h^{k+1/2})$ pre-asymptotic error estimate whenever $\ns,\nm$ are small. The main reason is the term $T_4$, and more specifically $T_{4,3}$ and $T_{4,4}$. The term $T_{4,3}$ could be dealt with by adding a suitable stabilization term for the magnetic field, see $K_h$ in \eqref{eq:K}, and developing estimates as for \eqref{eq:T4-3-f} below. Term $T_{4,4}$ is more subtle, and directly related to the imposition of the divergence free condition for the field $\BB_h$. In the four-field scheme, we are able to deal with such term by strengthening such solenoidal condition at the discrete level with the introduction of a Lagrange multiplier.
\end{remark}

\section{Theoretical analysis of the four-field scheme}
\label{sec:theo-4f}

We preliminary make the following assumption on the velocity solution $\uu$ and the magnetic solution $\BB$ of problem \eqref{eq:variazionale}.

\smallskip
\noindent
\textbf{(RA1-4f) Regularity assumption for the consistency:}

\noindent
Let $(\uu, p, \BB)$ be the solution  of Problem \eqref{eq:variazionale}. 
Then 
$\uu(\cdot, t) \in \ZZs$ and
$\BB(\cdot, t) \in \WWs$ (cf. \eqref{eq:stars})
for a.e. $t \in I$.

\smallskip
\noindent
Under the Assumption \textbf{(RA1-4f)}, the discrete forms in \eqref{eq:forme_d} and \eqref{eq:J} satisfy for a.e. $t \in I$ and for all 
$\cc \in \ZZs \oplus \ZZd$ and $\TT \in \WWd$ the  consistency property
\eqref{eq:forme_con}, moreover referring to the forms \eqref{eq:K} and \eqref{eq:Y}, the following hold
\begin{equation}
\label{eq:forme_con4f}
K_h(\cc; \BB, \HH_h) = 0 \,,
\quad 
b(\BB, \psi_h) = 0 \,,
\qquad \text{$\forall \HH_h \in \WWd$, \, $\forall \psi_h \in \Rd$}
\end{equation}
i.e. all the forms in \eqref{eq:fem-4f} are consistent.

\subsection{Stability analysis}
\label{sub:stab-4f}


Recalling the definition \eqref{eq:Z_h}, consider the form
\[
\astabf \colon 
(\ZZd\times \WWd) \times
(\ZZd\times \WWd) \times
(\ZZd\times \WWd) \to \R 
\]
defined by
\begin{equation}
\label{eq:astabf}
\begin{aligned}
\astabf&((\cc, \TT); (\uu_h, \BB_h), (\vv_h, \HH_h))  :=
\ns \ash(\uu_h, \vv_h) + \nm \am(\BB_h, \HH_h) +
\\  
&
+ J_h(\TT; \uu_h, \vv_h) 
+ c_h(\cc; \uu_h, \vv_h) 
-d(\TT; \BB_h, \vv_h) 
+d(\TT; \HH_h, \uu_h) + K_h(\cc; \BB_h, \HH_h) 
 \,.
\end{aligned}
\end{equation}
Then Problem \eqref{eq:fem-4f} can be formulated as follows: find
$\uu_h \in L^\infty(0, T; \ZZd) $, $\BB_h \in L^\infty(0, T; \WWd)$, $\phi_h \in L^2(0, T; \Rd)$,
such that for a.e. $t \in I$
\begin{equation}
\label{eq:stab fem-4f}
\left \{
\begin{aligned}
\left( \partial_t \uu_h, \vv_h\right) + 
\left( \partial_t \BB_h, \HH_h\right) +  
\astabf((\uu_h, \BB_h), (\uu_h, \BB_h), (\vv_h, \HH_h))
-b(\HH_h, \phi_h)
&=\\
= (\ff, \vv_h)   + (\GG, \HH_h)  &
\\
b(\BB_h, \psi_h) + Y_h(\phi_h, \psi_h) & = 0
\end{aligned}
\right.
\end{equation}
for all $(\vv_h, \HH_h) \in \ZZd \times \WWd$, and for all $\psi_h \in \Rd$,
coupled with initial conditions \eqref{eq:ini cond-d}.

For any $\cc \in \ZZd$ we define the following norm and semi-norms  on  $\WW$ and $\Pkc_k(\Omega_h)$:
\begin{equation}
\label{eq:norme mag 4f}
\begin{gathered}
\normaKc{\BB}^2 :=
\sum_{f \in \EdgesI} h_f^2 \max \{\Vert \cc\Vert_{L^\infty(\omega_f)}^2 , 1 \} \Vert\jump{\bnabla \BB}_f \Vert_f^2 \,,
\quad
\normamfc{\BB}^2 := 
\nm \Vert \bnabla \BB \Vert^2  + \normaKc{\BB}^2,
\\
\normaY{\phi_h}^2 := \bdmY
\sum_{f \in \EdgesI} h_f^2
\Vert \jump{\nabla \phi_h}_f \Vert_f^2 \,.
\end{gathered}
\end{equation}

The following results are instrumental to prove the well-posedness of problem \eqref{eq:stab fem-4f}.
Since some of the derivations here below are standard, we do not provide the full proof for all the results.

\begin{proposition}[Coercivity of $\astabf$]
\label{prp:coef4}
Under the mesh assumption \cfan{(MA1)} 
let $\astabf$ be the form \eqref{eq:astabf}.
If the parameter $\bdma$ in \eqref{eq:forme_d} is sufficiently large there exists a real positive constant $\ccoe$ such that for any $\cc \in \ZZs \oplus \ZZd$ and $\TT \in \WWc$
the following holds
\[
\astabf((\cc, \TT); (\vv_h, \HH_h), (\vv_h, \HH_h)) \geq 
\ccoe \left( \normaV{\vv_h}^2 + \normamfc{\HH_h}^2\right) 
\]
for all $\vv_h \in \ZZd$, $\HH_h \in \WWc$.
\end{proposition}

\begin{proposition}[Inf-sup stability]
\label{prp:infsup}
There exists a real positive constant $\widehat{\beta} >0$ such that for all $\psi_h \in \Rd$ the following holds
\[
\widehat{\beta}\Vert \psi_h\Vert \leq
 \frac{h^{1/2}}{\bdmY^{1/2}} \normaY{\psi_h} + \sup_{\HH_h \in \WWd} \frac{b(\HH_h, \psi_h)}{\Vert \bnabla \HH_h\Vert} \,.
\]
\end{proposition}

\begin{proof}
The proof follows the same guidelines in \cite[Lemma 6]{BF:2007}, but without making use of a quasi-uniformity mesh assumption.
In the following $C$ will denote a generic uniform positive constant.
Let $\psi_h \in \Rd$, from \cite[Corollary 2.4]{Girault-book} there exists $\HH \in \HH^1_0(\Omega)$ such that
\begin{equation}
\label{eq:infsup0}
\diver \HH = \psi_h \,, 
\qquad \qquad 
\Vert \bnabla \HH \Vert \leq C \Vert \psi_h \Vert \,.
\end{equation} 
Employing the interpolation operator $\PWWd$ introduced in Lemma \ref{lm:int-i} we infer
\begin{equation}
\label{eq:infsup1}
\begin{aligned}
\Vert \psi_h \Vert^2 &= 
(\diver \HH, \psi_h) =
(\diver (\HH - \HHi), \psi_h) + (\diver (\HHi), \psi_h)
\\
&=
(\HH - \HHi, \nabla \psi_h) + b(\HHi, \psi_h)
& \quad & \text{(int. by parts)}
\\
&=
(\HH - \HHi, (I - \intcip)(\nabla \psi_h)) + b(\HHi, \psi_h)
& \quad & \text{(by \eqref{eq:int-orth})}
\\
& =: T_1 + T_2 \,.
\end{aligned}
\end{equation}
For the addendum $T_1$ a Cauchy-Schwarz inequality, bound \eqref{eq:err-int}, Lemma \ref{lm:cip} combined with the definition of norm $\normaY{\cdot}$, and bound \eqref{eq:infsup0} imply
\begin{equation}
\label{eq:infsup2}
\begin{aligned}
T_1 & \leq
\biggl(\sum_{E \in \Omega_h} h_E^{-1} \Vert \HH - \HHi\Vert_E^2\biggr)^{1/2}
\biggl(\sum_{E \in \Omega_h} h_E \Vert  (I - \intcip)(\nabla \psi_h)\Vert_E^2\biggr)^{1/2}
\\
& \leq 
C h^{1/2} \Vert \bnabla \HH \Vert \, 
\biggl( 
\sum_{f \in \EdgesI} h_f^2 \Vert \jump{\nabla \psi_h}_f\Vert_f^2
\biggr)^{1/2}
\leq 
C \frac{h^{1/2}}{\bdmY^{1/2}}  \, \normaY{\psi_h} \, \Vert \psi_h \Vert \,.
\end{aligned}
\end{equation}
Concerning $T_2$, combining \eqref{eq:err-int} and \eqref{eq:infsup0} we obtain
\begin{equation}
\label{eq:infsup3}
\begin{aligned}
T_2 &
\leq C 
\frac{b(\HHi, \psi_h)}{\Vert \bnabla \HH\Vert} \Vert \psi_h\Vert
\leq C 
\frac{b(\HHi, \psi_h)}{\Vert \bnabla \HHi\Vert} \Vert \psi_h\Vert
\leq 
 C \sup_{\HH_h \in \WWd}
\frac{b(\HH_h, \psi_h)}{\Vert \bnabla \HH_h\Vert} \Vert \psi_h\Vert \,.
\end{aligned}
\end{equation}
The proof easily follows from \eqref{eq:infsup1}--\eqref{eq:infsup3}.
\end{proof}

\begin{proposition}[Well-posedness of \eqref{eq:fem-4f}]
\label{prp:well-4f}
Under the mesh assumption \cfan{(MA1)} 
if the parameter $\bdma$ in \eqref{eq:forme_d} is sufficiently large, Problem \eqref{eq:stab fem-4f} admits a unique solution $(\uu_h, \BB_h, \phi_h)$. 
Moreover the following bounds hold (cf. \eqref{eq:data})
\begin{align}
\label{eq:stab4f-1}
\Vert \uu_h \|_{L^\infty(0, T; \Ldue)}^2 + 
\Vert \BB_h \|_{L^\infty(0, T; \Ldue)}^2 
&\lesssim \data^2 \,,
\\
\label{eq:stab4f-2}
\Vert \uu_h(\cdot, T) \Vert^2 + \Vert \BB_h(\cdot, T) \Vert^2 +
\int_0^T \|\uu_h(\cdot, t)\|_{(\uu_h(\cdot, t), \BB_h(\cdot, t))}^2 \, {\rm d}t  &+ \nonumber
\\
+\int_0^T \|\BB_h(\cdot, t)\|_{4\rm f, \uu(\cdot, t)}^2 \, {\rm d}t  
+\int_0^T \normaY{\phi_h(\cdot, t)}^2 \, {\rm d}t  
&\lesssim \data^2 \,.
\end{align}
\end{proposition}

\begin{proof}
The existence of a unique solution $(\uu_h, \BB_h, \phi_h)$ to Cauchy problem \eqref{eq:stab fem-4f} can be derived 
using analogous arguments to that in \cite[Theorem 1]{BF:2007} and follows by the Lipschitz continuity of $\astabf$ and the inf-sup stability of Proposition \ref{prp:infsup}. 
The stability bounds can be proved using analogous techniques to that in the proof of Proposition \ref{prp:well-3f}. 

\end{proof}

\subsection{Error analysis}
\label{sub:err4}

We here derive the convergence estimates for the four-field problem.

Let $(\uu_h, \BB_h, \phi_h)$ be the solution of Problem \eqref{eq:stab fem-4f}, then for a.e. $t \in I$, we introduce the following shorthand notation 
(valid for all sufficiently regular vector fields ${\bf H}:\Omega\times [0,T] \rightarrow {\mathbb R}^3$)
\begin{equation}
\label{eq:normastab4}
\normamf{{\bf H}}(t) := \Vert {\bf H}(\cdot, t) \|_{{4\rm f}, \uu_h(\cdot, t)} \,.
\end{equation} 

We make the following assumption on the solution of the continuous problem.

\smallskip
\noindent
\textbf{(RA2-4f) Regularity assumptions on the exact solution (error analysis):} 
\begin{itemize}
\item[(i)] $\uu, \BB \in L^2(0, T; \WW^{k+1}_{\infty}(\Omega_h)) \cap  H^1(0, T; \HH^{k+1}(\Omega_h))$\,,
\item[(ii)] $\uu, \BB \in L^\infty(0, T; \WW^1_\infty(\Omega_h))$\,.
\end{itemize}
We also introduce the following useful quantity for the error analysis
\begin{equation}
\label{eq:quantities4f}
\lmf^2 :=  \max \{\nm\,,  h (\data^2 + 1) \}   \, ,
\end{equation}
which is associated to the convective or diffusive nature of the discrete problem.

%

\begin{proposition}[Interpolation error estimate for the magnetic field]
\label{prp:interpolationmf}
Let Assumption \cfan{(MA1)} hold.
Furthermore, if $k=1$ let also Assumption \cfan{(MA2)} hold. Then, under the regularity assumption  \cfan{(RA2-4f)}, referring to \eqref{eq:fquantities} and \eqref{eq:quantities4f} for a.e. $t \in I$ the following holds:
\[
 \normamf{\EEi}^2 \lesssim 
\lmf^2 \, \Vert \BB \Vert_{\WW^{k+1}_\infty(\Omega_h)}^2 \, h^{2k}  \,.
\]
\end{proposition}

\begin{proof}
Direct calculations yield (here $\omega_E$ denotes the union of the elements in $\Omega_h$ with an edge in common with $E$)
\begin{equation}
\label{eq:eiK}
\begin{aligned}
\normaK{\EEi}^2 = 
& \sum_{f \in \EdgesI} h_f^2 
\max \{\Vert \uu_h\Vert_{\LL^\infty(\omega_f)}^2 , 1 \}
\Vert \jump{\bnabla \EEi}_f \Vert_f^2 
\\
& \begin{aligned}
& \lesssim
\sum_{E \in \Omega_h} h_E^{2k+1}
(\Vert \uu_h\Vert_{\LL^\infty(\omega_E)}^2 + 1) \,
\vert \BB \vert_{k+1,E}^2 
& &\text{(similarly to \eqref{eq:utilejump2})}
\\
& \lesssim
h^{2k+1} \left(
\Vert \BB \Vert_{\WW^{k+1}_\infty(\Omega_h)}^2 \Vert \uu_h\Vert^2 +
 \vert \BB \vert_{k+1,\Omega_h}^2
\right)
& &\text{(Lm. \ref{lm:inverse} \& H\"{o}lder ineq.)}
\\
& \lesssim
h^{2k+1} 
\left(\Vert \BB \Vert_{\WW^{k+1}_\infty(\Omega_h)}^2 
\data^2 
+ \vert \BB \vert_{k+1,\Omega_h}^2
\right) \, .
& &\text{(by \eqref{eq:stab4f-1})}
\end{aligned}
\end{aligned}
\end{equation}
Whereas from \eqref{eq:err-int} we infer
$
\nm\Vert \bnabla \EEi\Vert^2 \lesssim
\nm  h^{2k} \vert \BB \vert_{k+1, \Omega_h}^2 \,.
$
The proof easily follows from the definition of $ \normamf{\cdot}$ in \eqref{eq:normastab4}.
\end{proof}

We now prove the following error estimation.
\begin{proposition}[Discretization error]
\label{prp:error equation 4}
Let Assumption \cfan{(MA1)} hold. Furthermore, if $k=1$ let also Assumption \cfan{(MA2)} hold.
Let $\ccoe$ be the uniform constant in Proposition \ref{prp:coef4}.
Then, under the consistency assumption \cfan{(RA1-4f)} and assuming that the parameter $\bdma$ (cf. \eqref{eq:forme_d}) is sufficiently large, referring to \eqref{eq:fquantities}, for a.e. $t \in I$ the following holds
\begin{equation}
\label{eq:error equation 4}
\frac{1}{2}\partial_t \Vert \eeh  \Vert^2 + 
\frac{1}{2}\partial_t \Vert \EEh \Vert^2 +
\ccoe \left(\normastab{\eeh}^2 +
\normamf{ \EEh}^2 \right) + \normaY{\phi_h}^2\leq 
\sum_{i=1}^4 T_i +  T_{5, 4\rm{f}} 
\end{equation}
where $T_i$ are defined in \eqref{eq:Ti-3f} and
\begin{equation}
\label{eq:Ti-4f}
T_{5, 4\rm{f}} := K_h(\uu_h; \EEi, \EEh) + b(\EEi, \phi_h) \,.
\end{equation}
\end{proposition}

\begin{proof}
Employing Proposition \ref{prp:coef4}, recalling definition of $\normaY{\cdot}$ in \eqref{eq:norme mag 4f}, and equations \eqref{eq:stab fem-4f} and \eqref{eq:variazionale}, we infer
\[
\begin{aligned}
&\frac{1}{2}\partial_t \Vert \eeh  \Vert^2 + 
\frac{1}{2}\partial_t \Vert \EEh \Vert^2 +
\ccoe \left(\normastab{\eeh}^2 +
\normamf{ \EEh}^2\right)  + \normaY{\phi_h}^2 \leq 
\\
&\leq
\left(\partial_t \eeh  \,, \eeh \right) + 
\left(\partial_t \EEh  \,, \EEh \right) +
\astabf((\uu_h, \BB_h); (\eeh, \EEh), (\eeh, \EEh)) 
+ Y_h(\phi_h, \phi_h)
\\
&=
(\ff, \eeh) + (\GG, \EEh) -
\left(\partial_t \uui  \,, \eeh \right) -
\left(\partial_t \BBi  \,, \EEh \right) +
\\
& \quad -
\astabf((\uu_h, \BB_h); (\uui, \BBi), (\eeh, \EEh))+
b(\EEh, \phi_h) - b(\BB_h, \phi_h) 
\\
&=
\left(\partial_t \eei  \,, \eeh \right) +
\left(\partial_t \EEi  \,, \EEh \right) +
\ns \as(\uu,  \eeh) + \nm \am(\BB, \EEh) + c(\uu; \uu, \eeh)  +
\\
& \quad  
-d(\BB; \BB,  \eeh) 
+d(\BB; \EEh, \uu)  -
\astabf((\uu_h, \BB_h); (\uui, \BBi), (\eeh, \EEh)) - b(\BBi, \phi_h) \,.
\end{aligned}
\]
The proof follows from \eqref{eq:forme_con} and  \eqref{eq:forme_con4f}, and observing that 
$\diver \BB =0$ implies
\[
\begin{aligned}
-b(\BBi, \phi_h)=
b(\EEi, \phi_h)\,.
\end{aligned}
\]
\end{proof}


In order to derive the convergence estimate, we make the following assumption on the parameter $\bdmY$ in \eqref{eq:Y}:
\begin{equation}
\label{eq:bdmY}
\bdmY \leq \frac{\ccoe}{32 \clemma}
\min_{f \in \EdgesI} \left\{1, \, \frac{\max\{\Vert \uu_h \Vert_{\LL^\infty(\omega_f)}^2, 1 \}}{\Vert \uu \Vert_{\LL^\infty(\omega_f)}^2} \right\}
\,,
\end{equation}
see also Remark \ref{rem:lite-farl}.

Since the bounds of the terms $T_1,T_2,T_3$ in \eqref{eq:error equation 4} derived in Section \ref{sub:err} are tailored also to the four-field scheme, 
we target directly the remaining ones.

\begin{lemma}[Estimate of $T_4$]
\label{lm:t4f}
Under the assumptions of Proposition \ref{prp:error equation 4}  and   the regularity assumption \cfan{(RA2-4f)},
and assuming that the parameter $\bdmY$ (cf. \eqref{eq:Y}) satisfies \eqref{eq:bdmY}, 
the following holds
\begin{equation}
\label{eq:T4f}
\begin{aligned}
T_4 &\leq 
C   \left(\gunoinf + h \gunoinf^2 + h \frac{\ccoe}{8}\right) \Vert \EEh \Vert^2
+ \frac{\ccoe}{4}  \normaK{\EEh}^2  +
\frac{1}{4}\normaY{\phi_h}^2 + 
\\
& \quad 
+ C \gunoinf^2 \vert \BB \vert^2_{k+1, \Omega_h} h^{2k+1} +
 \gunoinf (\vert \uu \vert^2_{k+1, \Omega_h} + \vert \BB \vert^2_{k+1, \Omega_h}) h^{2k+2} \,.
\end{aligned}
\end{equation}
\end{lemma}

\begin{proof}
We employ the same splitting  of $T_4$ derived in \eqref{eq:T40}.
From \eqref{eq:T41} and \eqref{eq:T42} we obtain
\begin{equation}
\label{eq:T4-12-f}
T_{4,1} + T_{4,2} \leq
C \gunoinf \, \Vert \EEh \Vert^2 + 
C \gunoinf (\vert \uu\vert_{k+1, \Omega_h}^2 + \vert \BB\vert_{k+1, \Omega_h}^2) h^{2k+2} \,. 
\end{equation}
The term $T_{4,3}$ can be written as follows
\begin{equation}
\label{eq:T43f0}
\begin{aligned}
T_{4,3} & = 
- (\EEi, (\bnabla \EEh) \uu) = 
(\EEi, (\bnabla \EEh) (\Pzerok{0}\uu - \uu)) -
(\EEi, (\bnabla \EEh) \Pzerok{0}\uu) 
=: \alpha_1 + \alpha_2 \,.
\end{aligned}
\end{equation}
The term $\alpha_1$ can be bounded as follows:
\begin{equation}
\label{eq:T4alpha1}
\begin{aligned}
\alpha_1 &\leq 
\Vert  \bnabla \EEh (\uu - \Pzerok{0}\uu) \Vert  \, \Vert \EEi \Vert 
& \quad & \text{(Cau.-Sch. ineq.)}
\\
 &\leq 
C \Vert \EEh \Vert \, \Vert \uu \Vert_{\WW^1_\infty(\Omega_h)} \, 
 \Vert \EEi \Vert 
& \quad & \text{(by \eqref{eq:support})}
\\
 &\leq 
C \Vert \uu \Vert_{\WW^1_\infty(\Omega_h)} \, \Vert \EEh \Vert^2 + 
C h^{2k+2}\Vert \uu \Vert_{\WW^1_\infty(\Omega_h)} \, \vert \BB \vert_{k+1, \Omega_h}^2 \, .
& \quad & \text{(Young ineq. \& \eqref{eq:err-int})}
\end{aligned}
\end{equation}
Concerning $\alpha_2$, we set we set 
$-(\bnabla_h \EEh) (\Pzerok{0}\uu) := \pp_{k-1}  \in [\Pk_{k-1}(\Omega_h)]^3$, and we use the same calculations in  \eqref{eq:T3alpha40}--\eqref{eq:T3alpha4} (with $\theta=\ccoe/2$) obtaining  
\begin{equation}
\label{eq:T4alpha2}
\begin{aligned}
\alpha_2 &  \leq 
C \Vert \uu \Vert_{\WW^1_\infty(\Omega_h)}^2 h^{2k+1} \vert \BB\vert^2_{k+1, \Omega_h} +
\frac{\ccoe}{8}\normaK{\EEh}^2 + \frac{C \ccoe}{8} h \, \Vert  \EEh \Vert^2 \,.
\end{aligned}
\end{equation}
Collecting \eqref{eq:T4alpha1} and \eqref{eq:T4alpha2} in \eqref{eq:T43f0} we have
\begin{equation}
\label{eq:T4-3-f}
\begin{aligned}
T_{4,3} &\leq C\left(\gunoinf + \frac{\ccoe}{8} h\right)
\Vert \EEh \Vert^2 
+ \frac{\ccoe}{8}  \normaK{\EEh}^2 
+ C h^{2k+1} \gunoinf \vert \BB \vert_{k+1, \Omega_h}^2  
(\gunoinf + h) \,.
\end{aligned}
\end{equation}
We now estimate the term $T_{4,4}$. We preliminary observe the following:
let $p_k \in \Pkc_k(\Omega_h)$ and let $\widehat{p}$ be its mean value,
being $\BB_h \cdot \nn = 0$ on $\partial \Omega$ then from \eqref{eq:fem-4f}, we infer
\[
(\diver \BB_h, p_k) = 
b(\BB_h, p_k - \widehat{p}) = 
-Y_h(\phi_h, p_k - \widehat{p}) =  -Y_h(\phi_h, p_k) \,.
\]
Therefore, recalling the definition of $\Pt$ in Lemma \ref{lm:Pt} and that $\diver \BB=0$, we have 
\begin{equation}
\label{eq:T44f0}
\begin{aligned}
T_{4,4}  &= (\diver \BB_h, \, \EEh \cdot \uu) 
= 
(\diver \BB_h,\, (I - \Pt)(\EEh \cdot \uu)) +
(\diver \BB_h, \, \Pt(\EEh \cdot \uu))
\\
&= 
(\diver \BB_h,\, (I - \Pt)(\EEh \cdot \uu)) -
Y_h(\phi_h, \Pt(\EEh \cdot \uu))
\\
&= 
((I - \Pt)(\EEh \cdot \uu), \, \diver \BB_h) +
Y_h(\phi_h, (\Pt\uu) \cdot \EEh - \Pt(\EEh \cdot \uu)) -
Y_h(\phi_h, (\Pt\uu) \cdot \EEh)
\\
& =: \beta_1 + \beta_2 + \beta_3 \,. 
\end{aligned} 
\end{equation}
We estimate the terms $\beta_i$ in \eqref{eq:T44f0}. 
Being $\diver \BB =0$, the term $\beta_1$ can be bounded as follows
(using Lemmas \ref{lm:inverse} and \ref{lm:Pt}).
\begin{equation}
\label{eq:beta1}
\begin{aligned}
\beta_1 &\leq 
\sum_{E \in \Omega_h} \Vert (I - \Pt)(\EEh \cdot \uu) \Vert_E 
\left( \Vert \bnabla \EEh  \Vert_E + 
\Vert \bnabla \EEi \Vert_E \right)
\\
&\leq 
\biggl( \sum_{E \in \Omega_h} h_E^{-2}\Vert (I - \Pt)(\EEh \cdot \uu) \Vert_E^2 \biggr)^{1/2}
\biggl( 
\sum_{E \in \Omega_h} h_E^2\Vert \bnabla \EEh  \Vert_E^2 +
 \sum_{E \in \Omega_h} h_E^2
\Vert \bnabla \EEi \Vert_E^2 
 \biggr)^{1/2}
\\
& \leq
C \Vert  \uu \Vert_{\WW^1_\infty(\Omega_h)}  \,
\Vert \EEh \Vert 
( \Vert \EEh\Vert + h^{k+1}|\BB|_{k+1, \Omega_h})
\\
& \leq
C \Vert  \uu \Vert_{\WW^1_\infty(\Omega_h)}  \,
\Vert \EEh \Vert^2 +  
C h^{2k+2} \,\Vert  \uu \Vert_{\WW^1_\infty(\Omega_h)}  \,
|\BB|_{k+1, \Omega_h}^2 \,.
\end{aligned}
\end{equation}
For the term $\beta_2$ we have
\begin{equation}
\label{eq:beta2}
\begin{aligned}
\beta_2 & \leq 
\frac{1}{8} \normaY{\phi_h}^2 + 
C  \sum_{f \in \EdgesI} h_f^2 \bdmY
\Vert \jump{\nabla ((\Pt\uu) \cdot \EEh - \Pt(\EEh \cdot \uu))}_f \Vert_f^2
\\
& \leq 
\frac{1}{8} \normaY{\phi_h}^2 + 
C  \sum_{E \in \Omega_h} h_E 
\Vert \nabla ((\Pt\uu) \cdot \EEh - \Pt(\EEh \cdot \uu)) \Vert_E^2
& \quad & \text{(by \eqref{eq:utile-jump})} 
\\
& \leq 
\frac{1}{8} \normaY{\phi_h}^2 + 
C  \sum_{E \in \Omega_h} h_E^{-1} 
\Vert (\Pt\uu) \cdot \EEh - \Pt(\EEh \cdot \uu) \Vert_E^2
& \quad & \text{(Lm. \ref{lm:inverse})} 
\\
& \leq 
\frac{1}{8} \normaY{\phi_h}^2 + 
C  h \Vert  \uu \Vert_{\WW^1_\infty(\Omega_h)}^2  \,
\Vert \EEh \Vert^2 \, .
& \quad & \text{(Lm. \ref{lm:Pt})} 
\end{aligned}
\end{equation}
Finally for the term $\beta_3$ we have
\begin{equation}
\label{eq:beta3}
\begin{aligned}
\beta_3 & \leq 
\frac{1}{8} \normaY{\phi_h}^2 + 
2 \bdmY\sum_{f \in \EdgesI} h_f^2 
\Vert \jump{\nabla ((\Pt \uu) \cdot \EEh)}_f \Vert_f^2
\\
& \leq 
\frac{1}{8} \normaY{\phi_h}^2 + 
4\bdmY\sum_{f \in \EdgesI} h_f^2 \left(
\Vert \jump{\bnabla (\Pt\uu)}_f |\EEh| \Vert_f^2 + 
\Vert \jump{\bnabla \EEh}_f |\Pt \uu| \Vert_f^2
\right)
\\
& \leq 
\frac{1}{8} \normaY{\phi_h}^2 +
4 \bdmY\sum_{f \in \EdgesI} h_f^2 \left( 
\Vert\bnabla (\Pt\uu)\Vert_{\LL^\infty(\omega_f)}^2
\Vert \EEh \Vert_f^2 +
\Vert \Pt \uu\Vert_{\LL^\infty(\omega_f)}^2
\Vert \jump{\bnabla \EEh}_f \Vert_f^2 \right)
\\
& \leq 
\frac{1}{8} \normaY{\phi_h}^2 + 
4 \clemma \bdmY h
\Vert \uu\Vert_{\WW^1_\infty(\Omega_h)}^2
\Vert \EEh \Vert^2 +
4 \clemma \bdmY
\sum_{f \in \EdgesI} h_f^2
 \Vert \uu\Vert_{\LL^\infty(\omega_f)}^2 
\Vert \jump{\bnabla \EEh}_f \Vert_f^2
\\
& \leq 
\frac{1}{8} \normaY{\phi_h}^2 + 
C h
\Vert \uu\Vert_{\WW^1_\infty(\Omega_h)}^2
\Vert \EEh \Vert^2 +
\frac{\ccoe}{8} \normaK{ \EEh}^2
\end{aligned}
\end{equation}
where in the penultimate inequality we used  \eqref{eq:Pt2} and \eqref{eq:utile-jump}, whereas in the last inequality we employed  \eqref{eq:bdmY}.

Therefore from \eqref{eq:T44f0}--\eqref{eq:beta3} we infer
\begin{equation}
\label{eq:T44fF}
T_{4,4} \leq 
C \gunoinf (1 + h \gunoinf)\Vert \EEh \Vert^2 + 
\frac{1}{4}\normaY{\phi_h}^2 + 
\frac{\ccoe}{8} \normaK{ \EEh}^2 + 
C \gunoinf \vert \uu \vert_{k+1, \Omega_h}^2 \, h^{2k+2}  \,.
\end{equation}
The proof now follows from \eqref{eq:T4-12-f}, \eqref{eq:T4-3-f}, \eqref{eq:T44fF}.
\end{proof}

\begin{lemma}[Estimate of $T_{5, 4\rm{f}}$]
\label{lm:t5f}
Under the assumptions of Proposition \ref{prp:error equation}  and   the regularity assumption \cfan{(RA2-4f)}, the following holds
\begin{equation}
\label{eq:T5f}
T_{5, 4\rm{f}} \leq 
\frac{\ccoe}{4} \normaK{\EEh}^2 +
\frac{1}{4} \normaY{\phi_h}^2  + 
C h (1 + \gunoinf^2 +\data^2) \Vert \BB\Vert_{\WW^{k+1}_\infty(\Omega_h)}^2 h^{2k} \,.
\end{equation}
\end{lemma}

\begin{proof}
Recalling the definition of norm $\normaK{\cdot}$, employing the Cauchy-Schwarz inequality, the Young inequality, and bound \eqref{eq:eiK}, we have
\[
\begin{aligned}
K_h(\uu_h; \EEi, \EEh) &\leq 
\normaK{\EEi} \, \normaK{\EEh}
\\
&\leq 
\frac{\ccoe}{4} \normaK{\EEh}^2 + 
C \, h^{2k+1} 
(\Vert \BB\Vert_{\WW^{k+1}_\infty(\Omega_h)}^2 \data^2 + \vert \BB\vert_{\HH^{k+1}(\Omega_h)}^2) \,.
\end{aligned}
\]
The second add in $T_{5, 4\rm{f}}$ can be bounded as follows:
\[
\begin{aligned}
& b(\EEi, \phi_h)  = 
-(\EEi, \nabla \phi_h) = 
-(\EEi, (I - \intcip)\nabla \phi_h)
&\quad
& \text{(int. by parts \& \eqref{eq:int-orth})}
\\
& \leq 
\biggl(\sum_{E \in \Omega_h} h_E^{-1} \Vert \EEi \Vert_E^2 \biggr)^{1/2}
\biggl(\sum_{E \in \Omega_h} h_E \Vert (I - \intcip)\nabla \phi_h
 \Vert_E^2 \biggr)^{1/2}
 &\quad
& \text{(Cau.-Sch- ineq.)}
\\
& \leq
C \bdmY^{-1} h^{2k+1} \vert \BB \vert_{k+1, \Omega_h}^2 + 
\frac{1}{4} \bdmY\sum_{f \in \EdgesI} \Vert \jump{\nabla \phi_h}_h\Vert_f^2
 &\quad
& \text{(\eqref{eq:err-int} \& Young ineq.)}
\\
& \leq
C  h^{2k+1} \gunoinf^2
\vert \BB \vert_{k+1, \Omega_h}^2 + 
\frac{1}{4}\normaY{\phi_h}^2 \,.
\end{aligned}
\]

\end{proof}

\begin{proposition}[error estimate]
\label{prp:convf}
Let Assumption \cfan{(MA1)} hold. Furthermore, if $k=1$ let also Assumption \cfan{(MA2)} hold.
Then, under the consistency assumption \cfan{(RA1-4f)} and the regularity assumption \cfan{(RA2-4f)}
and assuming that the parameter $\bdma$ (cf. \eqref{eq:forme_d}) is sufficiently large, the parameter $\bdmY$ (cf. \eqref{eq:Y}) satisfies \eqref{eq:bdmY}, and $h \lesssim 1$,
referring to \eqref{eq:quantities} and to \eqref{eq:quantities4f},  the following holds
\begin{equation}
\label{eq:conv-4f}
\begin{aligned}
&\Vert (\uu - \uu_h)(\cdot, T) \Vert^2 + \Vert (\BB_h - \BB_h)(\cdot, T)\Vert^2 +
\\
& +
\int_0^T \normastab{(\uu - \uu_h)(t)}^2 \, {\rm d}t +
\int_0^T \normamf{(\BB - \BB_h)(\cdot, t)}^2 \, {\rm d}t +
\int_0^T \normaY{\phi_h(\cdot, t)}^2 \, {\rm d}t
\lesssim 
\\
&(\ls^2 \Vert \uu \Vert_{L^2(0,T; \WW^{k+1}_\infty(\Omega_h))}^2 + 
\lmf^2 \Vert \BB \Vert_{L^2(0,T; \WW^{k+1}(\Omega_h))}^2) h^{2k} +
\\
&+
( \Vert \uu \Vert_{H^1(0,T; \HH^{k+1}_\infty(\Omega_h))}^2 + 
 \Vert \BB \Vert_{H^1(0,T; \HH^{k+1}(\Omega_h))}^2) h^{2k+2}
 \,,
\end{aligned}
\end{equation}
where the hidden constant depends also on $\gunoinf$.
\end{proposition}

\begin{proof}
We start by noticing that from \eqref{eq:int-v}, \eqref{eq:err-int}, Proposition \ref{prp:interpolation} and Proposition \ref{prp:interpolationmf} we infer
\begin{equation}
\label{eq:conv-4f-ei}
\begin{aligned}
\Vert \eei(\cdot, T) \Vert^2 &+ \Vert \EEi(\cdot, T)\Vert^2 +
\int_0^T \normastab{\eei(\cdot, t)}^2 \, {\rm d}t  +
\int_0^T \normamf{\EEi(\cdot, t)}^2 \, {\rm d}t  
\\
&\lesssim 
h^{2k+2} \left(\vert \uu(\cdot, T) \vert_{k+1, \Omega_h}^2 + 
\vert \BB(\cdot, T)\vert_{k+1, \Omega_h}^2\right) +
\\
&+
(\ls^2 \Vert \uu \Vert_{L^2(0,T; \WW^{k+1}_\infty(\Omega_h))}^2 + 
\lmf^2 \Vert \BB \Vert_{L^2(0,T; \WW^{k+1}_\infty(\Omega_h))}^2) h^{2k} 
 \,.
\end{aligned}
\end{equation}
From Proposition \ref{prp:error equation 4}, using Lemmas \ref{lm:t1}--\ref{lm:t3} plus Lemmas \ref{lm:t4f} and \ref{lm:t5f}, further choosing $\theta = \frac{\ccoe}{2}$ in \eqref{eq:T1}, \eqref{eq:T2} and \eqref{eq:T3}, we obtain 
\[
\begin{aligned}
\partial_t \Vert \eeh  \Vert^2 &+ 
\partial_t \Vert \EEh \Vert^2 +
\left(\normastab{\eeh}^2 +
\normamt{\BB_h}^2 \right) 
\lesssim \\
& \left(\gunoinf + \gunoinf^2 h + \frac{\ccoe}{8}h\right)
 \left( \Vert \eeh \Vert^2 + \Vert \EEh \Vert^2 \right) + \\
& +
 \, (\ls^2 \, \Vert \uu \Vert_{\WW^{k+1}_\infty(\Omega_h)}^2 + \lmf^2 \, \Vert \BB \Vert_{\WW^{k+1}_\infty(\Omega_h)}^2) h^{2k} + 
\gunoinf^2 |\BB|_{k+1, \Omega_h}^2 \, h^{2k+1}+
\\
& +
\gunoinf  (
|\uu|_{k+1,\Omega_h}^2 + |\BB|_{k+1,\Omega_h}^2 + 
|\uu_t|_{k+1,\Omega_h}^2 + |\BB_t|_{k+1,\Omega_h}^2) h^{2k+2}
\end{aligned}
\]
with initial condition 
$\eeh(\cdot, 0) = \EEh(\cdot, 0) = 0$ (cf. \eqref{eq:ini cond-d}).
Therefore, employing the Gronwall lemma we finally have
\begin{equation}
\label{eq:conv-4f-eh}
\begin{aligned}
\Vert \eeh(\cdot, T) \Vert^2 &+ \Vert \EEh(\cdot, T)\Vert^2 +
\int_0^T \normastab{\eeh(\cdot, t)}^2 \, {\rm d}t  +
\int_0^T \normamt{\EEh(\cdot, t)}^2 \, {\rm d}t 
\\
&\lesssim 
(\ls^2 \Vert \uu \Vert_{L^2(0,T; \WW^{k+1}_\infty(\Omega_h))}^2 + 
\lmf^2 \Vert \BB \Vert_{L^2(0,T; \WW^{k+1}_\infty(\Omega_h))}^2) h^{2k} +
\\
&+
( \Vert \uu \Vert_{H^1(0,T; \HH^{k+1}(\Omega_h))}^2 + 
 \Vert \BB \Vert_{H^1(0,T; \HH^{k+1}(\Omega_h))}^2) h^{2k+2}
 \,,
\end{aligned}
\end{equation}
where the hidden constant depends also on $\gunoinf$.
The proof now follows by the triangular inequality.

\end{proof}
\begin{remark}
As already observed, similarly to Remark \ref{rem:4F-teaser}, the above result expresses the pressure robustness and the quasi-robustness of the four-field method. Furthermore, it is immediate to check that for small values of $\ns,\nm$ (compared to $h$, c.f. \eqref{eq:quantities}, \eqref{eq:quantities4f}) it also guarantees a quicker $O(h^{k+1/2})$ error reduction rate. \end{remark}

\begin{remark}\label{rem:lite-farl}
Note that, asymptotically for small $h$, the right hand side in condition \eqref{eq:bdmY} for the parameter $\bdmY$ is expected to behave as $\frac{\ccoe}{32 \clemma}$ (where we recall the constants $\ccoe$ and $\clemma$ appear in Proposition \ref{prp:coef4} and Lemma \ref{lm:Pt}, respectively).
Therefore, although such condition in principle depends on the velocity $\uu$, in practice it is expected to be fairly independent of the exact solution.
Note that a requirement of type \eqref{eq:bdmY} is already present in the literature of nonlinear fluidodynamics, for instance in \cite{BF:2007}.
\end{remark}

\section{A basic investigation for the pressure variable}
\label{sec:pressione}

Deriving optimal error estimates for the pressure variable (which is not a trivial consequence of the error bounds for velocity and magnetic fields) is outside the scopes of the present work. We here limit ourselves in providing a basic result and some comments in a companion remark. 

\begin{proposition}[pressure error estimate]
\label{prp:pressione}
Let $(\uu, p, \BB)$ be the solution of Problem \eqref{eq:primale} and let $(\uu_h, p_h, \BB_h)$ be the solution of Problem \eqref{eq:fem-3f} (resp. \eqref{eq:fem-4f}).
Then, under the assumptions of Proposition \ref{prp:conv} (resp. Proposition \ref{prp:convf}) and assuming 
$p \in L^2(0, T; H^k(\Omega_h))$
the following holds
\begin{equation}
\label{eq:conv-pressione}
\begin{aligned}
\int_0^T &\Vert (p - p_h)(t) \Vert^2 \, {\rm d}t 
\lesssim 
\Vert p \Vert_{L^2(0,T; H^{k}(\Omega_h))}^2  h^{2k} + 
\left \Vert \partial_t (\uu - \uu_h) \right\Vert_{L^2(0, T; {\VVd}')}^2  +
\\ 
&
+
(\Vert \uu \Vert_{L^2(0,T; \HH^{k+1}(\Omega_h))}^2 +
\Vert \BB \Vert_{L^2(0,T; \HH^{k+1}(\Omega_h))}^2) 
(\ns^2 + \gzhinf^2 h^2) h^{2k}+
\\
& + 
(\gunoinf^2 + \gunohinf^2 + \ns + \gzhinf^2 h) \times
\\
& \times \left(
\Vert \uu - \uu_h \Vert_{L^\infty(0, T; \LL^2(\Omega))}^2 + \Vert \BB - \BB_h \Vert_{L^\infty(0, T; \LL^2(\Omega))}^2 + 
\int_0^T \normastab{\uu -\uu_h}^2 \, {\rm d}t \right)
\end{aligned}
\end{equation}
where $\Gamma$ was defined in \eqref{eq:gunoinf} and
\[
\begin{aligned}
\gzhinf &:= 1 + \Vert \uu_h \Vert_{L^\infty(0, T; \LL^\infty(\Omega_h))} 
+ \Vert \BB_h \Vert_{L^\infty(0, T; \LL^\infty(\Omega_h))} 
\\
\gunohinf &:= \Vert \uu_h \Vert_{L^2(0, T; \LL^\infty(\Omega_h))} + \Vert \BB_h \Vert_{L^2(0, T; \WW^1_\infty(\Omega_h))}  \,.
\end{aligned}
\]
\end{proposition}

\begin{proof}
%
Combining the classical inf-sup arguments for the BDM element in \cite{boffi-brezzi-fortin:book} with the Korn inequality in \cite{korn} and the Poincar\'e inequality for piecewise regular functions \cite[Corollary 5.4]{DiPietro}, for a.e. $t \in I$ there exists $\ww_h \in \VVd$ such that
\begin{equation}
\label{eq:p-infsup}
\Vert \ww_h \Vert_{1,h} \lesssim 1 \,, 
\qquad \text{and} \qquad
\Vert p_h - \Pi_{k-1} p\Vert 
\lesssim b(\ww_h \,, p_h - \Pi_{k-1} p) \,.
\end{equation}
Furthermore, recalling the definition of $L^2$-projection operator, being $\diver(\VVd) \subseteq \Pk_{k-1}(\Omega_h)$, we infer
\begin{equation}
\label{eq:p-orth}
b(\ww_h \,, \Pi_{k-1}p) =
b(\ww_h \,,  p) \,.
\end{equation}
Therefore, combining \eqref{eq:p-infsup} and \eqref{eq:p-orth} with Problems \eqref{eq:variazionale} and \eqref{eq:fem-3f} (equivalently \eqref{eq:fem-4f}),
 for a.e. $t \in I$ we infer
\begin{equation}
\label{eq:p-error}
\begin{aligned}
&\Vert p - p_h\Vert 
\lesssim
\Vert p_h -  \Pi_{k-1} p\Vert +  \Vert p -  \Pi_{k-1} p\Vert 
\lesssim
b(\ww_h \,, p_h - p)  +  \Vert p -  \Pi_{k-1} p\Vert 
\\
& =
\ns \bigl( \as(\uu, \ww_h) - \ash(\uu_h, \ww_h) \bigr) +
\bigl( c(\uu; \uu, \ww_h) - c_h(\uu_h; \uu_h, \ww_h) \bigr) + 
\\
&+ \bigl( d(\BB_h; \BB_h, \ww_h) - d(\BB; \BB, \ww_h) \bigr) 
- J_h(\BB_h; \uu_h, \ww_h) + \left(\partial_t (\uu - \uu_h), \ww_h \right) +  \Vert p -  \Pi_{k-1} p\Vert 
\\
& =: \sum_{i=1}^6 S_i \,.
\end{aligned}
\end{equation}
We preliminary notice that, employing again the Korn inequality  and the Poincar\'e inequality for piecewise regular functions, from bounds \eqref{eq:utile-jump} and \eqref{eq:p-infsup} we have
\begin{equation}
\label{eq:norme-wh}
\begin{aligned}
\Vert \ww_h \Vert^2 + 
\Vert \bnabla_h \ww_h \Vert^2 &\lesssim 1
\\
\normaupw{\ww_h}^2 &\lesssim 
 \Vert \uu_h\Vert_{\LL^\infty(\Omega)} \, h
\\
\normacip{\ww_h}^2 
& \lesssim 
\max \{\Vert \BB_h\Vert_{\LL^\infty(\Omega)}^2, 1\} \, h
\end{aligned}
\end{equation}
where we also used $\bdma^{-1} \lesssim 1$.
We now estimate separately each term in the sum \eqref{eq:p-error}.

\noindent
$\bullet$ \, Estimate of $S_1$: recalling the regularity assumptions \textbf{(RA1-3f)} and \textbf{(RA1-4f)} we infer 
\begin{equation*}
\label{eq:S-1-0}
\begin{aligned}
S_1 &=  
\ns \biggl( (\beps_h(\uu - \uu_h) ,\, \beps_h(\ww_h)) +
\bdma \sum_{f \in \Edges} h_f^{-1} (\jump{\uu - \uu_h}_f ,\,\jump{\ww_h}_f)_f \biggr) +
\\
& - \ns \sum_{f \in \Edges}  (\media{\beps_h(\uu - \uu_h)\nn_f}_f ,\, \jump{\ww_h}_f)_f   
- \ns \sum_{f \in \Edges}
(\jump{\uu - \uu_h}_f ,\, \media{\beps_h(\ww_h) \nn_f}_f)_f
\\
& =: S_{1,1} + S_{1,2}  + S_{1,3}  
\end{aligned}
\end{equation*}
The Cauchy-Schwarz inequality  and \eqref{eq:p-infsup} yield
\begin{equation*}
\label{eq:S-1-1}
S_{1,1} \leq \ns \normaunoh{\uu - \uu_h} \, \normaunoh{\ww_h} \lesssim
\ns \normaunoh{\uu - \uu_h} \,.
\end{equation*}
The term $S_{1,2}$ can be bounded as follows
\begin{equation}
\label{eq:S-1-2}
\begin{aligned}
& S_{1,2} \lesssim 
\ns 
\biggl( \sum_{f \in \Edges}h_f  \Vert \media{\beps_h(\uu - \uu_h) \nn_f}_f \Vert_f^2\biggr)^{1/2}
\biggl( \bdma \sum_{f \in \Edges} h_f^{-1}\Vert \jump{\ww_h}_f \Vert_f^2 \biggr)^{1/2}
\\
&\begin{aligned}
& \lesssim 
\ns  
\biggl( \sum_{E \in \Omega_h} \Vert \beps_h(\uu - \uu_h) \Vert_E^2 +
\sum_{E \in \Omega_h} h_E^2 \vert \beps_h(\uu - \uu_h) \vert_{1,E}^2
 \biggr)^{1/2}
& & \text{(\eqref{eq:utile-jump1} \& \eqref{eq:p-infsup})}
\\
& \lesssim 
\ns 
\biggl( \normaunoh{\uu - \uu_h}^2 +
\sum_{E \in \Omega_h} h_E^2 \vert \beps_h (\eeh) \vert_{1,E}^2 +
\sum_{E \in \Omega_h} h_E^2 \vert \beps_h (\eei) \vert_{1,E}^2
 \biggr)^{1/2}
& & \text{(tri. ineq.)}
 \\
 & \lesssim 
\ns 
\biggl( \normaunoh{\uu - \uu_h}^2 +
\sum_{E \in \Omega_h} \Vert \beps_h (\eeh) \Vert_{E}^2 +
\sum_{E \in \Omega_h} h_E^{2k} \vert \uu \vert_{k+1,E}^2
 \biggr)^{1/2}
& &\text{(Lm. \ref{lm:inverse} \& \eqref{eq:int-v})} 
 \\
 & \lesssim 
\ns  
\biggl( \normaunoh{\uu - \uu_h}^2 +
h^{2k} \vert \uu \vert_{k+1,\Omega_h}^2
\biggr)^{1/2} \, .
& & \text{(tri. ineq.)}
\end{aligned}
\end{aligned}
\end{equation}
For the term $S_{1,3}$, employing \eqref{eq:utile-jump} and bound \eqref{eq:p-infsup} we deduce
\begin{equation*}
\label{eq:S-1-3}
\begin{aligned}
S_{1,3}  &\lesssim 
\ns  \biggl( \bdma \sum_{f \in \Edges} h_f^{-1}\Vert \jump{\uu - \uu_h}_f \Vert_f^2 \biggr)^{1/2}
\biggl( \sum_{f \in \Edges}h_f  \Vert \media{\beps_h(\ww_h) \nn_f}_f \Vert_f^2\biggr)^{1/2}
\\
& \lesssim 
\ns \normaunoh{\uu - \uu_h}
\biggl( \sum_{E \in \Omega_h}  \Vert \beps_h(\ww_h)\Vert_E^2\biggr)^{1/2} 
\lesssim 
\ns  \normaunoh{\uu - \uu_h} \,.
\end{aligned}
\end{equation*}
Then we conclude
\begin{equation}
\label{eq:S-1}
S_1 \lesssim 
\ns  
\biggl( \normaunoh{\uu - \uu_h}^2 +
h^{2k} \vert \uu \vert_{k+1,\Omega_h}^2
\biggr)^{1/2} \,.
\end{equation}

\noindent
$\bullet$ \, Estimate of $S_2$: Using analogous argument to  that in \eqref{eq:T20} we infer
\begin{equation*}
\label{eq:S-2-0}
\begin{aligned}
S_2 
& = 
\biggl( (( \bnabla \uu ) \, (\uu - \uu_h), \, \ww_h )  
- (( \bnabla_h \ww_h ) \, \uu_h, \, \uu - \uu_h) \biggr) +
\\
& \quad  
+ \sum_{f \in \EdgesI} ( (\uu_h\cdot \nn_f) \jump{\ww_h}_f ,\, \media{\uu - \uu_h}_f)_f 
+ \sum_{f \in \EdgesI}
( |\uu_h\cdot \nn_f| \jump{\uu - \uu_h}_f ,\,  \jump{\ww_h}_f)_f 
\\
&=: 
S_{2,1} + S_{2,2}  + S_{2,3}\,. 
\end{aligned}
\end{equation*}
The Cauchy-Schwarz inequality and the first bound in \eqref{eq:norme-wh} imply
\begin{equation*}
\label{eq:S-2-1}
\begin{aligned}
S_{2,1} &\leq \Vert \uu - \uu_h\Vert 
( \Vert \uu \Vert_{\WW^1_\infty(\Omega_h)} \, \Vert  \ww_h\Vert +  
\Vert \uu_h \Vert_{\LL^\infty(\Omega_h)} \, \Vert \bnabla_h \ww_h\Vert )
\\
&\lesssim 
( \Vert \uu \Vert_{\WW^1_\infty(\Omega_h)} +  
\Vert \uu_h \Vert_{\LL^\infty(\Omega_h)} )
\Vert \uu - \uu_h\Vert \,. 
\end{aligned}
\end{equation*}
The term $S_{2,2}$ can be bounded using analogous arguments to that in \eqref{eq:S-1-2}:
\begin{equation*}
\label{eq:S-2-2}
\begin{aligned}
S_{2,2} &\lesssim
\Vert \uu_h\Vert_{\LL^\infty(\Omega)}  
\biggl( \sum_{f \in \EdgesI} h_f \Vert \media{\uu - \uu_h}_f\Vert_f^2 \biggr)^{1/2}
\biggl( \bdma \sum_{f \in \EdgesI} h_f^{-1}\Vert \jump{\ww_h}_f\Vert_f^2 \biggr)^{1/2}
\\
&\lesssim
\Vert \uu_h\Vert_{\LL^\infty(\Omega)}  
\biggl(\Vert \uu - \uu_h\Vert^2 +
h^{2k+2} \vert \uu \vert_{k+1,\Omega_h}^2
\biggr)^{1/2} \,.
\end{aligned}
\end{equation*}
Concerning the term $S_{2,3}$, employing \eqref{eq:norme-wh} we have
\begin{equation*}
\label{eq:S-2-3}
S_{2,3} \lesssim  
\normaupw{\ww_h} 
\, \normaupw{\uu - \uu_h}
\lesssim 
\Vert \uu_h\Vert_{\LL^\infty(\Omega)}^{1/2} \, h^{1/2}    \, \normaupw{\uu - \uu_h}  \,.
\end{equation*}
Hence we get
\begin{equation}
\label{eq:S-2}
\begin{aligned}
S_2 &\lesssim
( \Vert \uu \Vert_{\WW^1_\infty(\Omega_h)} +  
\Vert \uu_h \Vert_{\LL^\infty(\Omega_h)}) 
\Vert \uu - \uu_h\Vert 
+
\\
& + \Vert \uu_h\Vert_{\LL^\infty(\Omega)}^{1/2}  h^{1/2}  \,   \normaupw{\uu - \uu_h} +
\Vert \uu_h\Vert_{\LL^\infty(\Omega)}  
h^{k+1} \vert \uu \vert_{k+1,\Omega_h} \,. 
\end{aligned}
\end{equation}

\noindent
$\bullet$ \, Estimate of $S_3$: Employing similar computations to that in \eqref{eq:T30} we derive

\begin{equation*}
\label{eq:S-3-0}
\begin{aligned}
S_3 &= 
\bigl( d(\BB_h - \BB; \BB, \ww_h) + 
(\BB - \BB_h, \bcurl_h(\ww_h \times \BB_h)) \bigr) +
\\
& 
+ \sum_{f \in \EdgesI}(\jump{\ww_h}_f \times \BB_h, (\BB - \BB_h) \times \nn_f)_f 
=: S_{3,1} + S_{3,2} \,.
\end{aligned}
\end{equation*}
From the Cauchy-Schwarz inequality and \eqref{eq:norme-wh} we infer
\begin{equation*}
\label{eq:S-3-1}
\begin{aligned}
S_{3,1} &\leq 
\Vert \BB - \BB_h\Vert 
\biggl( \Vert \bnabla \BB \Vert_{\LL^\infty(\Omega_h)} \Vert \ww_h\Vert + 
\Vert \bnabla_h \ww_h \Vert \, \Vert \BB_h \Vert_{\LL^\infty(\Omega)} +
\Vert \bnabla \BB_h \Vert_{\LL^\infty(\Omega_h)} \Vert \ww_h\Vert
\biggr)
\\
&\lesssim 
(\Vert \BB \Vert_{\WW^1_\infty(\Omega_h)} + 
\Vert \BB_h \Vert_{\WW^1_\infty(\Omega_h)})
 \Vert \BB - \BB_h\Vert \,.
\end{aligned}
\end{equation*}
Using again the similar computations to that in \eqref{eq:S-1-2} we get
\begin{equation*}
\label{eq:S-3-2}
\begin{aligned}
S_{3,2} &\lesssim
\Vert \BB_h\Vert_{\LL^\infty(\Omega)}  
\biggl( \sum_{f \in \EdgesI} h_f \Vert \BB - \BB_h\Vert_f^2 \biggr)^{1/2}
\biggl( \bdma\sum_{f \in \EdgesI} h_f^{-1}\Vert \jump{\ww_h}_f\Vert_f^2 \biggr)^{1/2}
\\
&\lesssim
\Vert \BB_h\Vert_{\LL^\infty(\Omega)}  
\biggl(\Vert \BB - \BB_h\Vert^2 +
h^{2k+2} \vert \BB \vert_{k+1,\Omega_h}^2
\biggr)^{1/2} \,.
\end{aligned}
\end{equation*}
Therefore we obtain
\begin{equation}
\label{eq:S-3}
S_3 \lesssim 
(\Vert \BB \Vert_{\WW^1_\infty(\Omega_h)} + 
\Vert \BB_h \Vert_{\WW^1_\infty(\Omega_h)})
 \Vert \BB - \BB_h\Vert 
+
\Vert \BB_h\Vert_{\LL^\infty(\Omega)}  
h^{k+1}
\vert \BB \vert_{k+1,\Omega_h} \,.
\end{equation}

\noindent
$\bullet$ \, Estimate of $S_4$: employing the regularity assumptions \textbf{(RA2-3f)} and \textbf{(RA2-4f)} and \eqref{eq:norme-wh} we have
\begin{equation}
\label{eq:S-4}
\begin{aligned}
S_4 
& \leq 
\normacip{\uu - \uu_h} \, \normacip{\ww_h}
\lesssim 
\max \{\Vert \BB_h\Vert_{\LL^\infty(\Omega)}, 1\} \, h^{1/2} \, \normacip{\uu - \uu_h} 
\,.
\end{aligned}
\end{equation}

\noindent
$\bullet$ \, Estimate of $S_5$ and $S_6$: by definition of dual norm and Lemma \ref{lm:bramble} we get
\begin{equation}
\label{eq:S-5-6}
S_5 + S_6 \lesssim 
\left \Vert \partial_t (\uu - \uu_h) \right\Vert_{{\VVd}'} +  \vert p \vert_{k, \Omega_h} h^{k} \,.
\end{equation}
Combining equations \eqref{eq:S-1} --\eqref{eq:S-5-6} in \eqref{eq:p-error} we finally obtain 
\begin{equation*}
\label{eq:rho_h}
\begin{aligned}
&\Vert p - p_h \Vert^2 \lesssim
\\
&
(\Vert \uu\Vert_{\WW^1_\infty(\Omega_h)}^2 +\Vert \uu_h\Vert_{\LL^\infty(\Omega_h)}^2 ) 
\Vert \uu - \uu_h \Vert^2  +
(\Vert \BB\Vert_{\WW^1_\infty(\Omega_h)}^2 +\Vert \BB_h\Vert_{\WW^1_\infty(\Omega_h)}^2 )  \Vert \BB - \BB_h \Vert^2  +
\\
&
(\ns + \Vert \uu_h\Vert_{\LL^\infty(\Omega)} h + \max \{\Vert \BB_h\Vert_{\LL^\infty(\Omega)}^2, 1\} h) \normastab{\uu - \uu_h}^2 +
\Vert \partial_t (\uu - \uu_h) \Vert_{{\VVd}'}^2 + 
\\
&
+ \biggl( 
(\ns^2 + \Vert \uu_h\Vert_{\LL^\infty(\Omega)}^2 h^2) \vert \uu \vert_{k+1, \Omega_h}^2 +
\Vert \BB_h\Vert_{\LL^\infty(\Omega)}^2 h^2 \vert \BB \vert_{k+1, \Omega_h}^2 +
\vert p \vert_{k, \Omega_h}^2
\biggr) h^{2k}
\end{aligned}
\end{equation*}
The proof follows integrating the previous bound over $(0, T)$.
\end{proof}

\begin{remark} \label{rm:pressure}
We observe that bound \eqref{eq:conv-pressione}
is independent of  the inverse of the viscosity parameters $\ns$ and $\nm$
and hinges on four terms.
The first term represents the standard interpolation error.
The term $\left \Vert \partial_t (\uu - \uu_h) \right\Vert_{L^2(0, T; {\VVd}')}^2$ expresses the approximation error in the time derivative of the velocities. Deducing an explicit bound for that term is not trivial and beyond the scopes of this contribution.
A possible attempt to estimate the time derivative of the velocity can be found in \cite[Theorem 3]{BF:2007}. The proposed estimate  is suboptimal for $k\geq 2$ polynomial degree.
The third term, in the convection dominated regime, has higher asymptotic order, indeed using Lemma \ref{lm:inverse} and Lemma \ref{lm:bramble}, from Proposition \ref{prp:conv} and Proposition \ref{prp:convf}, it holds
\[
\begin{aligned}
\ns^2 + \gzhinf^2 h^2 & \lesssim 
\ns^2 + \gunoinf^2 h^2 + \left( \Vert \uu - \uu_h\Vert_{L^\infty(0, T; \LL^2(\Omega))}^2  + \Vert \BB - \BB_h\Vert_{L^\infty(0, T; \LL^2(\Omega))}^2 \right) h^{-1}
\lesssim 
 h \,. 
\end{aligned}
\]
The last term consists of the error of the velocities (cf. Proposition \ref{prp:conv} and Proposition \ref{prp:convf}) multiplied by a factor depending on the $\WW^1_\infty$ norm of the discrete solutions.
Notice that such factor can be easily bounded as follows
\[
\gunoinf^2 + \gunohinf^2 + \ns + \gzhinf^2 h
\lesssim 
\gunoinf^2 + \ns + \left( \Vert \uu - \uu_h\Vert_{L^\infty(0, T; \LL^2(\Omega))}^2 + \Vert \BB - \BB_h\Vert_{L^\infty(0, T; \LL^2(\Omega))}^2 \right) h^{-5} \,. 
\]
Employing Proposition \ref{prp:conv} and Proposition \ref{prp:convf}, we can conclude that the last term in \eqref{eq:conv-pressione} behaves as $h^{4k - 5}$  for the three-field scheme
and as $(h^2+\ns^2+\nm^2) h^{4k-5}$ for the four-field scheme. Therefore such last term recovers the $O(h^{2k})$ asymptotic behavior for $k \geq 3$.
\end{remark}


\section{Numerical experiments}\label{sec:num}

In this section we numerically analyze the behaviour of the proposed schemes, with particular focus on the robustness in convection-dominant regimes.
We denote the schemes of Sections~\ref{sub:scheme3f} and~\ref{sub:scheme4f} as \threeF{} and \fourF{}, respectively,
while \noStab{} is a scheme \emph{without} the stabilization forms $J_h$, $K_h$ and $Y_h$.

Since we would like to investigate specific aspects of the proposed schemes, we will consider different analytic solutions and error indicators, which will be properly specified in each subsection.
However, we will always consider the unit cube, $\Omega=(0,\,1)^3$, as a spacial domain, and $T=1$ as final time.

In the subsequent analysis, we will use a family of four Delaunay tetrahedral meshes with decreasing mesh size 
generated by~\texttt{tetgen}~\cite{Si:2015:TAD} to discretize the domain. 
We refer to these meshes as \texttt{mesh 1}, \texttt{mesh 2} (Figure~\ref{fig:mesh}), \texttt{mesh 3} and \texttt{mesh 4}.

\begin{figure}[!htb]
\centering 
\begin{tabular}{ccc}
\includegraphics[width=0.35\textwidth]{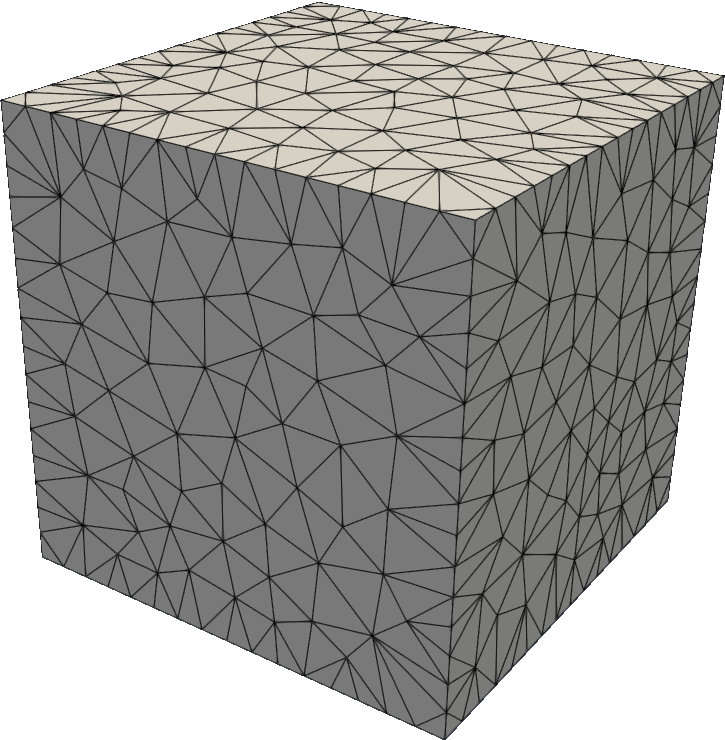}&
\qquad
\includegraphics[width=0.35\textwidth]{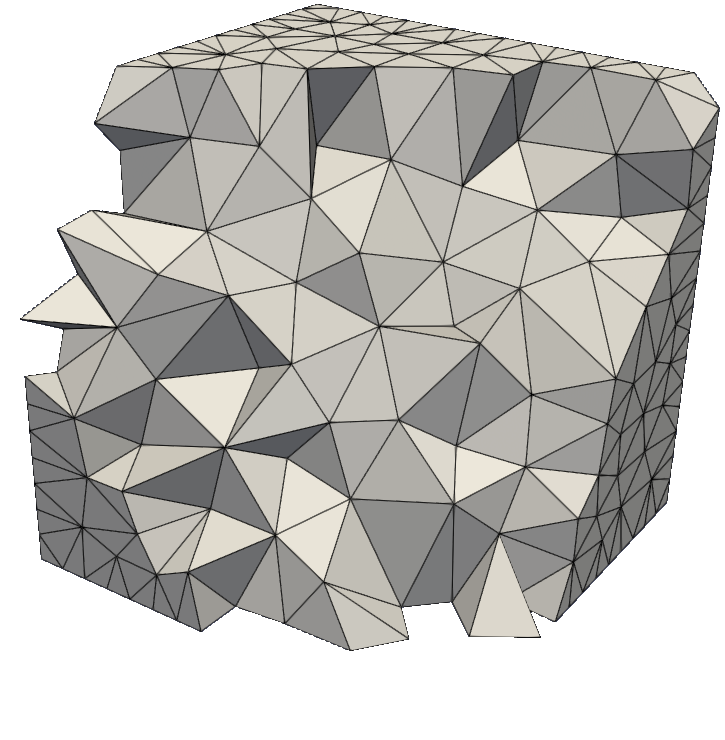}
\end{tabular}
\caption{One of the mesh used in the numerical experiments and a detail inside it.}
\label{fig:mesh}
\end{figure}

To deal with time discretization, we use an implicit Euler scheme.
In order to reduce as much as possible the error due to the integration in time,
we use a small starting time step size, $\tau$, for \texttt{mesh 1} (which will be specified case by case). 
Then, for the other meshes we halve such value, i.e., we use $0.5\tau$ for \texttt{mesh 2},
$0.25\tau$ for \texttt{mesh 3} and, finally,
$0.125\tau$ for \texttt{mesh 4}.

For each time step the nonlinear problem is solved using a fixed point strategy.
The velocity field and the magnetic field of the previous time step are used as a fist guess for these nonlinear iterations.

In the numerical tests, we always consider the lowest order case $k=1$.
Then, referring to Subsection \ref{sub:forms}, following~\cite{upwinding,DiPietro,MHD-linear}, we set $\bdma = 10$, $\bdmc = 1$, $\bdmJU = 5$ and $\bdmJD = 0.01$.
Additionally, since the forms $Y_h$ and $K_h$ have the same structure of the second term of the form $J_h$,
we set $\bdmY=\bdmK=0.01$.

\subsection{Example 1: a converge study}

In this subsection we are mainly interested in verifying that 
\threeF{} and \fourF{} exhibit the expected convergence rate under a convection-dominant regime.
To achieve this goal,
we fix $\nu=\nu_M=\nu_S$ and 
we consider the following set of values
$$
\nu = [\texttt{1.e-00}, \texttt{1.e-05}, \texttt{1.e-10}]\,.
$$
Notice that the first case, $\nu=\texttt{1.e+00}$, does not correspond to a convection dominant regime.
However, we consider also this case 
in order to ensure that the new stabilization forms do not affect the convergence of the solution in a diffusion dominant regime.

For all these values of $\nu$, we solve the MHD problem defined in Equation~\eqref{eq:primale},
where the right hand side and the boundary conditions are set in accordance with the exact solution 
\[
\begin{aligned}
\uu(x,\,y,\,z,\,t)&:= \left[\begin{array}{r}
                  \cos(0.25\pi t)\sin(\pi x)\cos(\pi y)\cos(\pi z)\\
                  \cos(0.25\pi t)\sin(\pi y)\cos(\pi z)\cos(\pi x)\\
                  -2\cos(0.25\pi t)\sin(\pi z)\cos(\pi x)\cos(\pi y)\\
                 \end{array}\right]\,,              \\
\BB(x,\,y,\,z,\,t)&:= \begin{bmatrix}
                  \cos(0.25\pi t)\sin(\pi y)\\
                  \cos(0.25\pi t)\sin(\pi z)\\
                  \cos(0.25\pi t)\sin(\pi x)
                 \end{bmatrix}\,,\\
\\                                  
p(x,\,y,\,z,\,t) &:= \cos(0.25\pi t)\left(\sin(x)+\sin(y)-2\sin(z)\right)\,.
\end{aligned}
\]

In this example, the coarser time step $\tau$ is set equal to $1/4$ which, due to the smoothness in time of the exact solution, is sufficiently small to ensure a negligible time discretization error.
For the schemes \threeF{} and \fourF{}, 
we will use the same error indicators for the velocity and pressure fields,
while the error indicator for the magnetic field will vary according to the scheme.
More specifically, we compute 
\[
\begin{aligned}
e_{\uu}&:=\|\uu(\cdot,T) - \uu_h(\cdot,T)\|_0 + \left(\int_0^T\|\uu(\cdot,t) - \uu_h(\cdot,t)\|_{\text{stab}}^2~\text{d}t\right)^{1/2}\,,\\
e_{p} &:=\left(\int_0^T\|p(\cdot,t)-p_h(\cdot,t)\|^2_0~\text{d}t\right)^{1/2}\,, \\
e_{\BB} &:= \|\BB(\cdot,T) - \BB_h(\cdot,T)\|_0 + \left(\int_0^T\|\BB(\cdot,t) - \BB_h(\cdot,t)\|_{M}^2~\text{d}t\right)^{1/2}\, , 
\end{aligned}
\]
for both \threeF{} and \fourF{}, where for the scheme \threeF{} we use 
\begin{equation*}
\begin{aligned}
& \|\ww\|_{M} := \nu_M \|\bnabla\ww\|^2 + \|\text{div}\ww\|^2 \,,
\end{aligned}
\end{equation*}
while for \fourF{} we adopt 
$$
\|\ww\|_{M} := \nu_M \|\bnabla\ww\|^2 + \mu_K\sum_{f \in \EdgesI} \text{max}\left\{1, \|\uu_h\|^2_{L^\infty(\omega_f)}\right\}h_f^2(\jump{\bnabla \ww}\,\jump{\bnabla \ww} \vert)_f\,.
$$
In Figure~\ref{fig:convExe1}, we collect the convergence lines for each scheme and error indicator.

Firstly, the stabilization forms added in the discretization of the problem do not affect the convergence rate of the discrete solution when $\nu=\texttt{1.e-00}$.
Indeed, in both schemes all the errors indicators decay as $O(h)$ as expected.

In a convective dominant regime, i.e., for $\nu=\texttt{1.e-05}$ and \texttt{1.e-10},
the schemes \threeF{} and \fourF{} exhibit exactly the same (pre-asymptotic) decay rates for the errors $e_p$ and $e_{\uu}$.
More specifically, for $e_{p}$ they both have a linear decay,
while we gain a factor of $1/2$ on the slope of the error $e_{\uu}$ for both schemes.
The trend of the error $e_{\BB}$ in a convective dominant regime is different.
Indeed, as it was predicted by the theory, if we consider the \threeF{} scheme, the decay is $O(h)$,
while the \fourF{} scheme gains a factor of $1/2$ also on the slope of $e_{\BB}$;
compare trend of the errors in the middle row in Figure~\ref{fig:convExe1}.

\begin{figure}[!htb]
\begin{center}
\begin{tabular}{cc}
\threeF &\fourF \\
\includegraphics[width=0.45\textwidth]{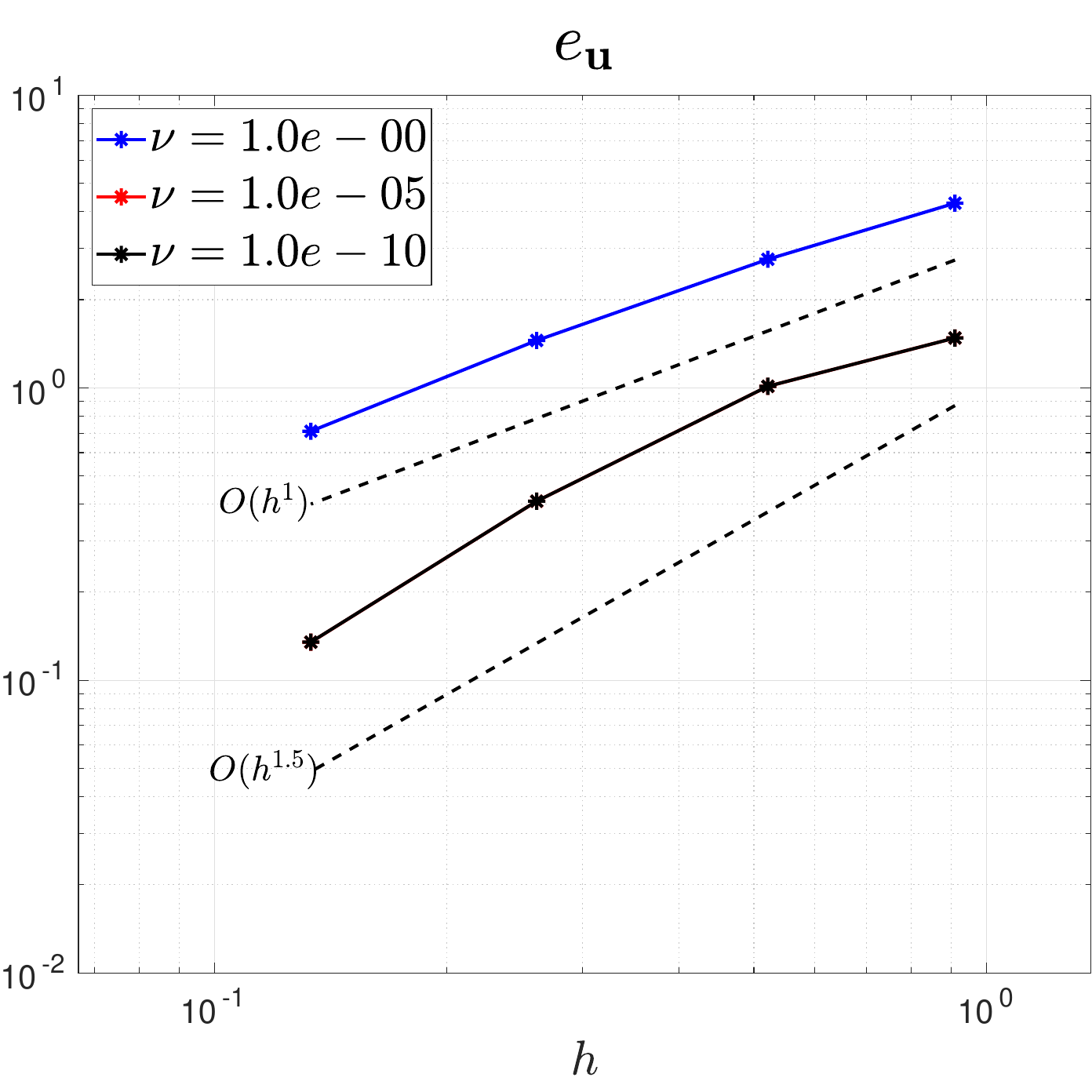} &\includegraphics[width=0.45\textwidth]{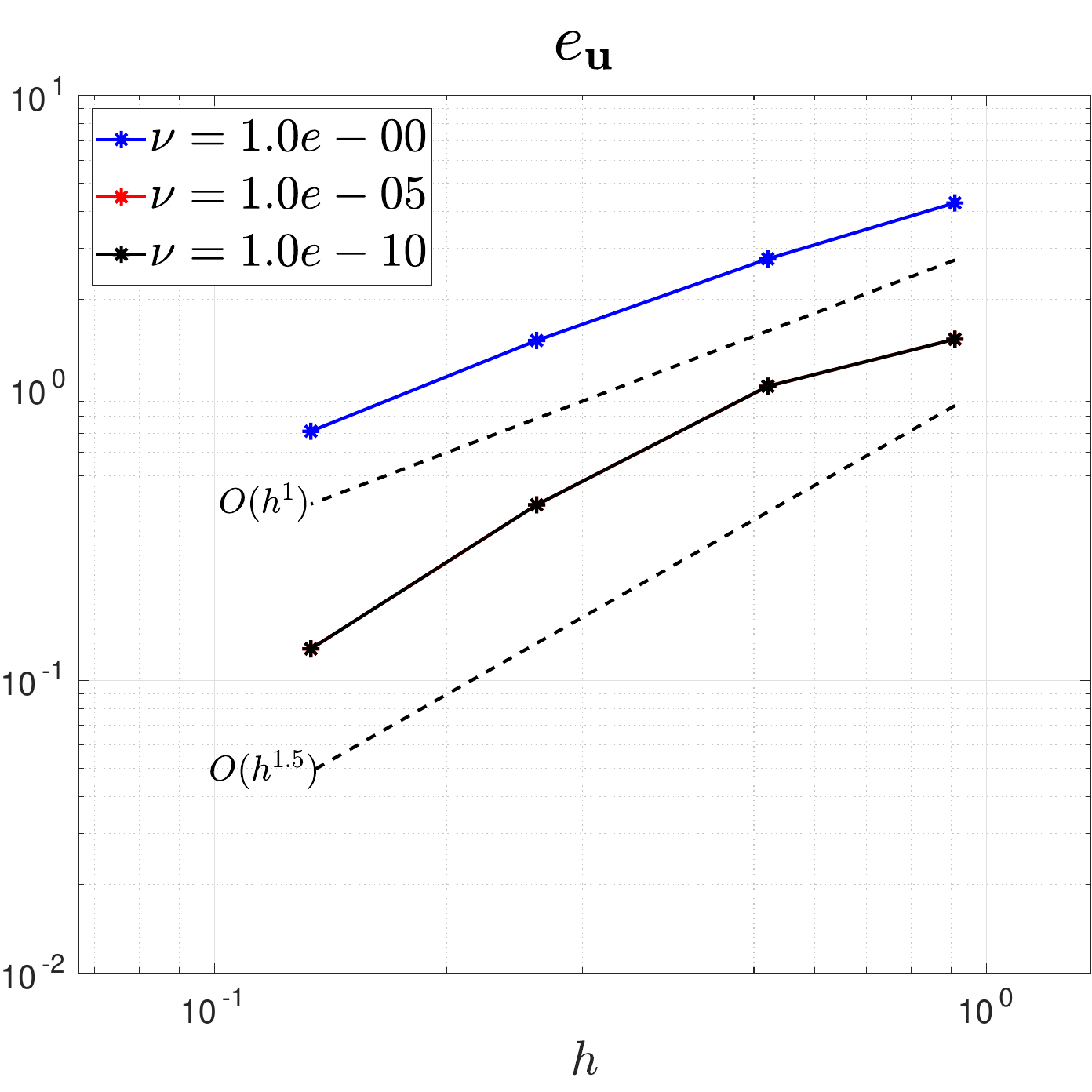}\\
\includegraphics[width=0.45\textwidth]{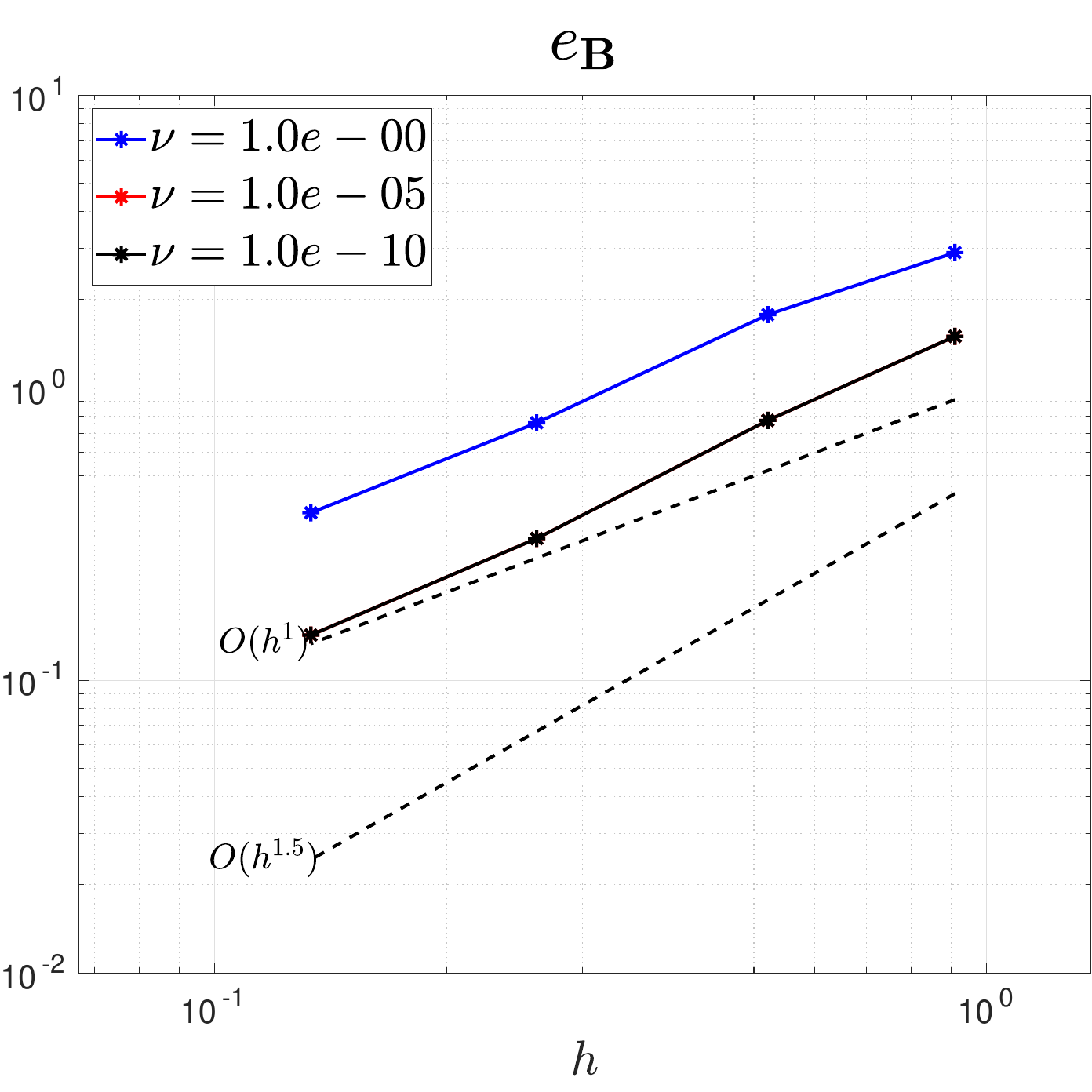} &\includegraphics[width=0.45\textwidth]{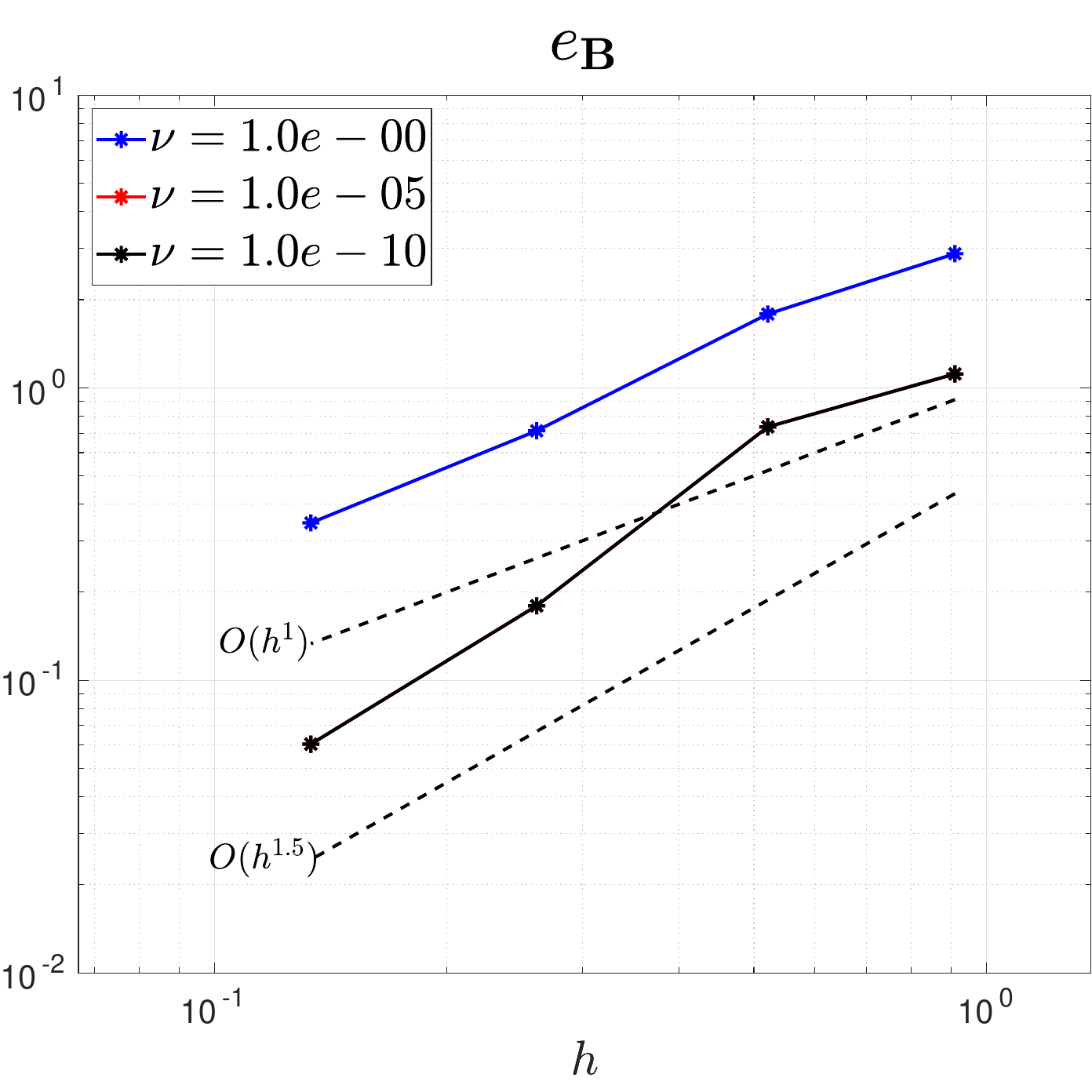}\\
\includegraphics[width=0.45\textwidth]{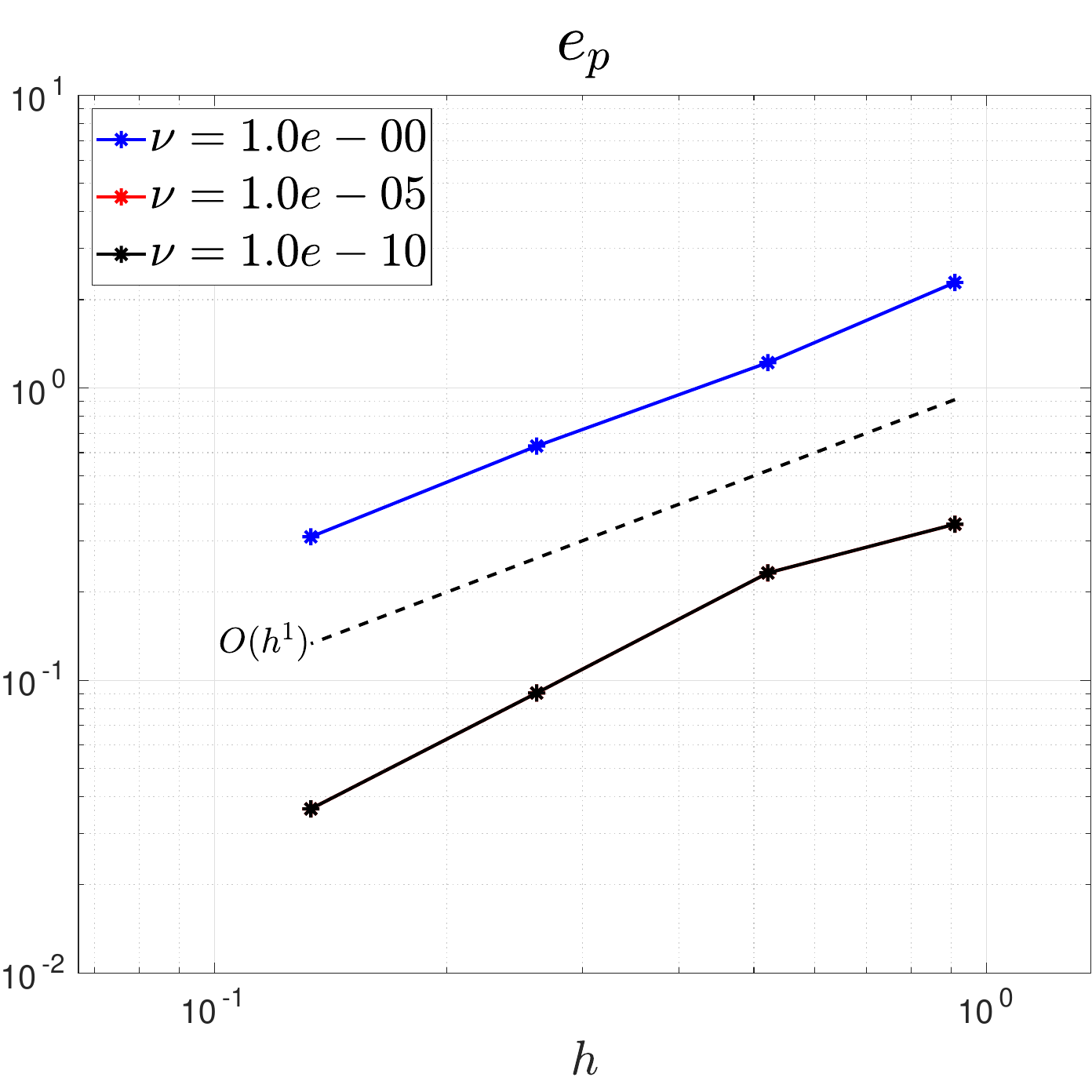} &\includegraphics[width=0.45\textwidth]{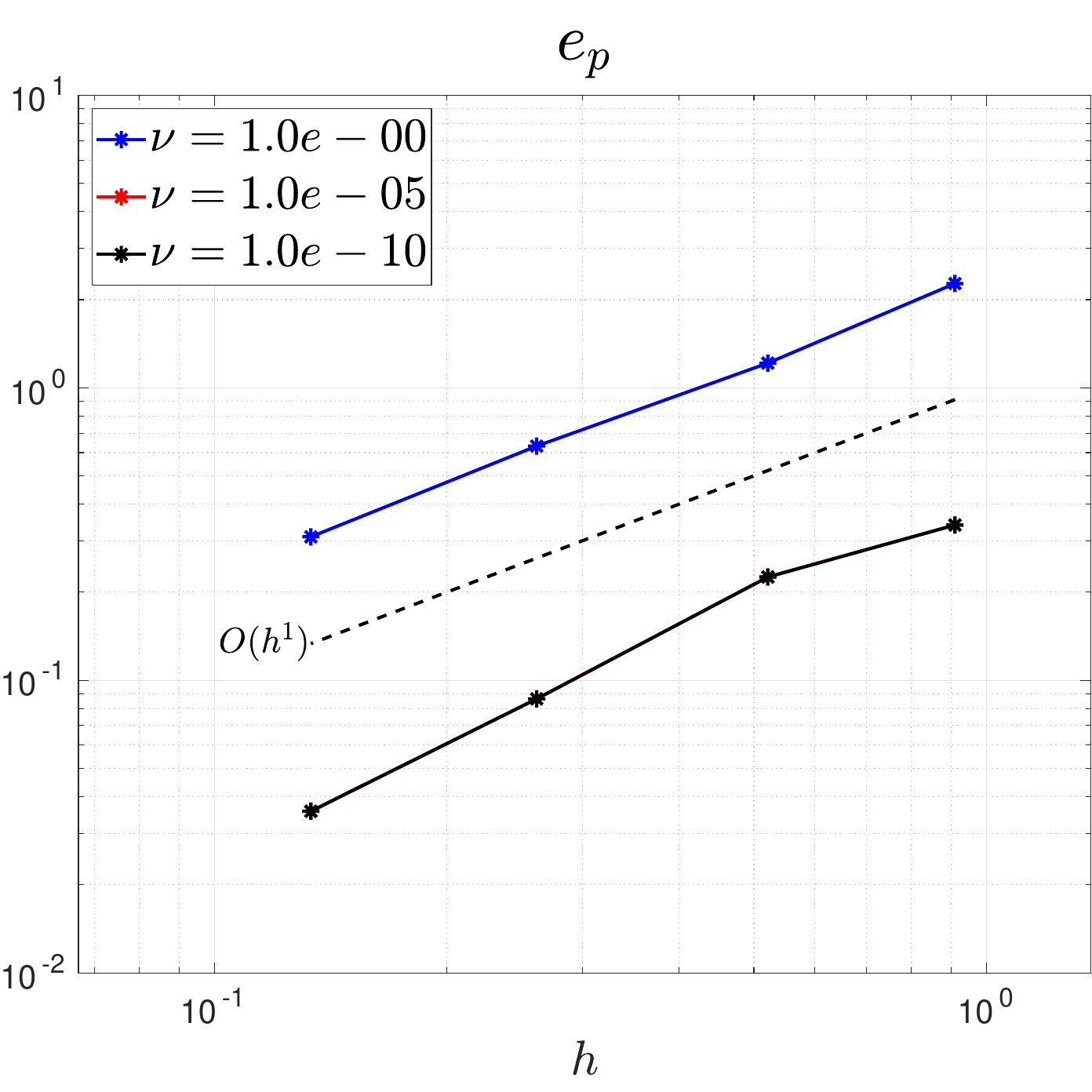}
\end{tabular}
\end{center}
\caption{Convergence histories of the schemes \eqref{eq:fem-3f} and \eqref{eq:fem-4f} with varying diffusive coefficients.}
\label{fig:convExe1}
\end{figure}

In order to have a clearer numerical evidence of the robustness of both schemes in a convection dominant regime, 
we compute all the error indicators keeping the same mesh and 
we vary the values of $\nu$ from \texttt{1.e-01} to \texttt{1.e-11}. 
In Figure~\ref{fig:varyNu}, we report these data for both \threeF{} and \fourF{} computed on \texttt{mesh 3}.
It is worth to notice that \emph{all} the error indicators remain nearly constant across different values of $\nu$.

\begin{figure}[!htb]
\begin{center}
\begin{tabular}{cc}
\includegraphics[width=0.4\textwidth]{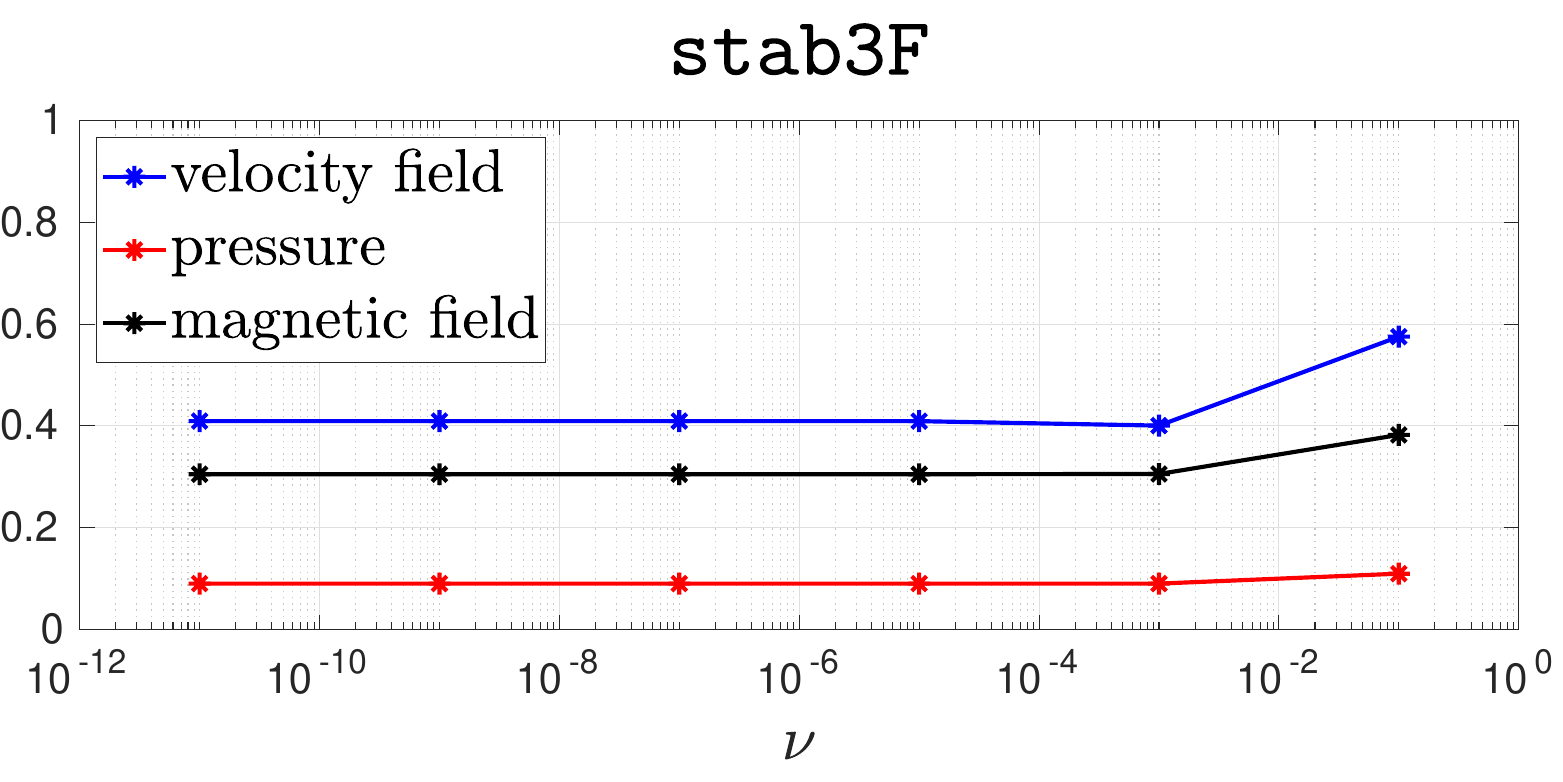} &\includegraphics[width=0.4\textwidth]{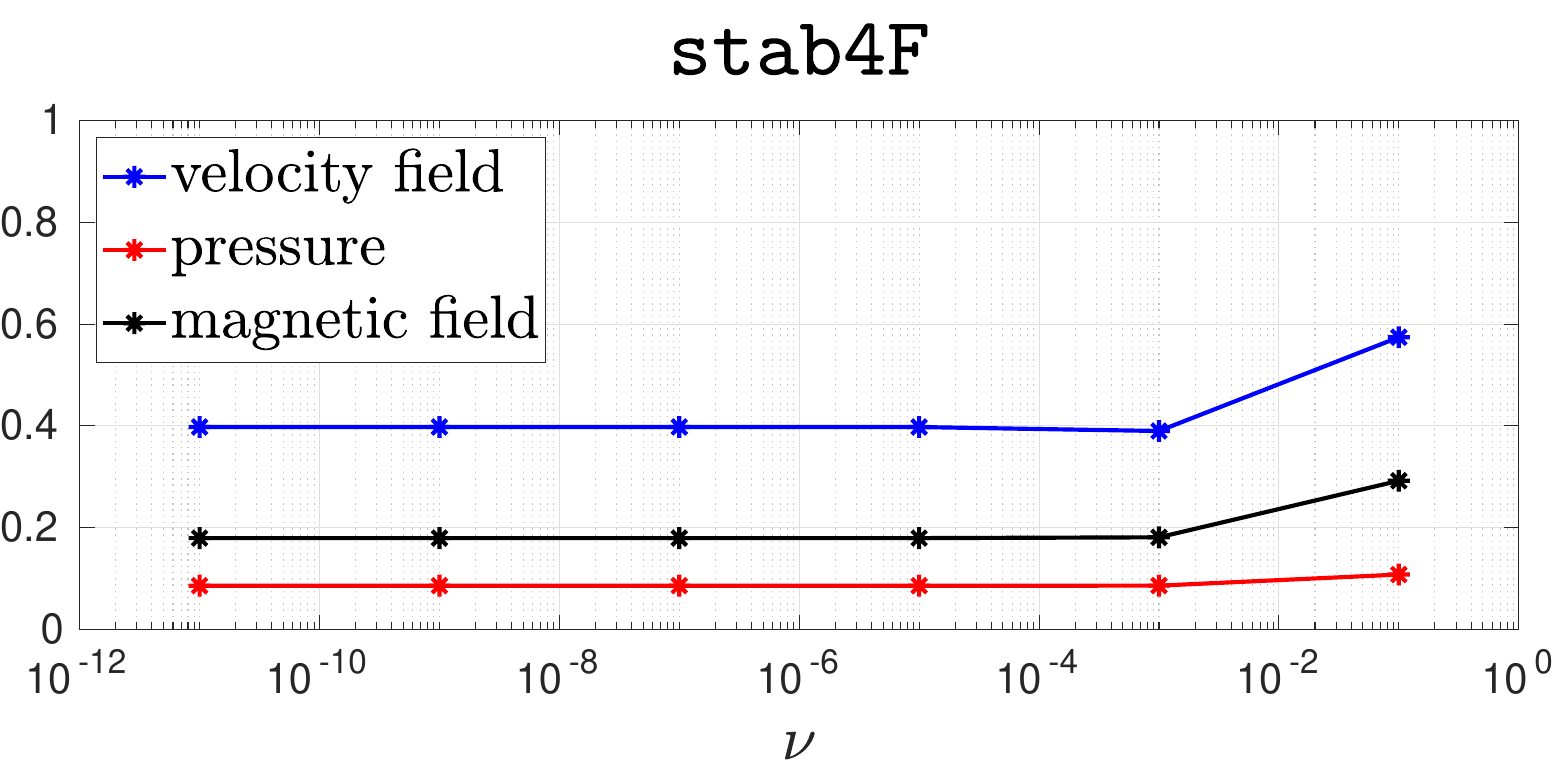}\\
\end{tabular}
\end{center}
\caption{Robustness analysis with respect to the diffusive coefficients.}
\label{fig:varyNu}
\end{figure}

\newpage

\subsection{Example 2: comparison among the three methods}

In this numerical section, we are going to compare the numerical solutions provided by the schemes \threeF{}, \fourF{}  and also a ``non stabilized'' method \noStab{}.
Such scheme \noStab{} corresponds to the \threeF{} method, but without the added stabilization form $J_h$ in \eqref{eq:J}.
Note that this \noStab{} method does benefit of the reliable upwind stabilization \eqref{eq:forme_d} for the fluid convective term, but lacks a specific convection stabilization related to the magnetic equations.

We solve the MHD equation in a convective dominant regime where $\nu_S=\texttt{1e-10}$ and $\nu_M=\texttt{1e-02}$,
with the right hand side and the boundary conditions set in accordance with the solution
\[
\begin{aligned}
\uu(x,\,y,\,z,\,t)&:= \left[\begin{array}{r}
                  t^6\sin(\pi x)\cos(\pi y)\cos(\pi z)\\
                  t^6\sin(\pi y)\cos(\pi z)\cos(\pi x)\\
                  -2t^6\sin(\pi z)\cos(\pi x)\cos(\pi y)\\
                 \end{array}\right]\,,              \\
\BB(x,\,y,\,z,\,t)&:= \begin{bmatrix}
                  \cos(0.25\pi t)\sin(\pi y)\\
                  \cos(0.25\pi t)\sin(\pi z)\\
                  \cos(0.25\pi t)\sin(\pi x)
                 \end{bmatrix}\,,\\
\\                                  
p(x,\,y,\,z,\,t) &:= \cos(0.25\pi t)\left(\sin(x)+\sin(y)-2\sin(z)\right)\,.
\end{aligned}
\]
Notice that pressure and magnetic fields are exactly the same as the previous example,
while the velocity field increases as a power of 6 in time, see the coefficient $t^6$ in front each component of $\uu$.
We made this choice of the exact solution in order to simulate the behaviour of a fluid flux 
that is initially at rest and rapidly accelerates towards the end of the simulation. From the mathematical standpoint, having a small velocity with respect to the magnetic field, at least initially, better helps to underline the usefulness of the magnetic stabilization (since we recall that also the \noStab{} scheme enjoys a jump stabilization endowed by the upwind discretization of the fluid convection).
To reduce as much as possible the error due to the time integration,
we set $\tau=1/64$.

In this numerical example, we will compute the following error indicators:
\[
\begin{aligned}
e_{\uu}&:=\|\uu(\cdot,T) - \uu_h(\cdot,T)\|_0 + \left(\int_0^T\ns \Vert \uu(\cdot,T)-\uu_h(\cdot,T)\Vert_{1,h}^2 + |\uu(\cdot,T)-\uu_h(\cdot,T)|_{\text{upw},\uu}^2 ~\text{d}t\right)^{1/2}\,,\\
e_{p} &:=\left(\int_0^T\|p(\cdot,t)-p_h(\cdot,t)\|^2_0~\text{d}t\right)^{1/2}\,,\\
e_{\BB}&:=\|\BB(\cdot,T) - \BB_h(\cdot,T)\|_0 + \left(\int_0^T\nu_M \|\nabla(\BB(\cdot,t) - \BB_h(\cdot,t))\|^2~\text{d}t\right)^{1/2}\,. 
\end{aligned}
\]
We have made this choice of error indicators because 
we aim to compute \emph{only} the norms and seminorms that are common to all the schemes: \noStab{}, \threeF{} and \fourF{}.

In Figure~\ref{fig:convExe2}, we present the convergence lines associated with the errors $e_{\uu}$ and $e_{\BB}$ for each scheme. As commented in \cite{Gerbeau} for a much simpler case, stabilizing the magnetic equations does lead to a more accurate velocity.
There is not much difference among the proposed schemes regarding the error $e_{p}$ so 
we do not show these converge lines.

\begin{figure}[!htb]
\begin{center}
\begin{tabular}{cc}
\includegraphics[width=0.45\textwidth]{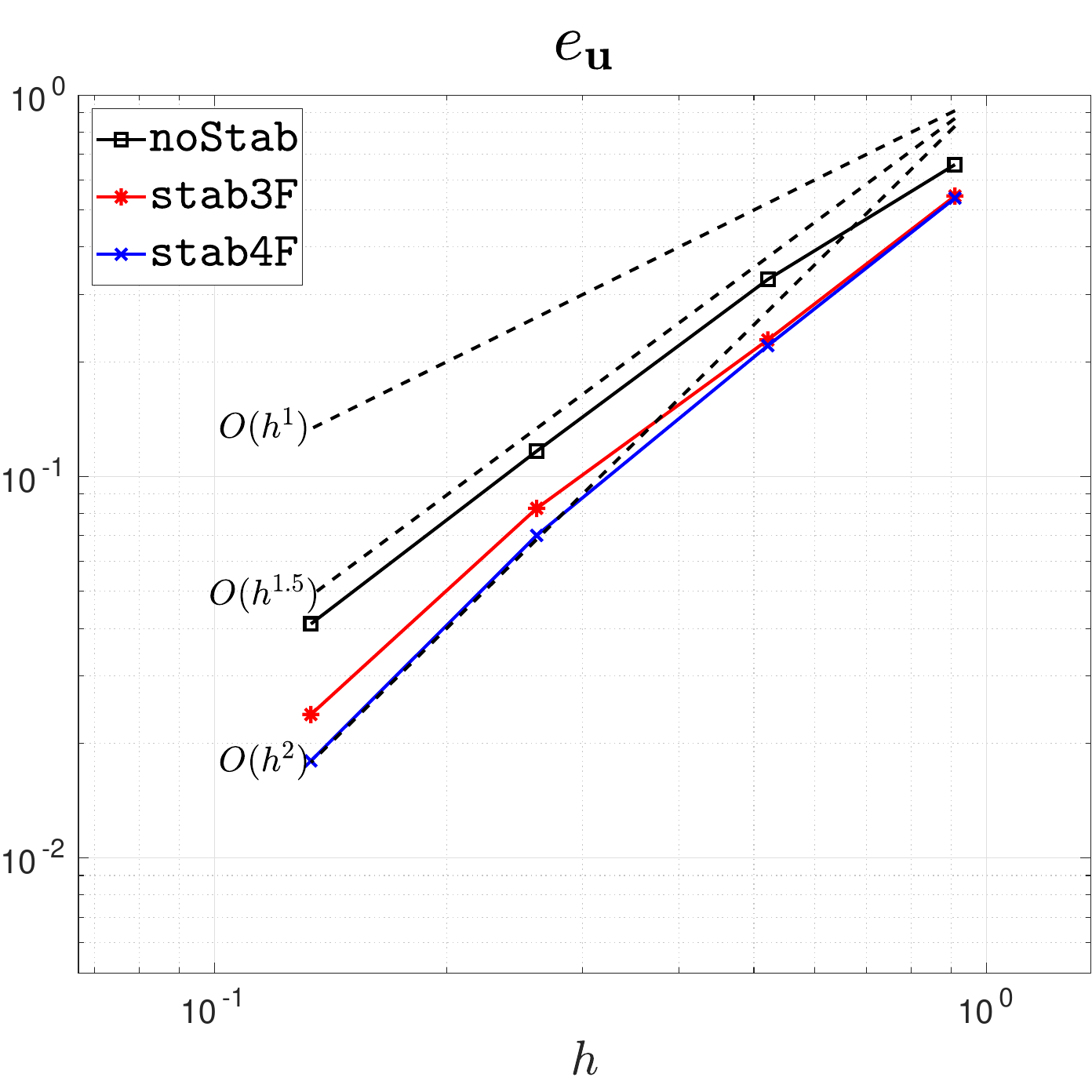} &\includegraphics[width=0.45\textwidth]{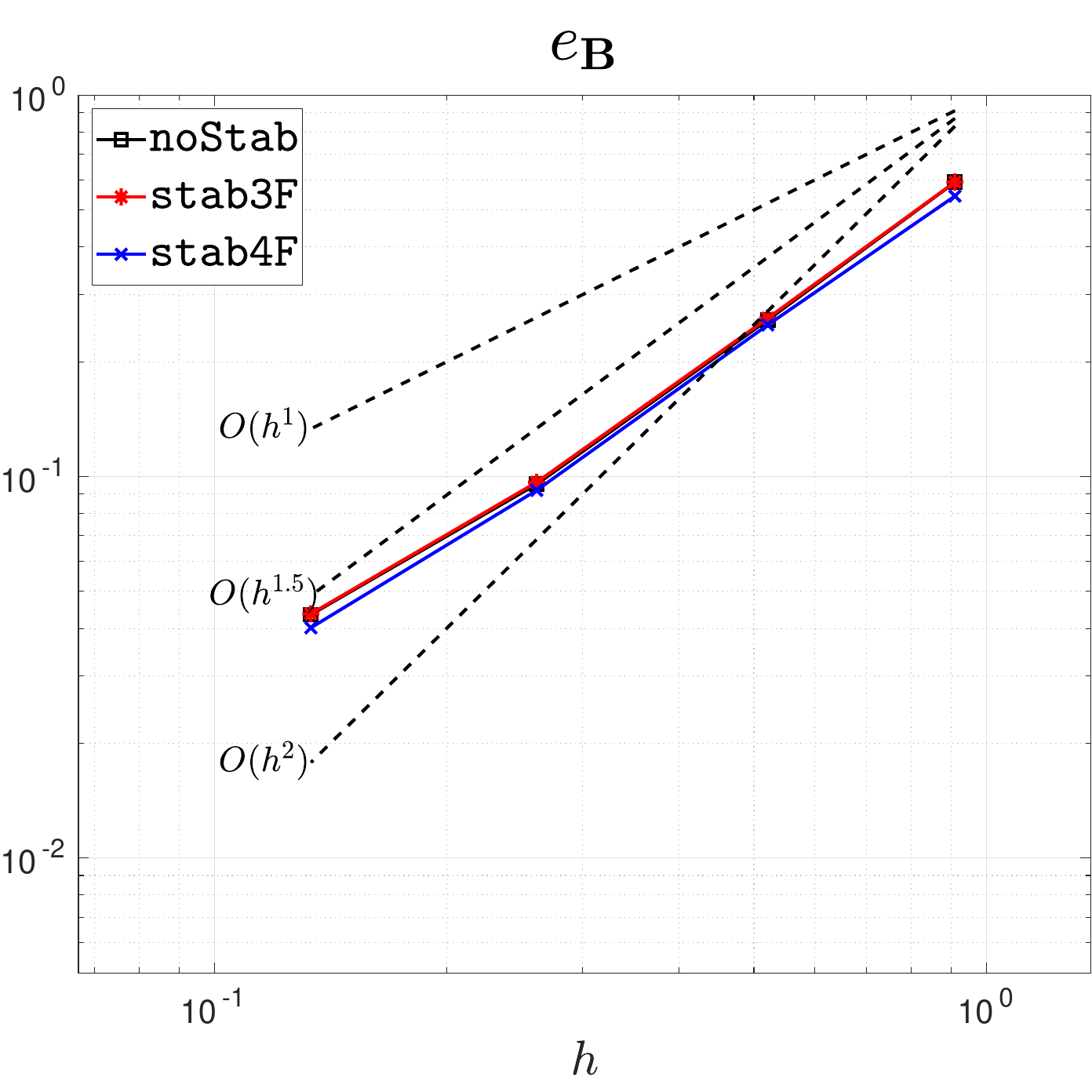}\\
\end{tabular}
\end{center}
\caption{
Convergence histories of the errors $e_{\uu}$ and $e_{\BB}$ with $\nu_S=\texttt{1e-10}$ and $\nu_M=\texttt{1e-02}$ varying the numerical scheme.}
\label{fig:convExe2}
\end{figure}
%


\bigskip\noindent
{\bf Acknowledgments}
\smallskip
LBdV and FD have been partially funded by the European Union (ERC, NEMESIS, project number 101115663). Views and opinions expressed are however those of the author(s) only and do not necessarily reflect those of the EU or the ERC Executive Agency.
All the three authors are members of the Gruppo Nazionale Calcolo Scientifico-Istituto Nazionale di Alta Matematica (GNCS-INdAM).

\newpage

\addcontentsline{toc}{section}{\refname}
\bibliographystyle{plain}
\bibliography{references}
\end{document}